\documentclass[aap]{imsart}

\RequirePackage{amsthm,amsmath,amsfonts,amssymb}
\RequirePackage[numbers]{natbib}
\RequirePackage[colorlinks,citecolor=blue,urlcolor=blue]{hyperref}
\RequirePackage{graphicx}

\usepackage{stackengine}
\usepackage{latexsym}
\usepackage{amsmath,amssymb,dsfont,amsfonts}
\usepackage{mathrsfs}

\usepackage{verbatim}
\usepackage{graphicx}
\usepackage{graphics}
\usepackage{bbm}
\usepackage{booktabs}
\usepackage[shortlabels]{enumitem}
\usepackage{xcolor}
\usepackage[ruled, vlined, linesnumbered]{algorithm2e}
\usepackage{subfigure}

\DeclareMathOperator*{\esssup}{ess\,sup} % copiar esta linha
\DeclareMathOperator*{\argmin}{arg\,min} % copiar esta linha

\startlocaldefs
\theoremstyle{plain}

\newtheorem{theorem}{Theorem}[section]
\newtheorem{remark}{Remark}[section]
\newtheorem{proposition}{Proposition}[section]
\newtheorem{lemma}[theorem]{Lemma}
\theoremstyle{remark}
\newtheorem{definition}[theorem]{Definition}
\newtheorem*{example}{Example}

\endlocaldefs

\begin{document}

\begin{frontmatter}
\title{Solving non-Markovian Stochastic Control Problems driven by Wiener Functionals}
%\title{A sample article title with some additional note\thanksref{t1}}
\runtitle{ Non-Markovian Stochastic Control}
%\thankstext{T1}{A sample additional note to the title.}

\begin{aug}
%%%%%%%%%%%%%%%%%%%%%%%%%%%%%%%%%%%%%%%%%%%%%%%
%% Only one address is permitted per author. %%
%% Only division, organization and e-mail is %%
%% included in the address.                  %%
%% Additional information can be included in %%
%% the Acknowledgments section if necessary. %%
%% ORCID can be inserted by command:         %%
%% \orcid{0000-0000-0000-0000}               %%
%%%%%%%%%%%%%%%%%%%%%%%%%%%%%%%%%%%%%%%%%%%%%%%
\author[A]{\fnms{Dorival}~\snm{Le\~ao}\ead[label=e1]{leao@estatcamp.com.br}},
\author[B]{\fnms{Alberto}~\snm{Ohashi}\ead[label=e2]{amfohashi@gmail.com}}
\and
\author[A]{\fnms{Francys A}~\snm{de Souza}\ead[label=e3]{francys@estatcamp.com.br}}
%%%%%%%%%%%%%%%%%%%%%%%%%%%%%%%%%%%%%%%%%%%%%%
%% Addresses                                %%
%%%%%%%%%%%%%%%%%%%%%%%%%%%%%%%%%%%%%%%%%%%%%%
\address[A]{Estatcamp\printead[presep={,\ }]{e1,e3}}

\address[B]{Departmento de Matem\'atica,
Universidade de Bras\'ilia\printead[presep={,\ }]{e2}}
\end{aug}

\begin{abstract}

In this article, we present a general methodology for stochastic control problems driven by the Brownian motion filtration including non-Markovian and non-semimartingale state processes controlled by mutually singular measures. The main result of this paper
is the development of a numerical scheme for computing near-optimal controls associated with controlled Wiener functionals via a finite-dimensional approximation procedure. The approach does not require functional differentiability assumptions on the value process and ellipticity conditions on the diffusion components. The general convergence of the method is established under rather weak conditions for distinct types of non-Markovian and non-semimartingale states. Explicit rates of convergence are provided in case the control acts only on the drift component of the controlled system. Near-closed/open-loop optimal controls are fully characterized by a dynamic programming algorithm and they are classified according to the strength of the possibly underlying non-Markovian memory. The theory is applied to stochastic control problems based on path-dependent SDEs and rough stochastic volatility models, where both drift and possibly degenerated diffusion components are controlled. Optimal control of drifts for nonlinear path-dependent SDEs driven by fractional Brownian motion with exponent $H\in (0,\frac{1}{2})$ is also discussed. Finally, we present a simple numerical example to illustrate the method.

%In this article, we present a general methodology for stochastic control problems driven by the Brownian motion filtration including non-Markovian and non-semimartingale state processes controlled by mutually singular measures. The main result of this paper
%is the development of a numerical scheme for computing near-optimal controls associated with controlled Wiener functionals via a finite-dimensional approximation procedure. The approach does not require functional differentiability assumptions on the value process and ellipticity conditions on the diffusion components. Explicit rates of convergence are provided under rather weak conditions for distinct types of non-Markovian and non-semimartingale states. Near-closed/open-loop optimal controls are fully characterized by a dynamic programming algorithm and they are classified according to the strength of the possibly underlying non-Markovian memory. The theory is applied to stochastic control problems based on path-dependent SDEs and rough stochastic volatility models, where both drift and possibly degenerated diffusion components are controlled. Optimal control of drifts for nonlinear path-dependent SDEs driven by fractional Brownian motion with exponent $H\in (0,\frac{1}{2})$ is also discussed. Finally, we present a simple numerical example to illustrate the method.

\end{abstract}



\end{frontmatter}
%%%%%%%%%%%%%%%%%%%%%%%%%%%%%%%%%%%%%%%%%%%%%%
%% Please use \tableofcontents for articles %%
%% with 50 pages and more                   %%
%%%%%%%%%%%%%%%%%%%%%%%%%%%%%%%%%%%%%%%%%%%%%%
\tableofcontents

\section{Introduction}
Let $\mathbf{C}_{n,T}$ be the set of continuous functions from $[0,T]$ to $\mathbb{R}^n$, let $\xi:\mathbf{C}_{n,T}\rightarrow\mathbb{R}$ be a Borel functional, let $\mathbb{F} = (\mathcal{F}_t)_{t\ge 0}$ be a fixed filtration and let $U^T_t; 0\le t \le T$ be a family of admissible $\mathbb{F}$-adapted controls defined over $(t,T]$. The goal of this paper is to develop a numerical scheme to solve a generic stochastic optimal control problem of the form

\begin{equation}\label{INTROpr1}
\sup_{\phi\in U^T_0}\mathbb{E}\big[\xi\big(X^\phi\big)\big],
\end{equation}
where $\{X^\phi; \phi\in U^T_0\}$ is a given family of $\mathbb{F}$-adapted controlled continuous processes. A common approach for the study of (\ref{INTROpr1}) (see e.g \cite{Davis_79, elkaroui}) is to consider for each control $u\in U_0^T$, the value process given by

\begin{equation}\label{valueINTR}
V(t,u) = \esssup_{\phi;\phi=u~\text{on}~[0,t]}\mathbb{E}\big[\xi\big(X^\phi\big)|\mathcal{F}_t\big]; 0\le t \le T.
\end{equation}
Two fundamental questions in stochastic control theory rely on sound characterizations of value processes and the development of concrete methods to produce either exact optimal controls $u^*\in U^T_0$ (when exists)

$$
\mathbb{E}\big[\xi(X^{u^*})\big] = \sup_{\phi\in U^T_0}\mathbb{E}\big[\xi\big(X^\phi\big)\big],
$$
or near-optimal controls (see e.g \cite{zhou}) which realize
\begin{equation}\label{INTROpr3}
\mathbb{E}\big[\xi(X^{u^*})\big] \ge \sup_{\phi\in U^T_0}\mathbb{E}\big[\xi\big(X^\phi\big)\big]-\epsilon,
\end{equation}
for an arbitrary error bound $\epsilon>0$. Exact optimal controls may fail to exist due to e.g lack of convexity. In this context, the standard approach is to consider near-optimal controls which exist under minimal hypotheses and are sufficient in most applications.

%The present article aims to provide a systematic method to compute and characterize $u^*$ realizing (\ref{INTROpr3}) for a given stochastic control problem (\ref{valueINTR}) driven by a generic controlled process .

Two major tools for studying stochastic controlled systems are Pontryagin's maximum principle and Bellman's dynamic programming. While these two methods are known to be very efficient for establishing some key properties (e.g existence of optimal controls, smoothness of the value functional, sufficiency of subclasses of controls, etc), the problem of solving explicitly or numerically a given stochastic control problem remains a critical issue in the field of control theory. Indeed, except for a very few specific cases, the determination of an optimal control (either exact or near) is a highly nontrivial problem to tackle. %\textcolor{red}{In this article, we provide a discrete-time dynamic programming algorithm to produce $u^*$ realizing (\ref{INTROpr3}) for possibly non-Markovian and non-semimartingale controlled states $\{X^\phi; \phi\in U^T_0\}$.}

%The present article presents a systematic method to compute and characterize for a given stochastic control problem (\ref{valueINTR}) driven by a generic controlled process.

In the Markovian case, a classical approach in solving stochastic control problems is given by the dynamic programming principle based on Hamilton-Jacobi-Bellman (HJB) equations. One popular approach is to employ verification arguments to check if a given solution of the HJB equation coincides with the value function at hand, and obtain as a byproduct the optimal control. Discretization methods also play an important role towards the resolution of the control problem. In this direction, several techniques based on Markov chain discretization schemes \cite{kushner2}, Krylov's  regularization and shaking coefficient techniques (see e.g \cite{krylov1,krylov2}) and Barles-Souganidis-type monotone schemes \cite{barles} have been successfully implemented. We also refer the more recent probabilistic techniques on fully non-linear PDEs given by \cite{fahim} and the randomization approach of \cite{pham, pham1,pham2}.

\subsection{Review of literature}
Beyond the Markovian context, the value process (\ref{valueINTR}) cannot be reduced to a deterministic PDE and the control problem (\ref{INTROpr1}) is much more delicate. The work \cite{nutz2} employs techniques from quasi-sure analysis to characterize one version of the value process as the solution of a second order backward stochastic differential equation (2BSDE) (see \cite{soner}) under a non-degeneracy condition on the diffusion component of a controlled non-Markovian stochastic differential equation driven by Brownian motion (henceforth abbreviated by CNM-SDE-BM). The work \cite{nutz1} derives a dynamic programming principle in the context of model uncertainty and nonlinear expectations. Inspired by \cite{pham}, under the weak formulation of the control problem, \cite{fuhrman} shows a value process can be reformulated under a family of dominated measures on an enlarged filtered probability space where the CNM-SDE-BM might be degenerated. It is worth to mention that under a nondegeneracy condition on diffusion components of CNM-SDEs-BM, (\ref{valueINTR}) can also be viewed as a fully nonlinear path-dependent PDE in the sense of \cite{touzi2} via its relation with 2BSDEs (see section 4.3 in \cite{touzi2}). In this direction, \cite{possamai} derived a dynamic programming principle for a stochastic control problem w.r.t a class of nonlinear kernels. They obtained a well-posedness result for general 2BSDEs and established a link with path-dependent PDEs in possibly degenerated cases. We also drive attention to \cite{qiu} who characterizes (\ref{valueINTR}) driven by a CNM-SDE-BM as a solution of a suitable HJB-type equation and uniqueness is established under a non-degeneracy condition on the diffusion components. Under strong a priori functional differentiability conditions (in the sense of \cite{dupire,cont1}) imposed on (\ref{valueINTR}), one can apply functional It\^o's formula to arrive at verification-type theorems. In this direction, we refer to e.g. \cite{cont2} and \cite{saporito}.

Discrete-type schemes for the optimal value (\ref{INTROpr1}) driven by CNM-SDEs-BM were studied by \cite{dolinsky}, \cite{zhang2}, \cite{ren} and \cite{tan}. In \cite{zhang2,ren}, the authors provide monotone schemes in the spirit of Barles-Souganidis for fully nonlinear path-dependent PDEs in the sense of \cite{touzi2} and hence one may apply their results for the study of (\ref{INTROpr1}). Under elipticity conditions, by employing weak convergence methods in the spirit of Kushner and Depuis, \cite{tan} provides a discretization method for the optimal value (\ref{INTROpr1}). Convergence rates are available only under strong regularity conditions on the value process (\ref{valueINTR}) (see \cite{zhang2}) or in the state independent case \cite{dolinsky,tan}.

Recently, a dynamic programming principle for stochastic Volterra equations was proved by \cite{viens}, by extending the framework of \cite{dupire} to a more general setting involving fractional Brownian motions and related processes. See also \cite{wang}. Infinite-dimensional Riccati-type equations associated with Linear-Quadratic control problems driven by linear stochastic Volterra equations were derived by \cite{jaber1,jaber2}. In this direction, we also drive attention to the early works by \cite{Hu1} and \cite{duncan}.

%We emphasize that in all the works mentioned
%above, the general question of

\subsection{Main setup and contributions}
The main contribution of this paper is the development of a numerical scheme for computing near optimal controls of stochastic systems adapted to the Brownian motion filtration and parameterized by possibly mutually singular measures. The methodology is based on a weak version of functional It\^o calculus developed by \cite{LEAO_OHASHI2017.1} and inspired by \cite{LEAO_OHASHI2013}. For a given sequence $\epsilon_k\downarrow 0$ and a given $d$-dimensional Brownian motion, we consider a discrete-type skeleton $\mathscr{D} = \{\mathcal{T},A^{k}; k\ge 1\}$ (see Definition \ref{discreteskeleton}), where $A^k$ is a sequence of stepwise-constant martingales with jumps $|\Delta A^k(T^k_n)|\le \epsilon_k$ and jumping times $\mathcal{T} = \{T^k_n; n\ge 1\}$.

The stochastic control problem (\ref{INTROpr3}) is replaced by

$$
\mathbb{E}\big[\xi(X^{k,u^{k,*}})\big] \ge \sup_{u^k\in U^{k,e(k,T)}_0}\mathbb{E}\big[\xi\big(X^{k,u^k}\big)\big]-\epsilon,
$$
where the controlled state $\mathcal{X} = \{X^{k,u^k}; u^k\in U_0^{k,e(k,T)}\}$ is an $\mathscr{D}$-adapted stepwise constant process,
where $e(k,T)$ is a suitable number of steps to recover (\ref{INTROpr1}) over the entire period $[0,T]$ as the discretization level $k$ goes to infinity and the class of controls $U^{k,e(k,T)}_0$ is constituted by $\mathscr{D}$-adapted stepwise constant controls $u^k$ with jumps $\{u^k_j; j=0,\ldots, e(k,T)-1\}$.

By using measurable selection arguments, we \textit{aggregate} a $\mathscr{D}$-adapted version $\big((V^k)_{k\ge 0},\mathscr{D}\big)$ of (\ref{valueINTR}) into a finite sequence of upper semianalytic value functions $\mathbb{V}^k_m:\mathbb{H}^{k,m}\rightarrow\mathbb{R}; m=0,\ldots, e(k,T)-1$. Here, $\mathbb{H}^{k,m}$ is the $m$-fold Cartesian product of $\mathbb{W}_k\times \mathbb{R}^n$, where $\mathbb{W}_k$ is suitable finite-dimensional space which accommodates the dynamics of the structure $\mathscr{D}$ and $\mathbb{R}^n$ is the state-space. Starting from a given controlled state $\mathbb{V}^k_{e(k,T)}$, the value functionals satisfy

\begin{eqnarray}\label{DPPInt}
\mathbf{U}^k_j(\mathbf{o}^k_j,\theta) &:=& \int_{\mathbb{W}_k}\mathbb{V}^{k}_{j+1}\Big(\mathbf{o}^{k}_{j}, \mathfrak{X}^k_{j+1}(\theta, \mathbf{o}^k_{j}, w^k)\Big)\nu^k(dw^k)\\
\nonumber\mathbb{V}^k_j(\mathbf{o}^k_j) &:=& \sup_{\theta\in \mathbb{A}}\mathbf{U}^k_j(\mathbf{o}^k_j,\theta),~j=e(k,T)-1, \ldots, 0.
\end{eqnarray}
In (\ref{DPPInt}), $\mathfrak{X}^k_{j+1}$ is a transition kernel of $\mathcal{X}$ which represents the jump of the controlled state $\mathcal{X}$ from $j$ to $j+1$. The variable $\theta$ is a constant control in a compact action space $\mathbb{A}$, $\mathbf{o}^k_j$ is a path which carries all the possible observed controlled states and, eventually, the noise history until step $j$. The variable $w^k$ is the noise variable which has to be integrated w.r.t the law $\nu^k$ of $(\Delta T^k_1,\Delta A^k(T^k_1))$.

This article shows that the dynamic programming equation (\ref{DPPInt}) constitutes the building block to solve stochastic control problems much beyond the classical Markovian states, traditionally treated by Markov chain approximations and PDE methods.
The main results of this paper (Theorem \ref{absrate1} and \ref{THdriftrate}) show that the recursive functional equation (\ref{DPPInt}) allows us to compute near optimal controls for (\ref{INTROpr3}) as the discretization level $\epsilon_k\downarrow 0$. None differentiability condition on the value process (\ref{valueINTR}) is required and the scheme is implementable using e.g regression Monte-Carlo techniques.

%The generality of our scheme can be summarized by distinct choices one can make for the functionals $(\Sigma,\ell_j; 1\le j\le e(k,T)$ in (\ref{GF}).

As a test of the relevance of the theory, we show that our methodology can be applied to controlled Wiener functionals of the form

$$
X^u = F\big(B,(\alpha,\sigma)(X^u,u)  \big),
$$
where the non-anticipative functionals $(\alpha,\sigma)$ admit Lipschitz regularity w.r.t. control variables which in turn implies $L^2$-Lipschitz regularity of $u \mapsto X^u$ w.r.t control processes (Assumption (B1) in section \ref{basics}). In the sequel, we denote $Y_t = \{Y(s); 0\le s\le t\}$ as the path of a process $Y$ until a time $t$. We investigate three typical examples illustrating distinct types of path-dependency which fit into the theory developed in this paper:

\begin{equation}\label{caseA}
dX^u(t) = \alpha(t,X^u_t,u(t))dt + \sigma(t,X^u,u(t))dB(t),
\end{equation}

\begin{equation}\label{caseB}
dX^u(t) = \alpha(t,X^u_t),u(t))dt + \sigma dB_H(t),
\end{equation}
and

\begin{equation}\label{caseC}
\left\{
\begin{array}{lc}
dX^u(t) = X^u(t) \mu(u(t))dt + X^u(t)  \vartheta(Z(t),u(t)) dB(t) \\
dZ(t) = \nu dB_H(t)- \beta(Z(t)-m)dt,
\end{array}
\right.
\end{equation}
where $B_H$ is a fractional Brownian motion (henceforth abbreviated by FBM) with exponent $0 < H < \frac{1}{2}$. In case of (\ref{caseA}), the driving state is a CNM-SDE-BM and the lack of Markov property is due to non-anticipative functionals $\alpha$ and $\sigma$ which may depend on the whole path of $X^u$. In this case, the controlled state $X^u$ satisfies a pseudo-Markov property in the sense of \cite{claisse}. Case (\ref{caseB}) illustrates a fully non-Markovian case: the controlled state $X^u$ is driven by a path-dependent nonlinear drift and by a very singular transformation of the Brownian motion into a non-Markovian and non-semimartingale noise. In particular, there is no probability measure on the path space such that the controlled state in (\ref{caseB}) is a semimartingale. Case (\ref{caseC}) has been devoted to a considerable attention in recent years in the context of rough stochastic volatility models when $0 < H < \frac{1}{2}$. In this direction, see e.g \cite{bayer} and other references therein.

The controlled dynamics (\ref{caseA}), (\ref{caseB}) and (\ref{caseC}) can be analyzed via Euler-type schemes of the form

%$X^{k,u^k}_{T^k_{n-1}}$ carries the whole information of the controlled state $\{\Delta X^{k,u^k}(T^k_j); 1\le j\le n-1\}$ until step $n-1$

$$
\Delta X^{k,u^k}(T^k_n)=\alpha \Big(T^k_{n-1},X^{k,u^k}_{T^k_{n-1}},u^k_{n-1}\Big)\Delta {T^k_{n}}
+ \Sigma \Big(T^{k}_{n-1},X^{k,u^k}_{T^{k}_{n-1}},u^k_{n-1},\ell_n (\mathcal{A}^k_n)\Big),
$$
for $1\le n\le e(k,T)$, where $\mathcal{A}^k_n = \big(\Delta T^k_1, \Delta A^k(T^k_1),\ldots, \Delta T^k_n, \Delta A^k(T^k_n)\big)$ and $\ell_n$ is a functional of the noise path $\mathcal{A}^k_n$ which distinguishes distinct types of non-Markovian states. Proposition \ref{openclosed} shows that closed / open loop optimal controls associated with (\ref{DPPInt}) are closely related to the strength of the possibly underlying non-Markovian memory. For instance, (\ref{caseA}) and (\ref{caseB}) can be analyzed by setting

$$
\Sigma \Big(T^{k}_{n-1},X^{k,u^k}_{T^{k}_{n-1}},u^k_{n-1},\ell_n (\mathcal{A}^k_n)\Big) = \sigma \Big( T^k_{n-1}, X^{k,u^k}_{T^{k}_{n-1}}, u^k_{n-1} \Big)\ell_n (\mathcal{A}^k_n).
$$
In the case (\ref{caseA}), we have $\ell_n(\mathcal{A}^k_n) = \Delta A^k(T^k_n)$ and the variable $\mathbf{o}^k_j$ in the transition kernel $\mathfrak{X}^k$ in (\ref{DPPInt}) is just a functional of the state. In this case, the optimal control $u^{k,*}$ derived from (\ref{DPPInt}) turns out to be a feedback control which is adapted w.r.t the controlled-state $X^{k,u^{k,*}}$. In the case (\ref{caseB}), the additional extrinsic memory generated by the FBM are encoded by $\ell_n(\mathcal{A}^k_n) = \Delta B^k_H(T^k_n)$, where $B^k_H$ is an $\mathscr{D}$-adapted discrete version of the FBM. The increments of FBM inserts a nontrivial memory in such way that the variable $\mathbf{o}^k_j$ in (\ref{DPPInt}) carries the whole noise path until time $j$ and the optimal control $u^{k,*}$ derived from (\ref{DPPInt}) turns out to be open-loop. The rough volatility case (\ref{caseC}) exhibits a more sophisticated type of memory

$$
\Sigma \Big(T^{k}_{n-1},X^{k,u^k}_{T^{k}_{n-1}},u^k_{n-1},\ell_n (\mathcal{A}^k_n)\Big) = X^{k,u^k}(T^{k}_{n-1}) \vartheta(Z^k(T^k_{n}),u^k_{n-1})\Delta A^{k,1}(T^k_n),
$$
where $\ell_n(\mathcal{A}^k_n) = (Z^k(T^k_{n}), \Delta A^{k,1}(T^k_n))$ and $Z^k(T^k_n)$ is an $\mathscr{D}$-adapted discrete version of the fractional Ornstein-Uhlenbeck process (see e.g \cite{cheridito}) considered as a functional of $B^k_H$. Similar to (\ref{caseB}), the optimal control $u^{k,*}$ derived from (\ref{DPPInt}) turns out to be open-loop.

Theorem \ref{absrate1} and \ref{THdriftrate} present the convergence analysis of the numerical scheme. In general, if $\xi$ is $\gamma$-H\"older continuous on the space of c\`adl\`ag paths (Assumption (A1) in section \ref{basics}), then for a given $0 < \beta < 1$ and $\epsilon >0$, we have

\begin{equation}\label{r1intr}
\Bigg|\sup_{\phi\in U^{k,e(k,T)}_0}\mathbb{E}\big[\xi\big(X^{k,u^k}\big)\big] - \sup_{\phi\in U^{T}_0}\mathbb{E}\big[\xi\big(X^{u}\big)\big]\Bigg|\lesssim \Big\{h_k^\gamma + |L_\epsilon|^\gamma \epsilon_k^{\gamma\beta}\Big\} + \epsilon\rightarrow 0,
\end{equation}
as $k\rightarrow +\infty$.

We now discuss the terms $(\epsilon^{\beta}_k, h_k, L_\epsilon)$ which constitute the fundamental quantities for the general convergence of the method for given  $\epsilon>0, 0 < \beta <1$ and a chosen sequence $\lim_{k\downarrow0}\epsilon_k=0$ associated with the Brownian motion approximation $\mathscr{D} = \{\mathcal{T}, A^{k}; k\ge 1\}$ as described in Definition \ref{discreteskeleton}. The component $\epsilon^{\gamma\beta}_k$ in the righ-hand side of (\ref{r1intr}) is a function of the quantity $\{\mathbb{E}|T-T^k_{e(k,T)}|^2\}^{\frac{\gamma}{2}}$ which is analyzed by large deviation principles (see Section \ref{LDappendix} and Lemma \ref{Tkrate}). The component $h_k$ is related to the convergence rate of $\mathcal{X}$ to the target controlled state $X$ which drives the control problem (\ref{INTROpr3}). In typical non-Makovian situations as described by the examples (\ref{caseA}), (\ref{caseB}) and (\ref{caseC}), the rate $h_k$ can be explicitly computed (see Propositions \ref{mainPrSDEBM}, \ref{mainPrSDEFBM} and \ref{mainPrRV}). The presence of the constant $L_\epsilon$ in (\ref{r1intr}) is due to a Lipschitz-type condition imposed on an approximation scheme for the generic controls in $U^T_0$ (see Lemma \ref{Liplemma}) and $L_\epsilon\rightarrow +\infty$ as $\epsilon \downarrow 0$. In general, under Assumptions (A1-B1), for a given $\epsilon >0$, (\ref{r1intr}) ensures the convergence of the method for general controlled states including (\ref{caseA}), (\ref{caseB}) and (\ref{caseC}) (see Theorem \ref{absrate1}). The explicit convergence rate is established in case the control only affects the drift in the examples (\ref{caseA}), (\ref{caseB}) and (\ref{caseC}). In this case, under a suitable convexity assumption (see Assumption (F1) in section \ref{pcdsection}), Theorem \ref{THdriftrate} says that one can take $L_\epsilon =0$ and the leading term in the right-hand side of (\ref{r1intr}) becomes only a function of $h^\gamma_k$. The full explicit convergence rate of the method depends on the unknown asymptotic rate $\lim_{\epsilon\rightarrow+\infty} L_\epsilon=+\infty$ and we postpone this analysis to a future project. For further details on the appearance of $L_\epsilon$ in (\ref{r1intr}), we refer the reader to Remark \ref{Liprem}.

The remainder of this article is organized as follows. Section \ref{basics} presents some notations and summarizes the standing assumptions of this article. Section \ref{CIDSsection} presents the basic discretization scheme and its application to the control problem. Section \ref{mainrsection} presents the main results of this paper and a discussion of the dynamic programming algorithm associated with the discrete version of the nonlinear expectation. Section \ref{Eulerproofs} presents the precise rates $h_k$ in (\ref{r1intr}) associated with (\ref{caseA}), (\ref{caseB}) and (\ref{caseC}). Section \ref{numersection} presents a pseudo-code for our method jointly with a simple illustrative numerical example. Appendix \ref{LDappendix} presents the large deviation-type results associated with the convergence rate of the method.  Appendix \ref{MSTappendix} presents the technical details on the measurable selection arguments associated with the proof of Theorem \ref{absrate1}.

%and a detailed proof of the dynamic programming algorithm described in section \ref{DPPsubsection}.

\section{Controlled stochastic processes}\label{basics}
\subsection{Notation}
At first, we introduce some notation. In the sequel, finite-dimensional spaces will be equipped with a norm $|\cdot|$ and $T$ is a finite terminal time. We further write $a\lesssim b$ for two positive quantities $a$ and $b$ to express an estimate of the form $a \le C b$, where $C$ is a generic constant which may differ from line to line. We use the notation $a\vee b:= \max\{a,b\}$ and $a\wedge b:= \min\{a,b\}$ for $a,b \in \mathbb{R}$. Throughout this article, $(\Omega, \mathbb{F},\mathbb{P})$ is the classical Wiener space, where $\Omega  = C(\mathbb{R}_+; \mathbb{R}^d)$ is equipped with the uniform convergence on compact sets of $\mathbb{R}_+$, $\mathbb{F}:=(\mathcal{F}_t)_{t\ge 0}$ is the
usual $\mathbb{P}$-augmentation of the filtration generated by the coordinate process, $\mathbb{P}$ is the Wiener measure. The $d$-dimensional Brownian motion will be denoted by $B = \{B^{1},\ldots,B^{d}\}$.

Let $L^p_a(\mathbb{P}\times Leb)$ be the Banach space of all $\mathbb{F}$-adapted finite-dimensional processes $Y$ such that

%\begin{equation}\label{Lpspace}
$$\mathbb{E}\int_0^T|Y(t)|^pdt < \infty,$$
%\end{equation}
for $1\le p< \infty$. Let us denote

$$\|f\|_\infty:=\sup_{0\le t\le T}|f(t)|.$$

%We also define $\mathbf{B}^p(\mathbb{F})$ as the space of all $\mathbb{F}$-adapted c\`adl\`ag processes $Y$ such that

%$$
%\|Y\|^p_{\mathbf{B}^p}:=\mathbb{E}\|Y\|^p_\infty < \infty,
%$$
%for $p\ge 1$.
The action space is a compact set

$$\mathbb{A}:=\{(x_1, \ldots, x_r)\in \mathbb{R}^r; \max_{1\le i\le r}|x_i|\le \bar{a}\},$$
for a given fixed $0 < \bar{a}< +\infty$ and positive integer $r\ge 1$.

In the sequel, we will denote $\mathbf{D}_{n,T}:= \{h:[0,T]\rightarrow \mathbb{R}^n~\text{with c\`adl\`ag paths}\}$ and we equip this linear space with the uniform convergence on $[0,T]$.
\subsection{Controlled Wiener functionals}
For a pair of finite $\mathbb{F}$-stopping times $(M,N)$, we denote

$$
]]M,N ]]: = \{(\omega,t); M(\omega) < t \le N(\omega)\},
$$
and $]] M, +\infty[[:=\{(\omega,t); M(\omega)< t < +\infty\}.$
In order to set up the basic structure of our control problem, we first need to define the class of admissible control processes: for each pair $(M,N)$ of a.s finite $\mathbb{F}$-stopping times such that $M < N~a.s$, we denote~\footnote{Whenever necessary, we can always extend a given $u\in U^N_M$ by setting $u = 0$ on the complement of a stochastic set $]]M,N]]$.}

$$
U^{N}_{M}:=\{\text{the set of all}~\mathbb{F}-\text{predictable processes}~u:~]]M,N]]\rightarrow \mathbb{A}; u(M+)~\text{exists}\}.
$$
For such family of processes, we observe they satisfy the following properties:

\

%~\footnote{We set $(u\otimes_0 v) = u$ if $u\in U^T_0$}
\begin{itemize}\label{3-proper}
  \item Restriction: $u\in U_M^N\Rightarrow u\mid_{]]M,P]]}\in U_{M}^P$ for $M < P \le N$~a.s.
  \item Concatenation: If $u\in U_{M}^N$ and $v\in U_N^P$ for $M<N< P~a.s$, then $(u\otimes_N v)(\cdot)\in U_M^P$, where
  \begin{equation}\label{concatenation}
  (u\otimes_N v)(r):=\left\{
\begin{array}{rl}
u(r); & \hbox{if} \ M < r \le N \\
v(r);& \hbox{if} \ N < r \le P.
\end{array}
\right.
\end{equation}
  \item Finite Mixing: For every $u,v\in U_M^N$ and $G\in \mathcal{F}_M$, we have
$$u\mathds{1}_G + v\mathds{1}_{G^c}\in U^N_M.$$
\end{itemize}

\begin{definition}\label{continuouscontrolled}
A continuous \textbf{controlled Wiener functional} is a map $X:U^T_0\rightarrow L^p_a(\mathbb{P}\times Leb)$ for some $p\ge 1$, such that for each $t\in [0,T]$ and $u\in U^T_0$, $\{X^u(s); 0\le s\le t\}$ depends on the control $u$ only on $(0,t]$ and $X^u$ has continuous paths for each $u\in U^T_0$.
%and $X(\cdot,u)$ has c\`adl\`ag paths.
\end{definition}

 %Throughout this paper, we assume the following regularity properties on the payoff functional:

%For simplicity, we omit the dependence of polar sets on $\mathcal{G}(X)$. A property $\mathcal{P}$ holds quasi-surely if $\mathcal{P}$ holds up to a polar set.

We now present the two standing assumptions of this article.

\

\noindent $\textbf{Assumption \text{(A1)}}$: The payoff $\xi:\mathbf{D}_{n,T}\rightarrow\mathbb{R}$ satisfies the following regularity assumption: There exists $\gamma \in (0,1]$ and a constant $\|\xi\|>0$ such that

%\begin{equation}\label{descA1}
$$|\xi(f) - \xi(g)|\le \|\xi\| \|f-g\|_\infty^\gamma,$$
%\end{equation}
for every $f,g\in \mathbf{D}_{n,T}$.

\

\noindent \textbf{Assumption (B1)}: There exists a constant $C$ such that

$$
\mathbb{E}\|X^u-X^\eta\|^2_{\infty}\le C \mathbb{E}\int_0^T|u(s)-\eta(s)|^2 ds,
$$
for every $u,\eta\in U^T_0$.

\begin{remark}
Path-dependent controlled SDEs with Lipschitz coefficients are typical examples of controlled Wiener functionals satisfying Assumption (B1). See Remark \ref{remarkassB1}
\end{remark}

%\begin{remark}
%Even though we are only interested in controlled Wiener functionals with continuous paths, we are forced to assume the payoff functional is defined on the space of c\`adl\`ag paths due to a discretization procedure. However, this is not a strong assumption since most of the functionals of interest admits extensions from $\mathbf{C}_{n,T}$ to $\mathbf{D}_{n,T}$ preserving Assumption (A1).
%\end{remark}

For a given payoff $\xi:\mathbf{D}_{n,T}\rightarrow\mathbb{R}$ and a controlled Wiener functional $X:U^T_0\rightarrow\mathbf{B}^2(\mathbb{F})$, we denote

$$
\xi_X(u): = \xi\big(X^u\big); u\in U^T_0
$$
and
\begin{equation}\label{valuepdef}
V(t,u):=\text{ess}~\sup_{v\in U^T_t}\mathbb{E}\Big[\xi_X(u\otimes_t v)|\mathcal{F}_t\Big];0\le t< T,u \in U^T_0,
\end{equation}
where $V(T,u) := \xi_X(u)$~a.s. We stress the process $V(\cdot, u)$ has to be viewed backwards for each control $u\in U^T_0$. Throughout this paper, in order to keep notation simple, we omit the dependence of the value process in (\ref{valuepdef}) on the controlled Wiener functional $X$ and the payoff $\xi$, so we write $V$ meaning as a map $V:U^T_0\rightarrow L^1_a(\mathbb{P}\times Leb)$.

Since we are not assuming that $\mathcal{F}_0$ is the trivial $\sigma$-algebra, we cannot say that $V(0)$ is deterministic. However, the finite-mixing property on the class of admissible controls implies that $\{\mathbb{E}\big[\xi_X(u\otimes_t \theta)|\mathcal{F}_t\big]; \theta\in U_t^T\}$ has the lattice property (see e.g Def 1.1.2 \cite{lamberton}) for every $t\in [0,T)$ and $u\in U^T_t$. In this case, it is known

$$\mathbb{E}\big[V(0)\big] = \sup_{v\in U^T_0}\mathbb{E}\big[\xi(X^v)\big].$$

\begin{remark} \label{remark_consist}
One can easily check that under Assumptions (A1-B1), for any $u\in U^T_0$, $\{V(s, u); 0\le s\le t\}$ depends only on the control $u$ restricted to the interval $[0,t]$. Moreover, $u\mapsto V(\cdot,u)$ is a continuous controlled Wiener functional in the sense of Definition \ref{continuouscontrolled} and $V$ is an $U^T_0$-supermartingale, in the sense that $V(\cdot, u)$ is an $\mathbb{F}$-supermartingale for each $u\in U^T_0$.
\end{remark}

\begin{definition}\label{epsilonDEF}
We say that $u\in U^T_0$ is an $\epsilon$-optimal control if
\begin{equation}\label{optimaldef1}
\mathbb{E}\big[\xi_X(u)\big]\ge \sup_{\eta\in U^T_0}\mathbb{E}\big[\xi_X(\eta)\big]-\epsilon.
\end{equation}
In case, $\epsilon=0$, we say that $u$ realizing (\ref{optimaldef1}) is an exact optimal control.
\end{definition}

\section{Discrete-type skeleton for the Brownian motion and controlled imbedded discrete structures}\label{CIDSsection}

In this section, we introduce what we call a \textit{controlled imbedded discrete structure} which is a natural extension of the approximation models presented in \cite{LEAO_OHASHI2017.1}. Our philosophy is to view a controlled Wiener functional as a family of simplified models one has to build in order to extract the relevant information for the obtention of a concrete description of value processes and the construction of their associated (near or exact) optimal controls.

\subsection{The underlying discrete skeleton}\label{discretesec}
The discretization procedure will be based on a class of pure jump processes driven by suitable waiting times which describe the local behavior of the Brownian motion. We briefly recall the basic properties of this skeleton. For more details, we refer the reader to the work \cite{LEAO_OHASHI2017.1}. We start by constructing a sequence $\mathcal{T} := \{T^k_n; n\ge 0\}$ of hitting times which will be the basis for our discretization scheme. For a given
sequence $\epsilon_k$ such that $\epsilon_k\downarrow 0$ as $k\rightarrow +\infty$, we set $T^{k}_0:=0$ and

\begin{equation}\label{stopping_times}
T^{k}_n := \inf\{T^{k}_{n-1}< t <\infty;  | B(t) - B(T^{k}_{n-1}) | = \epsilon_k\}, \quad n \ge 1,
\end{equation}
and $|\cdot |$ in (\ref{stopping_times}) corresponds to the maximum norm on $\mathbb{R}^d$. This implies

$$
\Delta T^k_{n}:=T^k_n-T^k_{n-1}=\min_{j\in\{1,2,\dots, d\}}{\{\Delta^{k,j}_n\}}~a.s,
$$
where

\begin{equation*}
\Delta^{k,j}_n := \inf\{0< t <\infty;  | B^j(t+T^{k}_{n-1}) - B^j(T^{k}_{n-1}) | = \epsilon_k\}, \quad n \ge 1.
\end{equation*}
Then, we define $A^k :=(A^{k,1}, \cdots , A^{k,d})$ by

%\begin{equation}\label{rw}
$$A^{k,j} (t) := \sum_{n=1}^{\infty} \left( B^j(T^{k}_{n}) - B^j(T^{k}_{n-1})  \right) \mathds{1}_{\{T^{k}_n\leq t \}};~t\ge0, ~ j=1, \ldots , d,$$
%\end{equation}
for integers $k\ge 1$.

The multi-dimensional filtration generated by $A^{k}$ is naturally characterized as follows. Let $\widetilde{\mathbb{F}}^k := \{\widetilde{\mathcal{F}}^k_t ; 0 \leq t <\infty\}$ be the filtration generated by $A^k$. We observe

$$
\widetilde{\mathcal{F}}^k_t \cap\{T^k_n \le t < T^k_{n+1}\} = \widetilde{\mathcal{F}}^k_{T^k_n}\cap \{T^k_n \le  t < T^k_{n+1}\}; t\ge 0,
$$
where $\widetilde{\mathcal{F}}^k_{T^k_n} = \sigma(A^{k}(s\wedge T^k_n); s\ge 0)$ for each $n\ge 0$. Let $\mathcal{F}^{k}_\infty$ be the completion of $\sigma(A^{k}(s); s\ge 0)$ and let $\mathcal{N}_{k}$ be the $\sigma$-algebra generated by all $\mathbb{P}$-null sets in $\mathcal{F}^{k}_\infty$. We denote $\mathbb{F}^{k} = (\mathcal{F}^{k}_t)_{t\ge 0}$, where $\mathcal{F}^{k}_t$ is the usual $\mathbb{P}$-augmentation (based on $\mathcal{N}_k$) satisfying the usual conditions.

\begin{definition}\label{discreteskeleton}
The structure $\mathscr{D} = \{\mathcal{T}, A^{k}; k\ge 1\}$ is called a \textbf{discrete-type skeleton} for the Brownian motion.
\end{definition}

By the strong Markov property, we observe that

\begin{enumerate}
  \item The jumps $\Delta A^{k,j}(T^k_n) = A^{k,j}(T^k_n)-A^{k,j}(T^k_{n}-); n=1, 2, \ldots$ are independent and identically distributed (iid).
  \item The waiting times $\Delta T^k_n; n=1, 2,\ldots$ are iid random variables in $\mathbb{R}_+$.
  \item The families $(\Delta A^{k,j}(T^k_n); n=1, 2, \ldots)$ and $(\Delta T^k_n; n=1, 2,\ldots)$ are independent.
\end{enumerate}
Moreover, it is immediate that $A^{k,j}$ is a square-integrable $\mathbb{F}^k$-martingale for each $j=1,\ldots, d$.
\begin{remark}
The skeleton given in Definition \ref{discreteskeleton} slightly differs from Def 2.1 in \cite{LEAO_OHASHI2017.1} and \cite{LEAO_OHASHI2017.2} if $d>1$. Indeed, the choice of the hitting times $\{T^k_n; n\ge 1\}$ differs from \cite{LEAO_OHASHI2017.1}. In the present work, we adopt (\ref{stopping_times}) rather than
$$
T^{k,j}_n := \inf\big\{T^{k,j}_{n-1}< t <\infty;  |B^{j}(t) - B^{j}(T^{k,j}_{n-1})| = \epsilon_k\big\}; \quad n \ge 1,
$$
for $j\in \{1,\ldots,d \}$. This allows us to reduce the complexity of the discretization scheme. Since $(A^{k,1}, \ldots, A^{k,d})$ is a $d$-dimensional stepwise constant martingale process satisfying properties (1), (2) and (3) above, then the basic underlying differential structure presented in \cite{LEAO_OHASHI2017.1} is still preserved. This article does not focus on representation results for controlled Wiener functionals. Hence the limiting differential structure will not be investigated and we leave this construction to a future work.
\end{remark}

Let us start to introduce a subclass $U^{k,T^k_n}_{T^k_\ell}\subset U_{T^k_\ell}^{T^k_n}; 0\le \ell < n < \infty$. For $\ell < n$, let $U^{k,T^k_n}_{T^k_\ell}$ be the set of $\mathbb{F}^k$-predictable processes of the form

\begin{equation}\label{controlform0}
v^k(t) = \sum_{j=\ell+1}^{n}v^{k}_{j-1}\mathds{1}_{\{T^k_{j-1}< t\le T^k_j\}}; \quad T^k_\ell < t \le T^k_n,
\end{equation}
where for each $j=\ell+1, \ldots, n$, $v^k_{j-1}$ is an $\mathbb{A}$-valued $\mathcal{F}^k_{T^k_{j-1}}$-measurable random variable. To keep notation simple, we use the shorthand notations

$$
U^{k,n}_\ell: = U^{k,T^k_n}_{T^k_\ell}; 0\le \ell < n,
$$
where $U^{k,\infty}_\ell$ is the set of all controls $v^k:~]]T^k_\ell, +\infty[[\rightarrow\mathbb{A}$ of the form

$$
v^k(t) = \sum_{j\ge \ell+1}v^{k}_{j-1}\mathds{1}_{\{T^k_{j-1}< t\le T^k_j\}}; \quad T^k_\ell < t,
$$
where $v^k_{j-1}$ is an $\mathbb{A}$-valued $\mathcal{F}^k_{T^k_{j-1}}$-measurable random variable for every $j\ge \ell+1$ for an integer $\ell \ge 0$.

We also use a shorthand notation for $u^k\otimes_{T^k_\ell} v^k$ as follows: with a slight abuse of notation, for $u^k\in U^{k,n}_0$ and $v^k\in U^{k,n}_{\ell}$ with $\ell < n$, we write

$$
(u^k\otimes_\ell v^k): = (u^k_0, \ldots, u^k_{\ell-1}, v^k_\ell, \ldots, v^k_{n-1} ).
$$
This notation is consistent because $u^k\otimes_{T^k_\ell} v^k$ only depends on the list of variables $(u^k_0, \ldots, u^k_{\ell-1}, v^k_\ell, \ldots, v^k_{n-1} )$ whenever $u^k:~]]0,T^k_n]]\rightarrow\mathbb{A}$ and $v^k:~]]T^k_\ell,T^k_n]]\rightarrow\mathbb{A}$ are controls of the form (\ref{controlform0}) for $\ell < n$.

%Of course, this notation extends naturally to concatenations over the sets $\widetilde{U}^{k,n}_m$.

Let us now introduce the analogous concept of controlled Wiener functional but based on the filtration $\mathbb{F}^k$. For this purpose, we need to introduce some further notations. Let us define

$$
e(k,t):= \Big\lceil \frac{\epsilon^{-2}_k t}{\chi_d}\Big\rceil; 0\le t\le T,
$$
where $\lceil x\rceil$ is the smallest integer greater or equal to $x\ge 0$ and

\begin{equation}\label{kappaDEF}
\chi_d:=\mathbb{E}\min\{\tau^1, \ldots, \tau^d\},
\end{equation}
where $(\tau^j)_{j=1}^d$ is an iid sequence of random variables with distribution $\inf\{t>0; |W(t)|=1\}$ for a real-valued standard Brownian motion $W$. From Lemma \ref{Tkrate}, for each $t\in [0,T]$, we know that
$$T^k_{e(k,t)}\rightarrow t~\text{a.s and in}~L^p(\mathbb{P}),$$
as $k\rightarrow+\infty$, for each $t\ge 0$ and $p\ge 1$.

\begin{remark}\label{growthlemma}
The number $e(k,T)$ should be interpretted as the number of necessary steps to compute a given nonlinear expectation (driven by a controlled Wiener functional) via a discrete-type dynamic programming equation. One can easily check that (see e.g \cite{Burq_Jones2008}) $0\le f(t) \le 1; t\ge 0$, where $f$ is the density of $\tau^1$ in (\ref{kappaDEF}). By Th 5 in \cite{gordon}, we get a lower bound

$$\frac{1}{2 d}\le \chi_d.$$
%\end{equation}
Therefore, for given $k\ge 1$ and $T$, the number of periods $e(k,T)$ grows no faster than the dimension of the driving Brownian motion.
\end{remark}

Let $O_T(\mathbb{F}^k)$ be the set of all stepwise constant $\mathbb{F}^k$-optional processes of the form

$$Z^k(t) = \sum_{n=0}^\infty Z^k(T^k_n)\mathds{1}_{\{T^k_n\le t\wedge T^k_{e(k,T)} < T^k_{n+1}\}}; 0\le t\le T,$$
where $Z^k(T^k_n)$ is $\mathcal{F}^k_{T^k_{n}}$-measurable for every $n\ge 0$ and $k\ge 1$.

Let us now present two concepts which will play a key role in this work.
\begin{definition}\label{GASdef}
A \textbf{controlled imbedded discrete structure} $\mathcal{Y} = \big((Y^k)_{k\ge 1},\mathscr{D}\big)$ consists of the following objects: a discrete-type skeleton $\mathscr{D}$ and a map $u^k\mapsto Y^{k,u^k}$ from $U^{k,e(k,T)}_0$ to $O_T(\mathbb{F}^k)$ such that

\begin{equation}\label{antiprop}
Y^{k,u^k}(T^k_{n+1})~\text{depends on the control only at}~(u^k_0, \ldots, u^k_n),
\end{equation}
for each integer $n\in \{0,\ldots,e(k,T)-1\}$.
\end{definition}

\begin{definition}\label{GASdefSTRONG}
A \textbf{strongly controlled Wiener functional} is a pair $(X,\mathcal{X})$, where $X$ is a controlled Wiener functional and $\mathcal{X} = \big((X^k)_{k\ge 1},\mathscr{D}\big)$ is a controlled imbedded discrete structure such that $\big\{\|X^{k,\phi}\|_\infty; \phi \in U^{k,e(k,T)}_0\big\}$ is uniformly integrable for each $k\ge 1$ and there exists a positive sequence $h_k \downarrow 0$ such that

\begin{equation}\label{keyassepsilon}
\sup_{\phi\in U^{k,e(k,T)}_0}\mathbb{E}\|X^{k,\phi} - X^{\phi}\|_{\infty}\lesssim h_k,
\end{equation}
for $k\ge 1$.

%\begin{equation}\label{keyassepsilon}
%\sup_{\phi\in U^{k,e(k,T)}_0}\mathbb{E}\|X^{k,\phi} - X^{\phi}\|^\gamma_{\infty}=O(h_k),
%\end{equation}
%for $k\ge 1$ and $0 < \gamma \le 1$.
\end{definition}
The concepts of controlled imbedded discrete structures and strongly controlled Wiener functionals are nonlinear versions of the structures analyzed in \cite{LEAO_OHASHI2017.1}.

\begin{remark}
Although the notion of strongly controlled Wiener functional is fully determined by the filtration $\mathbb{F}^k$ and the controls in the set $U^{k,e(k,T)}_0$, sometimes it is useful to extend $\phi \mapsto X^{k,\phi}$ from $U^{k,e(k,T)}_0$ to $U^T_0$. In typical examples, this is possible and natural due to our choice of the discrete-type skeleton $\mathscr{D} = \{\mathcal{T}, A^k; k\ge 1\}$. See sections \ref{subexbm}, \ref{FBMsub} and \ref{rvsub}.
\end{remark}

Next, we provide concrete examples of controlled imbedded discrete structures for the three classes of controlled states that we study in this article. In the sequel, we make use of the following notation

$$
\omega_t: = \omega(t\wedge \cdot); \omega \in \mathbf{D}_{n,T}.
$$

This notation is naturally extended to processes. We say that $F$ is a \textit{non-anticipative} functional if it is a Borel mapping and

$$F(t,\omega) = F(t,\omega_t); (t,\omega)\in[0,T]\times \mathbf{D}_{n,T}.$$

\subsection{Path-dependent controlled SDEs driven by Brownian motion}\label{subexbm}
The first example is the following $n$-dimensional controlled SDE

\begin{equation}\label{pdsdeBM}
dX^u(t) = \alpha(t,X^u_t,u(t))dt + \sigma(t,X^u_t,u(t))dB(t); 0\le t\le T,
\end{equation}
with a given initial condition $X^u(0)=x_0\in \mathbb{R}^n$. We define $\Lambda_T:=\{(t,\omega_t); t\in [0,T]; \omega\in \mathbf{D}_{n,T}\}$ and we endow this set with the metric
%$$d_{1/2}((t,\omega); (t',\omega')): = \sup_{0\le u\le T}\|\omega(u\wedge t) - \omega'(u\wedge t')\|_{\mathbb{R}^n} + |t-t'|^{1/2}.$$

$$d_{\theta}((t,\omega); (t',\omega')): = \|\omega_t - \omega'_{t'}\|_{\infty} + |t-t'|^\theta,$$
for $0 < \theta \le 1$. Then, $(\Lambda_T,d_{\theta})$ is a complete metric space equipped with the Borel $\sigma$-algebra. The coefficients of the SDE will satisfy the following regularity conditions:

\

\noindent \textbf{Assumption (C1)}: The non-anticipative mappings $\alpha: \Lambda_T\times \mathbb{A}\rightarrow \mathbb{R}^n$ and $\sigma:\Lambda_T\times\mathbb{A}\rightarrow \mathbb{R}^{n\times d}$ are Lipschitz continuous, i.e., there exists a pair of constants $K_{Lip}=(K_{1,Lip}, K_{2,Lip})$ such that

$$|\alpha(t,\omega, a) - \alpha(t',\omega',b)| + |\sigma(t,\omega,a) - \sigma(t',\omega',b)|
$$
$$\le K_{1,Lip}d_{\theta} \big((t,\omega); (t',\omega')\big) + K_{2,Lip}|a-b|,$$
for every $t,t'\in [0,T]$ and $\omega,\omega'\in \mathbf{D}_{n,T}$ and $a,b\in \mathbb{A}$. One can easily check by routine arguments that the SDE (\ref{pdsdeBM}) admits a strong solution and

$$
\sup_{u\in U^T_0}\mathbb{E}\|X^u_T\|^{2p}_{\infty}\le C(1+|x_0|^{2p})\exp(CT),
$$
where $X(0)=x_0$, $C$ is a constant depending on $T>0,p\ge 1$, $K_{Lip}$ and the compact set $\mathbb{A}$.

\begin{remark}\label{remarkassB1}
Due to Assumption (C1), one can apply Jensen and Burkholder-Davis-Gundy's inequalities to arrive at the following estimate: there exists a constant $C= C(d,T,K_{1,Lip},K_{2,Lip})$ such that

\begin{eqnarray*}
\mathbb{E}\|X^{u_1}_t - X^{u_2}_t\|^2_{\infty}&\le& C \int_0^t \mathbb{E}\|X^{u_1}_s - X^{u_2}_s\|^2_{\infty}ds\\
& &\\
&+& C \int_0^t \mathbb{E}| u^1(s) - u^2(s)|^2ds,
\end{eqnarray*}
for $0\le t\le T$ and $u_1,u_2\in U^T_0$. Then, by applying Gr\"{o}nwall's inequality, we observe the controlled SDE (\ref{pdsdeBM}) satisfies Assumption (B1).
\end{remark}

In the sequel, it is convenient to introduce the following notation: for each $t\ge 0$, we set

\begin{equation}\label{localt}
\bar{t}_k:=\max\{T^k_n; T^k_n \le t\}.
\end{equation}

Let us fix a control $\phi \in U^{T}_0$. At first, we construct an Euler-Maruyama-type scheme based on the random partition $(T^k_n)_{n\ge 0}$ as follows. Let us define $\mathbb{X}^{k,\phi}(0):=\mathbb{X}^{k,\phi}_0:=\mathbb{X}^{k,\phi}_{T^k_0} := x$ is the constant function $x$ over $[0,T]$. Let us define $\mathbb{X}^{k,\phi}(T^k_j):=(\mathbb{X}^{1,k,\phi}(T^k_j), \ldots, \mathbb{X}^{n,k,\phi}(T^k_j))$ as follows

\begin{eqnarray}
\nonumber\mathbb{X}^{k,\phi}(T^k_j)&:=&\mathbb{X}^{k,\phi}(T^k_{j-1}) + \int_{T^k_{j-1}}^{T^k_j}\alpha \Big(T^k_{j-1},\mathbb{X}^{k,\phi}_{T^k_{j-1}},\phi(s)\Big)ds\\
\label{Xeuler}& &\\
\nonumber&+& \int_{T^k_{j-1}}^{T^k_j}\sigma \Big(T^{k}_{j-1},\mathbb{X}^{k,\phi}_{T^{k}_{j-1}},\phi(s)\Big)dB(s),
\end{eqnarray}
for $j\ge 1$, where $\mathbb{X}^{k,\phi}(t) := \sum_{n=0}^\infty \mathbb{X}^{k,\phi}(T^k_n)\mathds{1}_{\{T^k_n \le t < T^k_{n+1}\}}; t\ge 0$. It is important to notice that if $\phi = (v^k_0, v^k_1, \ldots) \in U^{k,\infty}_0$, then

\begin{eqnarray}
\nonumber\mathbb{X}^{k,\phi}(T^k_j)&=&\mathbb{X}^{k,\phi}(T^k_{j-1}) +\alpha \Big(T^k_{j-1},\mathbb{X}^{k,\phi}_{T^k_{j-1}},v^k_{j-1}\Big)\Delta {T^k_{j}}\\
\label{Xeulerk}& &\\
\nonumber&+& \sigma \Big(T^{k}_{j-1},\mathbb{X}^{k,\phi}_{T^{k}_{j-1}},v^k_{j-1}\Big)\Delta A^k(T^k_j),
\end{eqnarray}
for $j\ge 1$.

Then, we set $\mathcal{X} = \big( (X^k)_{k\ge 1},\mathscr{D}\big)$ given by

\begin{equation}\label{xksde}
X^{k,\phi}(t):= \mathbb{X}^{k,\phi}(t\wedge T^k_{e(k,T)}),
\end{equation}
for $\phi\in U^{T}_0$.

\subsection{Path-dependent controlled SDEs driven by fractional Brownian motion}\label{FBMsub}
Let $X^u$ be the controlled Wiener functional of the form

\begin{equation}\label{limsdefbm}
dX^u(t) = \alpha(t,X_t,u(t))dt + \sigma dB_H(t),
\end{equation}
where $X(0)=x_0\in \mathbb{R}$, $\sigma$ is a constant, $\alpha$ is a non-anticipative functional satisfying Assumption (C1) and $B_H$ is a FBM with Hurst exponent $H \in (0,\frac{1}{2})$. For simplicity of presentation, we set the state dimension equals to one. Under Assumption C1, by a standard fixed point argument, one can show there exists a unique strong solution for (\ref{limsdefbm}) for each $u\in U^T_0$. Moreover, the following simple remark holds true.

%The path-dependence feature is more sophisticated than previous example because the lack of Markov property comes from distorting the driving Brownian motion by a singular kernel and also from a non-anticipative drift $\alpha$ which depends on the whole solution path.

\begin{remark}\label{remarkassB1FBM}
The controlled SDE (\ref{limsdefbm}) satisfies Assumption B1 due to modulus of continuity of $\alpha$ given by Assumption C1.
\end{remark}
Let $\mathbf{C}^\lambda_0$ be the space of $\lambda$-H\"{o}lder continuous real-valued functions on $[0,T]$ and starting at zero equipped with the usual norm. Let us define

\begin{equation*}
\begin{split}
K_{H,1}(t,s)&:=c_H t^{H-\frac{1}{2}}s^{\frac{1}{2}-H}(t-s)^{H-\frac{1}{2}},\\
K_{H,2}(t,s)&:=c_H(1/2-H)s^{\frac{1}{2}-H}\int_s^t u^{H-\frac{3}{2}}(u-s)^{H-\frac{1}{2}}du,
\end{split}
\end{equation*}
for $0< s < t$, where $c_H$ is a suitable constant (see e.g \cite{ohashifrancys}). For each $f \in \mathbf{C}^\lambda_0$, we define

\begin{equation}\label{LambdaH}
(\Lambda_Hf)(t):=\int_0^t [f(t)-f(s)]\partial_s K_{H,1}(t,s)ds - \int_0^t \partial_s K_{H,2}(t,s)f(s)ds,
\end{equation}
for $0\le t\le T.$
By Theorem 2.2 in \cite{ohashifrancys}, it is known that

$$B_H(t) = (\Lambda_H B)(t)~a.s.$$
In the sequel, we recall (see (\ref{localt})) $\bar{t}_k = \max\{T^k_n; T^k_n \le t\}$ and we define

$$
\bar{t}^{+}_k:= \min\{T^{k}_n ; \bar{t}_k < T^{k}_n\}\wedge T.
$$

The discrete structure for FBM with $0 < H< \frac{1}{2}$ is given by
$$B^{k}_H(t):=  \int_0^{\bar{t}_k} \partial_s K_{H,1}(\bar{t}_k,s)\big[A^{k}(\bar{t}_k) - A^{k}(\bar{s}^{+}_k)\big]ds - \int_0^{\bar{t}_k}\partial_s K_{H,2}(\bar{t}_k,s)A^{k}(s)ds.$$
We observe

$$B^k_{H}(t) = \sum_{n=0}^\infty B^k_{H}(T^k_n)\mathds{1}_{\{T^k_n\le t < T^k_{n+1}\}}; 0\le t\le T.$$

The controlled structure $\big((X^k)_{k\ge 1},\mathscr{D}\big)$ associated with (\ref{limsdefbm}) is given as follows: let us fix a control $\phi\in U^{T}_0$. Let us define

\begin{eqnarray}\label{eulerFBMd}
\mathbb{X}^{k,\phi}(T^k_{n})&:=&\mathbb{X}^{k,\phi}(T^k_{n-1}) + \int_{T^k_{n-1}}^{T^k_n}\alpha\Big(T^k_{n-1},\mathbb{X}^{k,\phi}_{T^k_{n-1}},\phi(s)\Big)ds\\
\nonumber& &\\
\nonumber&+& \sigma\Delta B^k_{H}(T^k_{n}),
\end{eqnarray}
where $\mathbb{X}^{k,\phi}(t) := \sum_{n=0}^\infty \mathbb{X}^{k,\phi}(T^k_n)\mathds{1}_{\{T^k_n \le t < T^k_{n+1}\}}; t\ge 0$. It is important to notice that if $\phi = (v^k_0, v^k_1, \ldots) \in U^{k,\infty}_0$, then

\begin{eqnarray}
\label{fulleFBM}\mathbb{X}^{k,\phi}(T^k_{n})&=&\mathbb{X}^{k,\phi}(T^k_{n-1}) + \alpha\Big(T^k_{n-1},\mathbb{X}^{k,\phi}_{T^k_{n-1}},v^k_{n-1}\Big)\Delta T^k_n\\
\nonumber& &\\
\nonumber&+& \sigma\Delta B^k_{H}(T^k_{n}),
\end{eqnarray}
for $n\ge 1$. We set $\mathcal{X} = \big( (X^k)_{k\ge 1},\mathscr{D}\big)$, where
\begin{equation}\label{xksdeFBM}
X^{k,\phi}(t):=\mathbb{X}^{k,\phi}(t\wedge T^k_{e(k,T)}); 0\le t\le T,
\end{equation}
for $\phi \in U^T_0$.

\subsection{Path-dependent controlled SDEs with rough stochastic volatility}\label{rvsub}
Let us now present the third type of non-Markovian controlled state investigated in this article:
\begin{equation}\label{priceprocess}
\left\{
\begin{array}{lc}
dX^u(t) = X^u(t) \mu(u(t))dt + X^u(t)  \vartheta(Z(t),u(t)) dB^1(t) \\
dZ(t) = \nu dW_H(t)- \beta(Z(t)-m)dt,& ~Z(0)=z_0,
\end{array}
\right.
\end{equation}
where $m=0$ (for simplicity), $\beta, \nu >0$ and $\vartheta,\mu$ satisfy the following assumption:

\

\noindent \textbf{Assumption (D1)}: $\vartheta$ is bounded and there exists a constant $C$ such that
$$
|\mu(a)-\mu(b)|\le C |a-b|, \quad |\vartheta(x,a) - \vartheta(y,b)|\le C \big\{|x-y| + |a-b| \big\}
$$
for every $x,y\in \mathbb{R}$, $a,b\in \mathbb{A}$.

\

%\begin{equation}\label{fractionalBM}
%W_H(t):=\int_0^t K_H(t,s)dW(s); 0\le t\le T
%\end{equation}
%where $H\in (0, \frac{1}{2})$, the singular kernel $K_H$ is given by (\ref{originalKH}) and
%\begin{equation}\label{BM}
%W:= \rho B^1 + \bar{\rho} B^2
%\end{equation}

The noise $W_H$ in (\ref{priceprocess}) is a FBM with exponent $H\in (0, \frac{1}{2})$. We assume $W_H$ is a functional of the Brownian motion

$$W:= \rho B^1 + \bar{\rho} B^2,$$
where $\bar{\rho}:=\sqrt{1-\rho^2}$ for $-1 < \rho < 1$. By Theorem 2.2 in \cite{ohashifrancys}, we can write

$$W_H = \rho(\Lambda_H B^1) + \bar{\rho}(\Lambda_HB^2),$$
where $\Lambda_H$ is given by (\ref{LambdaH}). For further details about rough stochastic volatility models, see e.g. \cite{bayer}. By It\^o's formula,

\begin{equation}\label{dade}
X^u(t) =X^u(0) \mathcal{E}\Big(\int \vartheta(Z(s),u(s))dB^1(s)  \Big)(t)\exp\Big(\int_0^t \mu(u(s))ds\Big),
\end{equation}
where $\mathcal{E}$ is the Dol\'eans-Dade exponential. It is well-known (see e.g Prop A1 \cite{cheridito})

%\begin{equation}\label{OUsolution2}
$$Z(t) = e^{-\beta t}z_0 + \nu W_H(t) -\beta \nu e^{-\beta t}\int_0^t W_H(u)e^{\beta u}du; 0\le t \le T.$$
%\end{equation}
Representation (\ref{dade}), Th 1 in \cite{Grigelionis} and Assumption (D1) allow us to state

$$
\sup_{u\in U^T_0}\mathbb{E}\| X^u_T\|^p_\infty < \infty,
$$
for every $p\ge 1$. Moreover, a similar computation as explained in Remark \ref{remarkassB1} gives Assumption (B1) for the controlled state (\ref{priceprocess}). An imbedded discrete structure for the volatility process is given by

\begin{equation}\label{discreteroughOU}
Z^k(t) := e^{-\beta \bar{t}_k}z_0 + \nu W^k_H(t) -\beta \nu e^{-\beta \bar{t}_k}\int_0^{\bar{t}_k} W^k_H(u)e^{\beta \bar{u}_k}du,
\end{equation}
for $0\le t\le T$, where, for each $i=1,2$, we define

\begin{equation}\label{discreterough1}
B^{k,i}_H(t):=  \int_0^{\bar{t}_k} \partial_s K_{H,1}(\bar{t}_k,s)\big[A^{k,i}(\bar{t}_k) - A^{k,i}(\bar{s}^{+}_k)\big]ds - \int_0^{\bar{t}_k}\partial_s K_{H,2}(\bar{t}_k,s)A^{k,i}(s)ds,
\end{equation}
and
$$W^k_H(t):= \rho B^{k,1}_H(t) + \bar{\rho} B^{k,2}_H(t); 0\le t\le T.$$

%\begin{equation}\label{discreterough1}
%W^{k,i}_H(t):=  \int_0^{\bar{t}_k} \partial_s K_{H,1}(\bar{t}_k,s)\big[A^{k,i}(\bar{t}_k) - A^{k,i}(\bar{s}^{+}_k)\big]ds - \int_0^{\bar{t}_k}\partial_s K_{H,2}(\bar{t}_k,s)A^{k,i}(s)ds,
%\end{equation}
%and
%$$W^k_H(t):= \rho W^{k,1}_H(t) + \bar{\rho} W^{k,2}_H(t); 0\le t\le T.$$

Let us fix a control $\phi\in U^T_0$. Starting from $\mathbb{X}^{k,\phi}(0) = x$, we define
\begin{eqnarray}
\nonumber \mathbb{X}^{k,\phi}(T^k_j)&:=&\mathbb{X}^{k,\phi}(T^k_{j-1}) + \int_{T^k_{j-1}}^{T^k_j}\mathbb{X}^{k,\phi}(T^k_{j-1})\mu(\phi(s))ds\\
\label{Seuler}& &\\
\nonumber&+& \mathbb{X}^{k,\phi}(T^k_{j-1})\int_{T^k_{j-1}}^{T^k_j}\vartheta (Z^k(T^k_{j-1}),\phi(s))dB^1(s),
\end{eqnarray}
for $j\ge 1$. It is important to notice that if $\phi = (v^k_0, v^k_1, \ldots) \in U^{k,\infty}_0$, then

\begin{eqnarray}
\label{eufulrv}\mathbb{X}^{k,\phi}(T^k_{j})&=&\mathbb{X}^{k,\phi}(T^k_{j-1}) + \mathbb{X}^{k,\phi}(T^k_{j-1}) \mu(v^k_{j-1})\Delta T^k_j\\
\nonumber& &\\
\nonumber&+& \mathbb{X}^{k,\phi}(T^k_{j-1})\vartheta \Big(Z^k(T^k_{j-1}),v^k_{j-1} \Big)\Delta A^{k,1}(T^k_j),
\end{eqnarray}
for $j\ge 1$.
Then, we set $\mathcal{X} = \big( (X^k)_{k\ge 1},\mathscr{D}\big)$ given by
\begin{equation}\label{xksdeRV}
X^{k,\phi}(t):=\mathbb{X}^{k,\phi}(t\wedge T^k_{e(k,T)}); 0\le t\le T,
\end{equation}
for $\phi \in U^T_0$.
\subsection{The controlled imbedded discrete structure for the value process}\label{CIDSVALUEsection}
In this section, we are going to describe the canonical controlled imbedded discrete structure associated with an arbitrary value process

$$V(t,u) = \esssup_{v\in U_t^T}\mathbb{E}\big[\xi_X(u\otimes_t v)|\mathcal{F}_t\big]; u\in U^T_0, 0\le t\le T,$$
where the payoff $\xi$ satisfies Assumption (A1) and $X$ is an arbitrary controlled Wiener functional satisfying Assumption (B1). Throughout this section, we are going to fix a controlled imbedded discrete structure

\begin{equation}\label{controlledstate}
u^k\mapsto X^{k,u^k}
\end{equation}
and we set

$$
\xi_{X^k}(u^k):=\xi\big(X^{k,u^k}\big),
$$
for $u^k\in U^{k,e(k,T)}_0$. We assume $\big\{\|X^{k,\phi}\|_\infty; \phi \in U^{k,e(k,T)}_0\big\}$ is uniformly integrable for each $k\ge 1$. We then define

\begin{equation}\label{discretevalueprocess}
V^{k}(T^k_n, u^k):=\esssup_{\phi^k\in U^{k,e(k,T)}_n}\mathbb{E}\Big[\xi_{X^k}(u^k\otimes_n\phi^k)\big|\mathcal{F}^k_{T^k_{n}}\Big]; n=1,\ldots, e(k,T)-1,
\end{equation}
with boundary conditions

$$V^k(0):=V^k(0,u^k):=\sup_{\phi^k\in U^{k,e(k,T)}_0}\mathbb{E}\big[\xi_{X^k}(\phi^k)\big],\quad V^{k}(T^k_{e(k,T)}, u^k): = \xi_{X^k}(u^k).$$

This naturally defines the map $V^k:U^{k,e(k,T)}_0\rightarrow O_T(\mathbb{F}^k)$ with jumps $V^k(T^k_n,u^k); n=1, \ldots, e(k,T)$ for $u^k\in U^{k,e(k,T)}_0$. One should notice that $V^k(T^k_n, u^k)$ only depends on $u^{k,n-1}:=(u^k_0, \ldots, u^k_{n-1})$ so it is natural to write

$$V^k(T^k_n, u^{k,n-1}) :=V^k(T^k_n, u^k); u^k\in U^{k,e(k,T)}_0, 0\le n\le e(k,T),$$
with the convention that $u^{k,-1}:=\mathbf{0}$. By construction, $V^k$ satisfies (\ref{antiprop}) in Definition \ref{GASdef}.

Similar to the value process $V$, we can write a dynamic programming principle for $V^k$ where the Brownian filtration is replaced by the discrete-time filtration $\mathcal{F}^k_{T^k_n}; n=e(k,T)-1,\ldots, 0$.

\begin{lemma}\label{lattice2}
Let $0\le n\le e(k,T)-1$. For each $\phi^k$ and $\varphi^k$ in $U^{k,e(k,T)}_n$, there exists $\theta^k\in U^{k,e(k,T)}_n$ such that

$$\mathbb{E}\Big[\xi_{X^k}(\pi^k\otimes_n\theta^k)|\mathcal{F}^k_{T^k_{n}}\Big] =\mathbb{E}\Big[\xi_{X^k}(\pi^k\otimes_n\phi^k)|\mathcal{F}^k_{T^k_{n}}\Big]\vee \mathbb{E}\Big[\xi_{X^k}(\pi^k\otimes_n\varphi^k)|\mathcal{F}^k_{T^k_{n}}\Big]~a.s,$$
for every $\pi^k\in U^{k,n}_0$. Therefore, for each $\pi^k\in U^{k,n}_0$

$$\mathbb{E}\Bigg[\esssup_{\theta^k\in U^{k,e(k,T)}_n}\mathbb{E}\Big[\xi_{X^k}(\pi^k\otimes_n\theta^k)|\mathcal{F}^k_{T^k_{n}}\Big]\Big|\mathcal{F}^k_{T^k_{j}}\Bigg] = \esssup_{\theta^k\in U^{k,e(k,T)}_n}\mathbb{E}\Big[\xi_{X^k}(k,\pi^k\otimes_n\theta^k)|\mathcal{F}^k_{T^k_{j}}\Big]~a.s,$$
for $0\le j\le n$ and $0\le n\le e(k,T)-1$.
\end{lemma}
\begin{proof}
For $\pi^k \in U^{k,n}_0$, let $G=\Big\{\mathbb{E}\Big[\xi_{X^k}(k,\pi^k\otimes_n\phi^k)|\mathcal{F}^k_{T^k_{n}}\Big] > \mathbb{E}\Big[\xi_{X^k}(\pi^k\otimes_n\varphi^k)|\mathcal{F}^k_{T^k_{n}}\Big] \Big\}$. Choose $\theta^k = \phi^k\mathds{1}_{G} + \varphi^k\mathds{1}_{G^c}$ and apply the finite mixing property to exchange the esssup into the conditional expectation (see e.g Prop 1.1.4 in \cite{lamberton}) to conclude the proof.
\end{proof}

\begin{proposition}\label{DPprop}
For each $u^k\in U^{k,e(k,T)}_0$, the discrete-time value process $V^k(\cdot, u^k)$ satisfies

\begin{equation} \label{DPE}
\begin{split}
&V^{k} (T^k_n , u^k) = \esssup_{\theta^{k}_n \in U^{k,n+1}_n}
\mathbb{E} \Bigg[ V^{k} \left(T^k_{n+1}, u^{k,n-1} \otimes_n \theta^k_n \right)    \mid \mathcal{F}^k_{T^k_{n}}\Bigg];~0\le n\le e(k,T)-1\\
&V^{k} (T^k_{e(k,T)} , u^k) = \xi_{X^k} (u^k)~a.s.
\end{split}
\end{equation}
On the other hand, if a
class of processes $\{Z^{k} (T^k_n, u^k); u^k \in U^{k,e(k,T)}_0; 0\le n\le e(k,T)\}$ satisfies the dynamic programming equation
(\ref{DPE}) for every $u^k\in U^{k,e(k,T)}_0$, then $Z^{k} (T^k_n, u^k)$ coincides with $V^{k}(T^k_n , u^k)~a.s$ for every $0\le n\le e(k,T)$ and for every $u^k \in U^{k,e(k,T)}_0$.
\end{proposition}
\begin{proof}
Fix $u^k\in U^{k,e(k,T)}_0$. By using Lemma~\ref{lattice2} and the identity

$$\esssup_{\phi^k\in U^{k,e(k,T)}_n}\mathbb{E}\Big[\xi_{X^k}(u^k\otimes_n\phi^k)|\mathcal{F}^k_{T^k_{n}}\Big] = \esssup_{\theta^k_n\in U^{k,n+1}_n}\esssup_{\phi^k\in U^{k,e(k,T)}_{n+1}}\mathbb{E}\Big[\xi_{X^k}(u^k\otimes_n(\theta^k_n\otimes_{n+1}\phi^k))|\mathcal{F}^k_{T^k_{n}}\Big]$$
a.s for each $0\le n\le e(k,T)-1$, the proof is straightforward, so we omit the details.
\end{proof}

\section{Main Results}\label{mainrsection}
In this section, we describe the main results of this paper. First, we present the dynamic programming algorithm and then, we provide two convergence results under structural conditions on controlled imbedded discrete structures
$\mathcal{X} = \big((X^k)_{k\ge 1},\mathscr{D}\big)$ associated with controlled Wiener functionals.

%The abstract convergence with works under minimal conditions and the result with the rate of convergence for optimal control of drifts.

\subsection{The Dynamic Programming Algorithm}\label{DPPsubsection}

In this section, we describe the dynamic programming algorithm for the discrete version

\begin{equation}\label{discreteOCPROBLEM}
V^k(0) = \sup_{\phi^k\in U^{k,e(k,T)}_0}\mathbb{E}\big[\xi_{X^k}(\phi^k)\big]
\end{equation}
of the nonlinear expectation $\sup_{\phi\in U^T_0}\mathbb{E}\big[\xi_X(\phi)\big]$ and for the corresponding optimal control based on the skeleton $\mathscr{D}$. In order to present the algorithm, it is important to describe the pathwise dynamics generated by $\mathscr{D}$. Let us define

$$\mathbb{I}_k^o:=\Big\{ (i^k_1, \ldots, i^k_d); i^k_\ell\in \{-1,0,1\}~\forall \ell \in \{1,\ldots, d\}~\text{and}~\sum_{j=1}^d|i^k_j|=1   \Big\}$$
and $$\mathbb{I}_k := \Big\{ \epsilon_k \left(i^k_1 \mathds{1}_{\{\mid i^k_1 \mid =1\}} + z^k_1 \mathds{1}_{\{\mid i^k_1 \mid \neq 1\}} , \ldots , i^k_d \mathds{1}_{\{\mid i^k_d \mid =1\}} +z^k_d \mathds{1}_{\{\mid i^k_d \mid \neq 1\}} \right) ; (i_1^k , \ldots , i^k_d) \in \mathbb{I}_k^o,$$
$$(z^k_1, \ldots , z^k_d) \in (-1,1)^d \Big\}. $$

For obvious reasons, $\mathbb{W}_k:=(0,+\infty)\times \mathbb{I}_k$ will be called as the \textit{noise space}. The $n$-fold Cartesian product of $\mathbb{W}_k$ is denoted by $\mathbb{W}_k^n$ and a generic element of $\mathbb{W}^n_k$ will be denoted by

$$
(w^k_1, \ldots, w^k_n)\in \mathbb{W}^n_k,
$$
where $w^k_r = (s^k_r,\tilde{i}^k_r)\in \mathbb{W}_k$ for $1\le r\le n$. Let us define

$$\Delta A^k(T^k_n):=(\Delta A^{k,1}(T^k_n), \ldots, \Delta A^{k,d}(T^k_n)),$$
where

%\begin{equation}\label{etaeq}
$$\Delta A^{k,j}(T^k_n) = B^j (T^k_n) - B^j (T^k_{n-1}),$$
%\end{equation}
for $1\le j\le d; n,k\ge 1$.
Let us define

$$
\mathcal{A}^k_n:= \Big(\Delta T^k_1, \Delta A^{k}(T^k_1), \ldots, \Delta T^k_n, \Delta A^{k}(T^k_n)\Big)\in \mathbb{W}^n_k~a.s.
$$
One should notice that $$\widetilde{\mathcal{F}}^k_{T^k_n} = (\mathcal{A}^k_n)^{-1}(\mathcal{B}(\mathbb{W}^n_k)),$$
where $\mathcal{B}(\mathbb{W}^n_k)$ is the Borel $\sigma$-algebra generated by $\mathbb{W}^n_k; n\ge 1$.

%In this article, we treat two types of non-Markovian systems: (Case I) One where the memory is transmitted by non-anticipative coefficients of a path-dependent SDE. (Case II) non-Markovian systems where the memory is extrinsic in the sense that the dependency is generated by the driving noise which is a singular transformation of the state Brownian motion.

We set $\mathbb{H}_k:= \mathbb{W}_k \times \mathbb{R}^n$, where $\mathbb{R}^n$ is the state-space. We denote $\mathbb{H}^{i}_k$ as the $i$-fold Cartesian product of $\mathbb{H}_k$. In the sequel, we fix an initial condition $x_0\in \mathbb{R}^n$. The elements of $\mathbb{H}^{j}_k$ will be denoted by

$$
\textbf{o}^{k}_i := \Big( (w^k_1,x^k_1), \ldots, (w^k_i, x^k_i) \Big).
$$
For convenience, we set $\mathbf{o}^k_0:=(0,0,x_0)$ and $\mathbb{H}_k^0:= \{(0,0,x_0)\}$. We need to introduce the following projections: For $\textbf{o}^{k}_i = \big( (w^k_1,x^k_1), \ldots, (w^k_i, x^k_i) \big)$, we set

$$\pi_2(\textbf{o}^{k}_i) := \big( (s^k_1,x^k_1), \ldots, (s^k_{i}, x^k_i) \big)\quad \pi_3(\textbf{o}^{k}_i) := \big( w^k_1, \ldots, w^k_i\big),$$
for $i\ge 1$.

Let $\gamma_{j}:\mathbb{H}^{j}_k\rightarrow \mathbf{D}_{n,T}$ be the map

\begin{equation}\label{gammafunc}
\gamma_{j} (\textbf{o}^{k}_{j}) = x_0 + \sum_{i=1}^{j} x^k_i 1\!\!1_{ \{t^k_i \leq \cdot \} },
\end{equation}
where $t^k_i:= \sum_{\ell=1}^i s^k_\ell$. For notational convenience, we set $\gamma_0(\mathbf{o}^k_0):=x_0$.

\begin{remark}\label{gammarem}
We observe $\gamma_j(\mathbf{o}^k_j); j=1, \ldots, e(k,T)$ in (\ref{gammafunc}) does not depend on the variables $(\tilde{i}^k_1,\ldots, \tilde{i}^k_j)$ and it is just a function of $(s^k_1, x^k_1, \ldots, s^k_j,x^k_j)$.
\end{remark}

%As a consequence, we define a function $g_j : \mathbb{H}^{k, j} \times \mathbb{A} \rightarrow \mathbb{R}^d$ satisfying

%\[
%g_j (\textbf{o}^{k}_{j-1},a^k_{j-1} , s^k_j , \tilde{i}^k_j) := \alpha \big(t^k_{j-1} ,\gamma_{j-1} (\textbf{o}^{k}_{j-1}),a^k_{j-1}\big) s^k_j + \sigma \big(t^k_{j-1} ,\gamma_{j-1} (\textbf{o}^{k}_{j-1}),a^k_{j-1}\big) \tilde{i}^k_j,
%\] where $t^k_{j-1}= s^k_1 + \cdots + s^k_{j-1}$.

In the sequel, we describe the dynamics of a controlled imbedded discrete structure $\mathcal{X}  = \big( (X^k)_{k\ge 1}, \mathscr{D} \big)$ as a function of the state-space and the noise as follows: Fix a set of Borel functions

$$\Delta x_j : \big(\pi_2\big(\mathbb{H}_k\big)\big)^{j-1} \times \mathbb{A}\times (0,\infty)\times \mathbb{R}^d \rightarrow \mathbb{R}^n,\quad \ell_j : \mathbb{W}^j_k\rightarrow \mathbb{R}^d,$$
for $j=1,\ldots, e(k,T)$, where we set $\big(\pi_2\big(\mathbb{H}_k\big)\big)^{0}:=\{0\}\times \mathbb{R}^n$.

For a given admissible control $u^k = (u^k_0, \ldots, u^k_{e(k,T)-1})\in U^{k,e(k,T)}_0$, we set $\Xi^{k,u^k}_0 := (0,0,x_0)$ and we assume a controlled imbedded discrete structure $\mathcal{X}  = \big( (X^k)_{k\ge 1}, \mathscr{D} \big)$ satisfies the following dynamics:

\

\noindent \textbf{Assumption (E1):}

\begin{equation}\label{feedH}
\Delta X^{k,u^k}(T^k_j) = \Delta x_j \Big(\pi_2(\Xi^{k,u^k}_{j-1}),u^k_{j-1} , \Delta T^k_j , \ell_j(\mathcal{A}^k_j) \Big),
\end{equation}
where

\begin{equation}\label{Xioperator}
\Xi^{k,u^k}_j :=\Big((\Delta T^k_1,\Delta A^k(T^k_1), \Delta X^{k, u^k}(T^k_1 ) ), \ldots, (\Delta T^k_j, \Delta A^k(T^k_j), \Delta X^{k, u^k}(T^k_j) ) \Big),
\end{equation}
for $1\le j\le e(k,T)$.

\

By construction,

\begin{equation}\label{compoID}
\gamma_{j}\big(\pi_2(\Xi^{k,u^k}_{j})\big)_\cdot = X^{k,u^k}(\cdot\wedge T^k_j)=x_0 +  \sum_{j=1}^{j} \Delta X^{k,u^k}(T^k_j) 1\!\!1_{ \{ T^k_j \leq \cdot \} }, \quad j=1, \ldots , e(k,T).
\end{equation}

Next, we illustrate Assumption (E1) with three fundamental examples.

\begin{example}

Let $X^u$ be the $n$-dimensional controlled SDE given by (\ref{pdsdeBM}). In view of (\ref{Xeulerk}), the functions $\Delta x_j$ in the Assumption (E1) are given as follows.

For each $\textbf{o}^{k}_j = \big( (s^k_1,\tilde{i}^k_1,x^k_1), \ldots, (s^k_{j}, \tilde{i}^k_j, x^k_j) \big)$,

%$$g(x_0,a^k_0, \mathbf{b}^k_1):= $$

%$g_j (\textbf{o}^{k}_{j-1},a^k_{j-1} , \mathbf{b}^k_j)=g_j (\textbf{o}^{k}_{j-1},a^k_{j-1} , s^k_j , \tilde{i}^k_j)$ is given by

\begin{eqnarray}\label{deltaBM}
 \Delta x_j \big(\pi_2(\textbf{o}^{k}_{j-1}),a^k_{j-1} , s^k_j,\ell_j(\pi_3(\mathbf{o}^k_j))\big) &=& \alpha \big(t^k_{j-1} ,\gamma_{j-1} (\pi_2(\textbf{o}^{k}_{j-1})),a^k_{j-1}\big) s^k_j\\
\nonumber & +& \sigma \big(t^k_{j-1} ,\gamma_{j-1} (\pi_2(\textbf{o}^{k}_{j-1})),a^k_{j-1}\big) \tilde{i}^k_j,
\end{eqnarray}
where $t^k_{j-1}= s^k_1 + \cdots + s^k_{j-1}$, for $1\le j\le e(k,T)$. In this case,

\begin{equation}\label{ellBM}
\ell_j(\pi_3(\mathbf{o}^k_j)) = \tilde{i}^k_j; 1\le j \le e(k,T).
\end{equation}
We observe $\ell_j(\pi_3(\mathbf{o}^k_j))$ is not path-dependent and the non-Markovian memory comes from the non-anticipative functionals $(\alpha,\sigma)$.

%It is important to notice (see Remark \ref{gammarem}) that  $g_j(\textbf{o}^{k}_{j-1},a^k_{j-1} , s^k_j,\tilde{i}^k_j)$ only depends on the last $j$th-coordinate sign $\tilde{i}^k_j$ of the information vector $(\tilde{i}^k_1, \ldots, \tilde{i}^k_j)$.
\end{example}

In the next examples, the lack of Markov property comes from an SDE non-anticipative functional and a fractional Brownian motion $B_H$. The process $B_H$ is a path-dependent singular transformation of the Brownian motion which incorporates more dependency on the dynamics of a controlled imbedded discrete structure $\mathcal{X}  = \big( (X^k)_{k\ge 1}, \mathscr{D} \big)$. Here, we revisit the two examples treated in sections \ref{FBMsub} and \ref{rvsub} but in the light of Assumption (E1).

For each information set $ \big(s^k_1,\tilde{i}^k_1, s^k_2,\tilde{i}^k_2, \ldots,   \big) \in \mathbb{W}^\infty_k$ and $s\ge 0$, we define

$$
\bar{s}^{+}_k:= \min\{t^{k}_n ; \bar{s}_k < t^{k}_n\}\wedge T~\text{and}~\bar{s}_k:=\max\{t^{k}_n; t^{k}_n\le s\},
$$
$$a^k(s):= \sum_{\ell=1}^\infty \tilde{i}^k_\ell \mathds{1}_{\{t^k_\ell \le s\}},$$
where we recall $t^k_{j-1}= s^k_1 + \cdots + s^k_{j-1}$. For each $\textbf{o}^{k}_j = \big( (s^k_1,\tilde{i}^k_1,x^k_1), \ldots, (s^k_{j}, \tilde{i}^k_j, x^k_j) \big)$, we define

\begin{equation}\label{ahdet}
\mathbb{A}^{k}_H(\pi_3(\mathbf{o}^k_j)):= \int_0^{t^k_j} \partial_s K_{K,1}(t^k_j,s)\Big[ a^k(t^k_j) - a^k(\bar{s}^+_k) \Big]ds - \int_0^{t^k_j}\partial_s K_{H,2}(t^k_j,s)a^k(s)ds,
\end{equation}
where $K_{H,1}$ and $K_{H,2}$ are given in section \ref{FBMsub}.

\begin{example}
Let $X^u$ be the one-dimensional controlled SDE given by (\ref{limsdefbm}).By construction, $\mathbb{A}^k_H(\mathcal{A}^k_j)= B^k_H(T^k_j); j\ge0$. In view of (\ref{fulleFBM}), the functions $\Delta x_j$ in the Assumption (E1) are given by: For each path $\mathbf{o}^k_j= \big( (s^k_1,\tilde{i}^k_1,x^k_1), \ldots, (s^k_{j}, \tilde{i}^k_j, x^k_j) \big)$,

\begin{equation}\label{deltaFBM}
 \Delta x_j (\pi_2(\textbf{o}^{k}_{j-1}),a^k_{j-1} , s^k_j, \ell_j(\pi_3(\mathbf{o}^k_j))) := \alpha \big(t^k_{j-1} ,\gamma_{j-1} (\pi_2(\textbf{o}^{k}_{j-1})),a^k_{j-1}\big) s^k_j + \sigma \ell_j(\pi_3(\mathbf{o}^k_j)) ,
\end{equation}
where
\begin{equation}\label{ellFBM}
\ell_j(\pi_3(\mathbf{o}^k_j)) = \big[\mathbb{A}^k_H(\pi_3(\mathbf{o}^k_j)) - \mathbb{A}^k_H(\pi_3(\mathbf{o}^k_{j-1}))\big];~1\le j\le e(k,T),
\end{equation}
and $\mathbb{A}^k_H(0,0,x_0):=0$. In contrast to (\ref{ellBM}), we observe $\ell_j(\pi_3(\mathbf{o}^k_j))$ in (\ref{ellFBM}) carries the whole path $(s^k_1,\tilde{i}^k_1, \ldots, s^k_j,\tilde{i}^k_j )$.
\end{example}

\begin{example}
Let $X^u$ be the controlled SDE given by (\ref{dade}). In view of (\ref{eufulrv}), the functions $\Delta x_j$ in the Assumption (E1) are given by: For each $\textbf{o}^{k}_j = \big( (s^k_1,\tilde{i}^k_1,x^k_1), \ldots, (s^k_{j}, \tilde{i}^k_j, x^k_j) \big)$,

\begin{eqnarray}
\nonumber \Delta x_j(\pi_2(\textbf{o}^{k}_{j-1}),a^k_{j-1} , s^k_j,\ell_j(\pi_3(\mathbf{o}^k_j))) &=&\gamma_{j-1} (\pi_2(\textbf{o}^{k}_{j-1}))(t^k_{j-1})\mu(a^k_{j-1}) s^k_j\\
\label{deltaRV} &+& \gamma_{j-1} (\pi_2(\textbf{o}^{k}_{j-1}))(t^k_{j-1})\vartheta(\mathbb{Z}^k(\pi_3(\mathbf{o}^k_j)),a^k_{j-1})  \tilde{i}^{k,1}_j,
\end{eqnarray}
where $t^k_{j-1}= s^k_1 + \cdots + s^k_{j-1}$. We set

$$\mathbf{W}^k_H(\pi_3(\mathbf{o}^k_j)):= \rho \mathbb{A}^{k,1}_H(\pi_3(\mathbf{o}^k_j)) + \bar{\rho} \mathbb{A}^{k,2}_H(\pi_3(\mathbf{o}^k_j))$$
and $\mathbf{W}^k_H(\pi_3(\mathbf{o}^k_j),u):= \mathbf{W}^k_H(\pi_3(\mathbf{o}^k_j))$ for $t^k_{j-1}\le u< t^k_j$. The function $\ell_j$ is given by

$$
\ell_j(\pi_3(\mathbf{o}^k_j)) = \big( \mathbb{Z}^k(\pi_3(\mathbf{o}^k_j)), \tilde{i}^{k,1}_j \big),
$$
where

\begin{equation}\label{zdet}
\mathbb{Z}^k(\pi_3(\mathbf{o}^k_j)):=e^{-\beta t^k_j} z_0 + \nu \mathbf{W}^k_H(\pi_3(\mathbf{o}^k_j))  -\beta \nu e^{-\beta t^k_j}\int_0^{t^k_j}\mathbf{W}^k_H(\pi_3(\mathbf{o}^k_j),u)e^{\beta \bar{u}_k}du,
\end{equation}
for $1\le j\le e(k,T)$.
\end{example}

The dynamic programming associated with $V^k(0)$ will be fully based on Proposition \ref{DPprop} but it involves delicate measurability issues. The first step is to \textit{aggregate} the map $u^k\mapsto V^k(\cdot,u^k)$ into a list of upper semi-analytic functions $\mathbb{V}^k_j:\mathbb{H}^{j}_k \rightarrow \overline{\mathbb{R}}$; $j=0, \ldots, e(k,T)$. We do this by using classical theory of analytic sets (see e.g \cite{bertsekas}). The technical proofs of all unexplained statements are given in the Appendix \ref{MSTappendix}.

In the sequel, (\ref{feedH}) and (\ref{Xioperator}) in Assumption (E1) are in force. Let us define

$$
\nu^k:=\text{Law of}~(\Delta T^k_1,\Delta A^k(T^k_1)). k\ge 1.
$$
We start with $\mathbb{V}^{k}_{e(k,T)}(\mathbf{o}^{k}_{e(k,T)}):=\xi(\gamma_{e(k,T)}(\mathbf{o}^{k}_{e(k,T)}))$ (see Remark \ref{gammarem}) and proceed backwards $j=e(k,T)-1, \ldots, 0$,

\begin{eqnarray}\label{DPP}
\mathbf{U}^k_j(\mathbf{o}^k_j,\theta) &:=& \int_{\mathbb{W}_k}\mathbb{V}^{k}_{j+1}\Big(\mathbf{o}^{k}_{j}, \mathfrak{X}^k_{j+1}(\theta, \mathbf{o}^k_{j}, w^k)\Big)\nu^k(dw^k)\\
\nonumber\mathbb{V}^k_j(\mathbf{o}^k_j) &:=& \sup_{\theta\in \mathbb{A}}\mathbf{U}^k_j(\mathbf{o}^k_j,\theta).
\end{eqnarray}
Here, the maps $\mathfrak{X}^k_{j+1}: \mathbb{A}\times \mathbb{H}_k^j\times \mathbb{W}_k\rightarrow \mathbb{W}_k\times \mathbb{R}^n$

$$
\mathfrak{X}^k_{j+1}(\theta, \mathbf{o}^k_{j}, w^k):=\Big(w^k, \Delta x_{j+1}\big(\pi_2(\textbf{o}^{k}_j),\theta, s^k, \ell_{j+1}(\pi_3(\mathbf{o}^k_j),w^k)\big)\Big)
$$
will be called the \textit{transition kernels} associated with $\mathcal{X}$. The function
\begin{eqnarray*}\label{osavf}
\mathbf{o}^k_j \mapsto \mathbf{U}^k_j(\mathbf{o}^k_j,\theta) &=& \mathbb{E}\Big[ \mathbb{V}^{k}_{j+1}\Big(\mathbf{o}^{k}_{j}, \mathfrak{X}^k_{j+1}(\theta, \mathbf{o}^k_{j}, \Delta T^k_{j+1},\Delta A^k(T^k_{j+1}) )\Big) \Big]
\end{eqnarray*}
is called the \textit{optimal state-action value function} at step $j$ and $\mathbb{V}^k$ is the value functional. One can show (see Propositions \ref{detarg} and \ref{AGREGATION}),

$$
\mathbb{V}^{k}_{j}(\Xi^{k,u^k}_{j}) = V^k(T^k_{j}, u^k) ~a.s; \quad j=e(k,T),\ldots,0,
$$
for each control $u^k\in U^{k,e(k,T)}_0$. Moreover, for every $\epsilon>0$, there exists a universally  measurable function $C^\epsilon_{k,j}:\mathbb{H}^{j}_k\rightarrow\mathbb{A}$ such that

\begin{equation}\label{opcontrol}
\mathbb{V}^{k}_j(\mathbf{o}^k_{j})\le \int_{\mathbb{W}_k}\mathbb{V}^{k}_{j+1}\Big(\mathbf{o}^{k}_{j}, \mathfrak{X}^k_{j+1}(C^{\epsilon}_{k,j}(\mathbf{o}^k_j), \mathbf{o}^k_{j}, w^k)\Big)\nu^k(dw^k) + \epsilon,
\end{equation}
for every $\mathbf{o}^k_{j}\in \{\mathbb{V}^{k}_j < + \infty\}$, where $j=e(k,T)-1,\ldots, 0$. In particular, if $\mathbb{H}^{j}_k = \{\mathbb{V}^k_j < +\infty\}$, for $j=e(k,T)-1, \ldots, 0$, then for every $\epsilon>0$ and $u^k\in U^{k,e(k,T)}_0$, the control $u^{k,\epsilon}_j$ defined by

\begin{equation}\label{explicitcontrolTEXT}
u^{k,\epsilon}_j = C^\epsilon_{k,j}(\Xi^{k,u^k}_j); j=e(k,T)-1, \ldots, 0
\end{equation}
realizes
$$
V^k(T^k_j,u^k)\le \mathbb{E}\big[V^k(T^k_{j+1}, u^k\otimes_{j}u^{k,\epsilon}_j)|\mathcal{F}^k_{T^k_j}  \big] + \epsilon~a.s,
$$
for every $j=e(k,T)-1, \ldots, 0$. The above construction suggests the following definition.

\begin{definition}\label{admissiblepairdiscrete}
A pair $(\xi,\mathcal{X})$ is \textbf{admissible} w.r.t the control problem (\ref{discreteOCPROBLEM}) if $\mathbb{H}^{j}_k = \{\mathbb{V}^k_j < +\infty\}$ for every $j=e(k,T)-1, \ldots, 0$ and $k\ge 1$.
\end{definition}

From dynamic programming principle (see Proposition \ref{epsiloncTH}), for any admissible pair $(\xi,\mathcal{X})$  w.r.t the control problem (\ref{discreteOCPROBLEM}), the control

\begin{equation}\label{fullexplicitcontrolTEXT1}
\phi^{*,k} := (\phi^{k,\eta_k(\epsilon)}_0, \phi^{k,\eta_k(\epsilon)}_1, \ldots, \phi^{k,\eta_k(\epsilon)}_{e(k,T)-1})
\end{equation}
constructed via (\ref{explicitcontrolTEXT}) (with $\eta_k(\epsilon)=\frac{\epsilon}{e(k,T)}$) realizes

$$
\sup_{u^k\in U^{k,e(k,T)}_0}\mathbb{E}\big[\xi_{X^k}(u^k)\big]\le \mathbb{E}\big[\xi_{X^k}(\phi^{*,k})\big]+ \epsilon.
$$

Next, we discuss the impact of the type of memory has in the obtention of (\ref{fullexplicitcontrolTEXT1}). When the lack of Markov property is intrinsic to the system, then the value functional is only a function of the time and the controlled state. This is not the case if the memory is generated by an additional non-Markovian extrinsic noise. This is summarized by the following result.

\begin{proposition}\label{openclosed}
Let $X$ be the controlled path-dependent SDE (\ref{pdsdeBM}) and let $\mathcal{X}$ be the associated Euler scheme given by (\ref{xksde}) restricted to $U^{k,e(k,T)}_0$. Then, the class of admissible controls arising from the dynamic programming equation (\ref{DPP}) based on $\mathcal{X}$ is constituted by feedback controls, i.e., they are adapted w.r.t the filtration generated by $X^{k,\phi^{k}}$ for each $\phi^k \in U^{k,e(k,T)}_0$. Let $X$ be the controlled SDEs (\ref{limsdefbm}) and (\ref{priceprocess}) with their Euler schemes $\mathcal{X}$ restricted to $U^{k,e(k,T)}_0$ given by (\ref{xksdeFBM}) and (\ref{xksdeRV}), respectively. Then, the class of admissible controls arising from (\ref{DPP}) based on $\mathcal{X}$ is constituted by open-loop controls.
\end{proposition}
\begin{proof}
To shorten notation, we set $m=e(k,T)$. In the case of the controlled SDE (\ref{pdsdeBM}), the function (see (\ref{ellBM})) $\ell_j(\pi_3(\mathbf{o}^k_j)) = \tilde{i}^k_j$ does not depend on the past $\big( (s^k_1,\tilde{i}^k_1),\ldots, (s^k_{j-1},\tilde{i}^k_{j-1}) \big)$ of the noise $\pi_3(\mathbf{o}^k_{j})$. Therefore, by the very definition of (\ref{deltaBM}), the domain of the value functional $\mathbb{V}^k_j; j=m,\ldots, 1$ can be reduced to the $j$-fold Cartesian product of $(0,\infty)\times \mathbb{R}^n$ and $\mathbb{V}^k_j(\mathbf{o}^k_j)$ becomes only a function of $\pi_2(\mathbf{o}^k_j) = \big(s^k_1,x^k_1,\ldots, s^k_j,x^k_j\big)$ which does not contain the variable $\tilde{i}^k_1, \ldots, \tilde{i}^k_{j}$. In this case, the complexity of the dynamic programming algorithm can be drastically reduced to

\begin{eqnarray}\label{feedDP}
\mathbf{U}^k_j(\pi_2(\mathbf{o}^k_j),\theta) &=& \int_{\mathbb{W}_k}\mathbb{V}^{k}_{j+1}\Big(\pi_2(\mathbf{o}^{k}_{j}), \mathfrak{X}^k_{j+1}(\theta,\pi_2( \mathbf{o}^k_{j}), w^k)\Big)\nu^k(dw^k)\\
\nonumber\mathbb{V}^k_j(\pi_2(\mathbf{o}^k_j)) &=& \sup_{\theta\in \mathbb{A}}\mathbf{U}^k_j(\pi_2(\mathbf{o}^k_j),\theta).
\end{eqnarray}
Here,

\begin{equation}\label{jumpX}
\mathfrak{X}^k_{j+1}(\theta, \pi_2(\mathbf{o}^k_{j}), w^k)=\Big(s^k,\Delta x_{j+1}(\pi_2(\textbf{o}^{k}_j),\theta, w^k))\Big)
\end{equation}

\[
 \Delta x_{j+1} \big(\pi_2(\textbf{o}^{k}_{j}),\theta , s^k,\tilde{i}^k\big) = \alpha \big(t^k_{j} ,\gamma_{j} (\pi_2(\textbf{o}^{k}_{j})),\theta \big) s^k + \sigma \big(t^k_{j} ,\gamma_{j} (\pi_2(\textbf{o}^{k}_{j})),\theta\big) \tilde{i}^k,
\]
for $w^k = (s^k,\tilde{i}^k)$, $\theta \in \mathbb{A}$, $\mathbf{o}^k_j \in \mathbb{H}_k^j$, where $j=m-1,\ldots, 0$. Hence, the admissible controls in $U^{k,m}_0$ arising from (\ref{DPP}) are actually feedback, i.e., they are adapted w.r.t the filtration generated by $X^{k,\phi^{k}}$ for each $\phi^k \in U^{k,m}_0$. For the controlled SDEs (\ref{limsdefbm}) and (\ref{priceprocess}), we have, respectively,

$$\ell_j(\pi_3(\mathbf{o}^k_j)) = \big[\mathbb{A}^k_H(\pi_3(\mathbf{o}^k_j)) - \mathbb{A}^k_H(\pi_3(\mathbf{o}^k_{j-1}))\big]
$$
and

$$\ell_j(\pi_3(\mathbf{o}^k_j)) = \big( \mathbb{Z}^k(\pi_3(\mathbf{o}^k_j)), \tilde{i}^{k,1}_j \big),$$
for $1\le j\le m$, where $\mathbb{A}^k_H$ and $\mathbb{Z}^k$ follow (\ref{ahdet}) and (\ref{zdet}), respectively. In those cases, at the first step $m=e(k,T)$ (in the backward procedure (\ref{DPP})), the value function $\mathbb{V}^k_{m-1}(\mathbf{o}^k_{m-1})$ is a function of $(s^k_1,\tilde{i}^k_1,x^k_1,\ldots, s^k_{m-1},\tilde{i}^k_{m-1},x^k_{m-1})$. Therefore, $\mathbb{V}^k_{j-1}$ becomes fully dependent on $(s^k_1,\tilde{i}^k_1,\ldots, s^k_{j-1},\tilde{i}^k_{j-1})$ for $j=1,\ldots, m$. Therefore, the admissible controls arising from (\ref{DPP}) are open-loop.
\end{proof}

\begin{remark}
In case $(\xi, \alpha,\sigma)$ are state-dependent (not path-dependent) functionals in the controlled SDE (\ref{pdsdeBM}), then the value functionals reduce to $\mathbb{V}^k_j(t^k_j, x^k_j)$, where $x^k_j$ represents the controlled state-variable at the $j$-th step for $j=e(k,T)-1, \ldots, 0$.
\end{remark}

\begin{remark}\label{leaorem} If the map $\mathbb{V}^{k}_{e(k,T)}:\mathbb{H}^{e(k,T)}_k\rightarrow\mathbb{R}$ is upper semicontinuous, then we can apply the Borel  measurable selection theorem (see, Prop 7.33 in \cite{bertsekas}) to conclude that the function $\mathbb{V}^{k}_j:\mathbb{H}^{j}_k\rightarrow\overline{\mathbb{R}}$ is also upper semicontinuous for each $j=e(k,T)-1,\ldots, 1$. Moreover, there exists a Borel measurable function $C^\star_{k,j}:\mathbb{H}^{j}_k\rightarrow\mathbb{A}$ such that

$$
\mathbb{V}^{k}_j(\mathbf{o}^k_{j}) = \int_{\mathbb{W}_k}\mathbb{V}^{k}_{j+1}\Big(\mathbf{o}^{k}_{j}, \mathfrak{X}^k_{j+1}(C^{\star}_{k,j}(\mathbf{o}^k_j), \mathbf{o}^k_{j}, w^k)\Big)\nu^k(dw^k)
$$
for every $\mathbf{o}^k_{j}\in \{\mathbb{V}^{k}_j < + \infty\}$, where $j=e(k,T)-1,\ldots, 1$. If $\mathbb{H}^{j}_k = \{\mathbb{V}^k_j < +\infty\}; j=e(k,T)-1, \ldots, 0$ then, for a given $u^k\in U^{k,e(k,T)}_0$, the list of controls $u^{\star,k}_j$ defined by

\begin{equation}\label{additionalH}
u^{\star,k}_j = C^\star_{k,j}(\Xi^{k,u^k}_j); j=e(k,T)-1, \ldots, 0,
\end{equation}
realizes

$$
V^k(T^k_j,u^k) = \mathbb{E}\big[V^k(T^k_{j+1}, u^k\otimes_{j}u^{\star,k}_j)|\mathcal{F}^k_{T^k_j}  \big] ~a.s,
$$
for every $j=e(k,T)-1, \ldots, 0$. In particular, if $\mathbb{V}^k_{e(k,T)}$ is continuous, then $C^{\star}_{k,j}:\mathbb{H}^{j}_k\rightarrow \mathbb{A}$ given by

\begin{equation}\label{explicitCk}
C^{\star}_{k,j}(\mathbf{o}^k_j) = \max \Big\{  a\in \mathbb{A}; \mathbb{V}^k_j(\mathbf{o}^k_j) = \int_{\mathbb{W}_k}\mathbb{V}^{k}_{j+1}\Big(\mathbf{o}^{k}_{j}, \mathfrak{X}^k_{j+1}(C^{\epsilon}_{k,j}(\mathbf{o}^k_j), \mathbf{o}^k_{j}, w^k)\Big)\nu^k(dw^k) \Big\}
\end{equation}
is a Borel measurable function, where the maximum in the above right-hand side is taken w.r.t the lexicographical order. See e.g \cite{dolinsky}.
Under these conditions, the control

$$
u^{\star,k}= \big(u^{\star,k}_0, \ldots, u^{\star,k}_{m-1}   \big)
$$
constructed via (\ref{additionalH}) and (\ref{explicitCk}) realizes

$$V^k(0)=\sup_{u^k\in U^{k,e(k,T)}_0}\mathbb{E}\big[\xi_{X^k}(u^k)\big]= \mathbb{E}\big[\xi_{X^k}(u^{\star,k})\big].$$

\end{remark}

\subsection{Convergence of optimal values}\label{CONVERGENCEsection}
Throughout this section, we assume that $(X,\mathcal{X})$ is a strongly controlled Wiener functional. In this section, we aim to prove the controlled imbedded discrete structure $\mathcal{V} = \big((V^k)_{k\ge 1},\mathscr{D}\big)$ given by (\ref{discretevalueprocess}) yields a solution of the control problem (\ref{optimaldef1}) in Definition \ref{epsilonDEF} for $k\ge 1$ large enough.

\begin{theorem}\label{absrate1}
Let $(X,\mathcal{X})$ be a strongly controlled Wiener functional satisfying Assumption (B1) and let $(h_k)_{k\ge 1}$ be the rate associated with (\ref{keyassepsilon}). We assume $(\xi,\mathcal{X})$ is admissible, where the payoff $\xi$ satisfies Assumption (A1), with $\gamma$-H\"{o}lder regularity. Let $\big((V^k)_{k\ge 1},\mathscr{D}\big)$ be the value process (\ref{discretevalueprocess}) associated with $\mathcal{X}=\big((X^k)_{k\ge 1},\mathscr{D}\big)$, where $\Delta X^{k}$ satisfies Assumption (E1). Let $\mathfrak{X}^k_{j+1}: \mathbb{A}\times \mathbb{H}^{j}_{k}\times \mathbb{W}_k\rightarrow \mathbb{R}^n; j=e(k,T)-1, \ldots, 0$ be the transition kernel associated with $\mathcal{X}$. Then, for each $n= e(k,T)-1,\ldots, 0$, we have

\begin{equation} \label{DPEexp}
\begin{split}
&V^{k} (T^k_n , u^k) = \sup_{\theta\in \mathbb{A}}
\int_{\mathbb{W}_k} \mathbb{V}^{k}_{n+1}\big(\Xi^{k,u^k}_n, \mathfrak{X}^k_{n+1}(\theta, \Xi^{k,u^k}_n, w^k)\big) \nu^k(dw^k)\\
&V^{k} (T^k_{e(k,T)} , u^k) = \xi_{X^k} (u^k)~a.s,
\end{split}
\end{equation}
for each control $u^k\in U^{k,e(k,T)}_0$. Moreover, for a given $\epsilon >0$ and $\beta \in (0,1)$, there exists a constant $C$ which depends on $\beta, \|\xi\|_\gamma$ and Assumption (B1) and there exists a constant $L_\epsilon$ which depends on $\epsilon>0$ such that

\begin{equation}\label{quasirate}
\Big|\sup_{\phi\in U^T_0}\mathbb{E}\big[\xi_X(\phi)\big] - V_k(0)\Big|\le C \big\{ h^\gamma_k + |L_\epsilon|^\gamma\epsilon_k^{\gamma \beta}\big\} + \epsilon,
\end{equation}
for every $k\ge 1$ sufficiently large. For a given $\epsilon>0$ and $k\ge 1$, let $\phi^{*,k,\epsilon} \in U^{k,e(k,T)}_0$ be a near optimal control associated with the control problem

$$
\mathbb{E}\big[ \xi_{X^k}\big(\phi^{*,k,\epsilon}\big)\big] > V_k(0) -\frac{\epsilon}{3};~k\ge 1.
$$
Then, $\phi^{*,k,\epsilon}\in U_0^T$ is a near optimal control for the Brownian motion driving stochastic control problem, i.e.,

\begin{equation}\label{nearOC}
\mathbb{E}\big[ \xi_X\big(\phi^{*,k,\epsilon}\big)\big] > \sup_{\phi\in U^T_0}\mathbb{E}\big[\xi_X(\phi)\big] -\epsilon,
\end{equation}
for every $k$ sufficiently large.

\end{theorem}

\begin{remark}\label{densityrem}
By Assumptions (A1-B1),

$$
\sup_{\phi\in U^T_0}\mathbb{E}\big[\xi_X(\phi)\big] =  \sup_{\phi\in \mathcal{R}}\mathbb{E}\big[\xi_X(\phi)\big],
$$
for any dense subset $\mathcal{R}$ of $U^T_0$ w.r.t $L^2_a(\mathbb{P}\times Leb)$-topology.
\end{remark}

We define $W_T:=\{(t,\omega); t\in [0,T]; \omega\in \Omega\}$ endowed with the metric

$$d_{1,\infty}((t,\omega); (t',\omega')): = \mathbf{d}_\infty(w,w') + |t-t'|,$$
where $\mathbf{d}_{\infty}$ is the metric of the compact-open topology on $\Omega$.
%$$d_{1/2}((t,\omega); (t',\omega')): = \sup_{0\le u\le T}\|\omega(u\wedge t) - \omega'(u\wedge t')\|_{\mathbb{R}^n} + |t-t'|^{1/2}.$$
Then, $(W_T,d_{1,\infty})$ is a separable complete metric space equipped with the Borel $\sigma$-algebra.

Let $\pi:W_T\rightarrow W_T$ given by

$$\pi(t.w):=(t,w_{t}),$$
for $(t,w) \in W_T$. Any $\mathbb{F}$-progressively measurable process $\phi \in U^T_0$ can be written as

\begin{equation}\label{Frep}
\phi(t,\omega) = F(\pi(t,\omega)),
\end{equation}
where $F:W_T\rightarrow \mathbb{A}$ is Borel.

Next, we define $U^T_0(Lip):=\{F\circ \pi; \|F\|_{Lip} < \infty\}$, where $F:W_T\rightarrow \mathbb{A}$ satisfies

$$\|F\|_{Lip}:=\sup_{(t,w),(t',w') \in W_T}\frac{|F(t,w) - F(t',w')|}{d_{1,\infty} \big((t,w); (t',w')\big)}< \infty.$$

We recall two important results: Lusin and McShane's theorem which we recall here for convenience of the reader.

\begin{lemma}\label{lusin}
Let $X$ be a Polish space, $E$ a topological space with a countable basis (e.g., $E=\mathbb{A})$, $\mu$ a Borel probability measure on $X$, and $f : X \rightarrow E$ a Borel map. Then for each $\epsilon  >0$ there is a compact set $K \subset X$ such that $\mu(X - K) \le \epsilon$ and $f|_K$ is continuous.
\end{lemma}

\begin{lemma}\label{mcshane}
Let $(X,d)$ be a metric space, $(S,d)$ a subset of $X$ and $f:S\rightarrow \mathbb{R}^r$ a bounded $L$-Lipschitz function. Then, it admits an $L$-Lipschitz extension $\tilde{f}:X\rightarrow \mathbb{R}^r$ such that $\sup_{x \in X}|\tilde{f}(x)| = \sup_{x \in S}|f(x)|$.
\end{lemma}

\begin{lemma}\label{Liplemma}
The set $U^T_0(Lip)$ is an $L^2$-dense subset of $U^T_0$.
\end{lemma}
\begin{proof}
We apply Lusin's theorem as follows. Recall $W_T=[0,T]\times \Omega$ is Polish and let us consider the finite positive Borel measure $\mu = (Leb \times \mathbb{P}) \circ \pi$ on $W_T$. Let $\phi \in U^T_0$ be a control represented by $\phi  = F \circ \pi$, where $F:W_T\rightarrow \mathbb{A}$ is a Borel map. For a given $\epsilon>0$, we invoke Lemma \ref{lusin} to get a compact set $K_\epsilon \subset W_T$ such that $\mu (W_T \setminus K_\epsilon) < \epsilon$ and $F|_{K_\epsilon}$ is continuous (hence uniformly continuous). Let $F^\epsilon:K_\epsilon\rightarrow \mathbb{A}$ be a Lipschitz function (see item (3) of Th 6.4.1 in \cite{lipref}) such that

\begin{equation}\label{apl}
\sup_{(t,w)\in K_\epsilon}|F(t,w) - F^\epsilon(t,w)| < \epsilon.
\end{equation}
By Lemma \ref{mcshane}, we can select an extension $G^\epsilon:W_T\rightarrow \mathbb{R}^r$ such that $G^\epsilon|_{K_\epsilon} = F^\epsilon$, $\sup_{x \in W_T}|G^\epsilon(x)|= \sup_{x\in K_\epsilon}|F^\epsilon(x)|\le \bar{a}$ and $G^\epsilon$ has the same Lipschitz constant of $F^\epsilon$. We set $\phi_\epsilon = G^\epsilon \circ \pi$. We observe

\begin{eqnarray*}
\mathbb{E}\int_0^T|\phi(t) - \phi_\epsilon(t)|^2dt  &=& \int_{W_T} |F-G^\epsilon|^2d\mu = \int_{K_\epsilon} |F-F^\epsilon|^2d\mu\\
&+& \int_{W_T\setminus K_\epsilon} |F-G^\epsilon|^2d\mu\\
&\le& \epsilon^2 T + \bar{a}^2 \epsilon.
\end{eqnarray*}
This completes the proof of Lemma \ref{Liplemma}.
\end{proof}

\noindent \textbf{Proof of Theorem \ref{absrate1}:} The proof of (\ref{DPEexp}) is given in Appendix \ref{MSTappendix}, Propositions \ref{detarg} and \ref{AGREGATION}. Here, we will prove (\ref{quasirate}) and (\ref{nearOC}). In the sequel, $C$ is a generic positive constant which may differ from line to line. Let $D(\mathbb{R}_+;\mathbb{R}^d)$ be the space of c\`adl\`ag functions equipped with the Skorohod topology $\mathbf{d}$. We define $\overline{W}_T:=\{(t,\omega); t\in [0,T]; \omega\in D(\mathbb{R}_+;\mathbb{R}^d)\}$ and we endow this set with the metric
%$$d_{1/2}((t,\omega); (t',\omega')): = \sup_{0\le u\le T}\|\omega(u\wedge t) - \omega'(u\wedge t')\|_{\mathbb{R}^n} + |t-t'|^{1/2}.$$

$$d_{1,s}((t,\omega); (t',\omega')): = \mathbf{d}(w,w') + |t-t'|.$$
Then, $(\overline{W}_T,d_{1,s})$ is a separable complete metric space equipped with the Borel $\sigma$-algebra.
Recall (see e.g item 3, Th 15.12 in \cite{he}) that the Skorohod topology $\mathbf{d}$ is equivalent to
the compact-open topology $\mathbf{d}_\infty$ restricted to the space $\Omega$. Therefore, we can view $\big(W_T, d_{1,\infty}\big) = \big(W_T, d_{1,s}\big)\subset \big(\overline{W}_T, d_{1,s}\big)$. Moreover, we can extend $\pi$ to $\overline{W}_T$ be setting $\pi(t,f): = (t,f_t) \in \overline{W}_T$ for each $(t,f) \in \overline{W}_T$. Having said that, for a given $\epsilon >0$, we choose $\phi_{\epsilon} = F_\epsilon \circ \pi \in U^T_0(Lip)$ with Lipschitz constant $L_\epsilon$ such that

\begin{equation}\label{RLip}
V(0) \le \mathbb{E}[\xi_X(\phi_\epsilon)] + \epsilon.
\end{equation}
By using Lemma \ref{mcshane}, we can select $\overline{F}_\epsilon: \overline{W}_T\rightarrow \mathbb{A}$ such that $\overline{F}_\epsilon|_{W_T} = F_\epsilon$ and with the same Lipschitz constant $L_\epsilon$ of $F_\epsilon$. We then define the control in $U^{k,e(k,T)}_0$ as follows
%We then define$$\bar{\phi}_\epsilon:= \overline{F}_\epsilon\circ \pi$$

$$\phi^k_\epsilon(t)  := \sum_{n=0}^\infty (\overline{F}_\epsilon\circ \pi)(T^k_n, A^k_T)\mathds{1}_{\{T^k_n < t\wedge T^k_{e(k,T)}\le T^k_{n+1}\}}; 0\le t\le T.$$

Because $d_{1,s}$ is weaker than $d_{1,\infty}$ on $D(\mathbb{R}_+;\mathbb{R}^d)$ (see e.g Th 15.12 in \cite{he}), we have

$$
\sup_{0\le t\le T}|\phi^k_\epsilon(t) - \phi_\epsilon(t)|^2\le C |L_\epsilon|^2 \Big\{ |\max_{n\ge 0}\Delta T^k_{n+1}|^2\mathds{1}_{\{T^k_n \le T\}} + |T-T^k_{e(k,T)}|^2\mathds{1}_{\{T^k_{e(k,T)}< T\}} + \epsilon_k^2\Big\},
$$
for every $k \ge 1$. Therefore, by using Lemma \ref{Tkrate} in Appendix \ref{LDappendix}, for a given $0 < \beta < 1$ and $0 < \gamma \le 1$, we get

\begin{eqnarray*}
\Bigg(\mathbb{E}\int_0^T|\phi^k_\epsilon(t) - \phi_\epsilon(t)|^2dt\Bigg)^{\frac{\gamma}{2}}&\le& C |L_\epsilon|^\gamma \Big\{\epsilon_k^{2\gamma}\big (1 + \epsilon_k^{-2}\big)^{\frac{\gamma(1-\beta)}{2}} + \epsilon^{\gamma \beta }_k \Big\}\\
&\le& C |L_\epsilon|^\gamma \epsilon_k^{\gamma \beta},
\end{eqnarray*}
for $k\ge 1$ sufficiently large. By Assumptions (A1-B1) and the fact that $\big(X,\mathcal{X}\big)$ is a strongly controlled Wiener functional, from (\ref{RLip}), we observe

\begin{eqnarray*}
V(0) &\le&\mathbb{E}[\xi_X(\phi_\epsilon)] + \epsilon\\
 &=& \mathbb{E}[\xi_X(\phi_\epsilon)] + \mathbb{E}[\xi_X(\phi^k_\epsilon)] - \mathbb{E}[\xi_X(\phi^k_\epsilon)] + \epsilon\\
&\le& \|\xi\|_\gamma \mathbb{E}\|X^{\phi_\epsilon} - X^{\phi^k_\epsilon}\|^\gamma_\infty + \mathbb{E}[\xi_X(\phi^k_\epsilon)] + \mathbb{E}[\xi_{X^k}(\phi^k_\epsilon)] - \mathbb{E}[\xi_{X^k}(\phi^k_\epsilon)] + \epsilon\\
&\le& \|\xi\|_\gamma \big(\mathbb{E}\|X^{\phi_\epsilon} - X^{\phi^k_\epsilon}\|^2_\infty\big)^{\frac{\gamma}{2}} +
\|\xi\|_\gamma \mathbb{E}\|X^{k,\phi^k_\epsilon} - X^{\phi^k_\epsilon}\|_\infty^\gamma + V_k(0) + \epsilon\\
&\le& C\big\{|L_\epsilon|^\gamma\epsilon_k^{\gamma \beta} + h^\gamma_k\}  + V_k(0) + \epsilon,
\end{eqnarray*}
for every $k\ge 1$ sufficiently large. Again, by Assumptions (A1-B1) and the fact that $\big(X,\mathcal{X}\big)$ is a strongly controlled Wiener functional, we observe

\begin{eqnarray*}
\Big|V_k(0) - \sup_{\phi \in U^{k,e(k,T)}_0}\mathbb{E}\big[ \xi_{X}(\phi) \big]\Big|&=&\Big|\sup_{\phi \in U^{k,e(k,T)}_0}\mathbb{E}\big[ \xi_{X^k}(\phi) \big] - \sup_{\phi \in U^{k,e(k,T)}_0}\mathbb{E}\big[ \xi_{X}(\phi) \big]\Big|\\
&\le& \sup_{\phi \in U^{k,e(k,T)}_0} \Big|\mathbb{E}\big[ \xi_{X^k}(\phi) \big] - \mathbb{E}\big[ \xi_{X}(\phi) \big]\Big|\\
&\le& \|\xi\|_\gamma \sup_{\phi \in U^{k,e(k,T)}_0} \mathbb{E}\|X^\phi - X^{k,\phi}\|^\gamma_\infty\\
&\le& C h^\gamma_k,
\end{eqnarray*}
for $k\ge 1$. Hence,

\begin{eqnarray}
\nonumber V_k(0)&=& \sup_{\phi \in U^{k,e(k,T)}_0}\mathbb{E}\big[ \xi_{X^k}(\phi) \big] - \sup_{\phi \in U^{k,e(k,T)}_0}\mathbb{E}\big[ \xi_{X}(\phi) \big] + \sup_{\phi \in U^{k,e(k,T)}_0}\mathbb{E}\big[ \xi_{X}(\phi) \big] \\
\label{oth}&\le& C h^\gamma_k + V(0),
\end{eqnarray}
for $k\ge 1$. Therefore, for a given $\epsilon >0$ and $\beta \in (0,1)$, we have

$$V_k(0) - V(0)\le C h^\gamma_k\quad \text{and}\quad V(0) - V_k(0)\le C\big\{ |L_\epsilon|^\gamma \epsilon_k^{\gamma \beta} + h^\gamma_k\big\} + \epsilon, $$
for every $k\ge 1$ sufficiently large. This shows that for every $\epsilon>0$ and $\beta \in (0,1)$, there exists a constant $C$ which depends on $\beta$, $\|\xi\|_\gamma$ and Assumption (B1) such that

$$|V(0) - V_k(0)|\le C \big\{ h^\gamma_k + |L_\epsilon|^\gamma\epsilon_k^{\gamma \beta}\big\} + \epsilon,$$
for every $k\ge 1$ sufficiently large. Now, for a given $\epsilon >0$, we have

$$|V_k(0) - V(0)| < \frac{\epsilon}{3},$$
for every $k\ge 1$ sufficiently large. We select $\phi^{*,k,\epsilon} \in U^{k,e(k,T)}_0$ such that

$$\mathbb{E}\big[ \xi_{X^k}(\phi^{*,k,\epsilon})\big]\ge V_k(0) -\frac{\epsilon}{3};~k\ge 1.$$
Since $\big(X,\mathcal{X}\big)$ is a strongly controlled Wiener functional, we also know that

$$\Bigg|\mathbb{E}\big[ \xi_{X^k}(\phi^{*,k,\epsilon})\big] - \mathbb{E}\big[ \xi_{X}(\phi^{*,k,\epsilon})\big]\Bigg| < \frac{\epsilon}{3},$$
for every $k\ge 1$ sufficiently large. Therefore,

$$\mathbb{E}\big[ \xi_{X}(\phi^{*,k,\epsilon})\big] + \frac{\epsilon}{3} > \mathbb{E}\big[ \xi_{X^k}(\phi^{*,k,\epsilon})\big] \ge V_k(0)-\frac{\epsilon}{3}> V(0) - \frac{2\epsilon}{3},$$
for every $k\ge 1$ sufficiently large. This concludes the proof.

\

\begin{remark}\label{Liprem}
The rate obtained in (\ref{quasirate}) is a function of the Lipschitz constant $L_\epsilon$ of a near $\epsilon$-optimal Lipschitz control in $U^T_0(Lip)$. See (\ref{RLip}). If there exists a Lipschitz optimal control $\phi^{*} \in U^T_0(Lip)$, i.e.,

$$\mathbb{E}\big[ \xi_{X}(\phi^{*})\big] = V(0),$$
then the rate becomes
$$
\Big|\sup_{\phi\in U^T_0}\mathbb{E}\big[\xi_X(\phi)\big] - V_k(0)\Big|\lesssim h^\gamma_k + \epsilon_k^{\gamma \beta},
$$
for every $k\ge 1$ sufficiently large. In general, none Lipschitz optimal control exists and the rate will depend on the growth of $L_\epsilon$ as $\epsilon\downarrow 0$. For a given $\beta \in (0,1)$ and $\gamma \in (0,1]$, the natural question is under what condition on $ \phi \mapsto \mathbb{E}[\xi_X(\phi)]$, there exists $\delta \in (0,1]$ such that

$$\lim_{k\rightarrow +\infty}|L_{\epsilon^\delta_k}|^\gamma \epsilon_k^{\gamma\beta}=0 ~?$$
The analysis of this asymptotic behavior is non-trivial and it will be postponed to a future project. The next section presents a more precise rate in case the diffusion coefficient is independent from the control variable.
\end{remark}

Next, we present the rates $h_k$ associated with the strongly controlled Wiener functionals (\ref{pdsdeBM}), (\ref{limsdefbm}) and (\ref{priceprocess}). In the sequel, we employ the notation $r_k\lesssim \epsilon_k^{a^-}$ for a non-negative sequence $r_k$ and $a>0$ as follows: For every $0 < \eta < 1$ sufficiently close to $1$, there exists a constant $C$ which depends on $\eta$ and $a$ such that

\begin{equation}\label{lessNOTA}
r_k \le C \epsilon_k^{a\eta},
\end{equation}
for every $k\ge 1$.

\begin{proposition}\label{mainPrSDEBM}
Let $X^u$ be the controlled SDE

\begin{equation}\label{sde1}
dX^u(t) = \alpha(t,X^u_t, u(t))dt + \sigma(t,X^u_t,u(t))dB(t),
\end{equation}
where the non-anticipative functionals $(\alpha,\sigma)$ satisfy Assumption (C1) with $0 < \theta \le 1$. Let $\mathcal{X} = \big((X^k)_{k\ge 1},\mathscr{D}\big)$ be the controlled imbedded discrete structure associated with $X$ given by (\ref{xksde}). Then,

$$
\sup_{\phi \in U^{k,e(k,T)}_0}\mathbb{E}\big\|X^{k,\phi}_T - X^\phi_T\big\|_\infty \lesssim \epsilon_k^{\big(\frac{1}{2}\wedge 2\theta\big)^-}\rightarrow 0,
$$
as $k\rightarrow \infty$. In case the diffusion component is independent from the control variables $\sigma(t,f_t,a) = \sigma(t,f_t)$, then

\begin{equation}\label{sstrong1}
\sup_{\phi \in U^{T}_0}\mathbb{E}\big\|X^{k,\phi}_T - X^\phi_T\big\|_\infty \lesssim \epsilon_k^{\big(\frac{1}{2}\wedge 2\theta\big)^-}  \rightarrow 0,
\end{equation}
as $k\rightarrow \infty$.

%Moreover,
%$$\Big|V_k(0) - \sup_{\phi\in U^T_0}\mathbb{E}\big[\xi_X(\phi)\big]\Big|\lesssim $$
\end{proposition}

\begin{proposition}\label{mainPrSDEFBM}
Let $X^u$ be the controlled SDE

\begin{equation}\label{sde2}
dX^u(t) = \alpha(t,X^u_t, u(t))dt + \sigma dB_H(t),
\end{equation}
where $B_H$ is a real-valued FBM with $0 < H < \frac{1}{2}$, the non-anticipative functional $\alpha$ satisfies Assumption (C1) with $0 < \theta \le 1$ and $\sigma >0$. Let $\mathcal{X} = \big((X^k)_{k\ge 1},\mathscr{D}\big)$ be the controlled imbedded discrete structure associated with $X$ given by (\ref{xksdeFBM}). Then,
%Take a sequence $\eta_k \downarrow 0$ which fulfills the growth conditions in Lemma \ref{Tkrate}. Under Assumption (C1),

\begin{equation}\label{sstrong2}
\sup_{\phi \in U^T_0}\mathbb{E}\|X^{k,\phi}_T -  X^{\phi}_T\|_\infty \lesssim \epsilon_k^{(H\wedge 2\theta)^-}\rightarrow 0,
\end{equation}
as $k\rightarrow +\infty$.

%Moreover,
%$$\Big|V_k(0) - \sup_{\phi\in U^T_0}\mathbb{E}\big[\xi_X(\phi)\big]      \Big|\lesssim \epsilon_k^{\rho} \rightarrow 0,~\text{as}~k\rightarrow \infty,$$
%for each $0 < \rho < H$.
\end{proposition}

\begin{proposition}\label{mainPrRV}
Fix $0 < H < \frac{1}{2}$. Let $X^u$ be the controlled SDE
\begin{equation}\label{sde3}
\left\{
\begin{array}{lc}
dX^u(t) = X^u(t) \mu(u(t))dt + X^u(t)  \vartheta(Z(t),u(t)) dB^1(t) \\
dZ(t) = \nu dW_H(t)- \beta(Z(t)-m)dt,& ~Z(0)=z_0,
\end{array}
\right.
\end{equation}
where $m=0$ (for simplicity), $\beta, \nu >0$ and $\vartheta,\mu$ satisfy Assumption (D1). Let $\mathcal{X} = \big((X^k)_{k\ge 1},\mathscr{D}\big)$ be the controlled imbedded discrete structure associated with $X$ given by (\ref{xksdeRV}). Then,

$$
\sup_{\phi \in U^{k,e(k,T)}_0}\mathbb{E}\|X^{k,\phi}_T - X^{\phi}_T\|_\infty \lesssim \epsilon_k^{H^-}\rightarrow 0,
$$
as $k\rightarrow +\infty$. In case $\vartheta$ is independent from the control variables $\vartheta(z,a) = \vartheta(z)$, then

\begin{equation}\label{sstrong3}
\sup_{\phi \in U^{T}_0}\mathbb{E}\|X^{k,\phi}_T - X^{\phi}_T\|_\infty \lesssim \epsilon_k^{H^-}\rightarrow 0,
\end{equation}
as $k\rightarrow +\infty$.

\end{proposition}

\begin{remark}\label{fullcontrolrem}
 We stress that (\ref{sstrong1}), (\ref{sstrong2}) and (\ref{sstrong3}) present the rates over the whole set of controls $U^{T}_0$ rather than $U^{k,e(k,T)}_0$. In order to remove the dependence on $L_\epsilon$ in Theorem \ref{absrate1}, those estimates will play a key role (see (\ref{cho1})) jointly with a suitable convexity assumption on the drift.
\end{remark}

The proofs of Propositions \ref{mainPrSDEBM}, \ref{mainPrSDEFBM} and \ref{mainPrRV} are given in Section \ref{Eulerproofs}.

\subsection{Optimal control of drifts}\label{pcdsection}

This section refines the convergence rate obtained in Theorem \ref{absrate1} in the particular case where only the drift term  contains the control variables. More precisely, we remove the dependence of $L_\epsilon$ in (\ref{quasirate}). For this purpose, we incorporate the following convexity assumption on the drift.

\

\noindent \textbf{Assumption (F1)}: For every  $\pi_2(\mathbf{o}^k_{i})= \{(0,x_0),  (s^k_1,x_1^k) , \ldots , (s^k_{i}, x_i^k)\}  \in \big(\pi_2(\mathbb{H}_k)\big)^i$ and $t_i^k = s^k_0 + s^k_1 + \cdots + s^k_i$, the set

\[
\mathbb{D}_{i+1} (\mathbf{o}^k_{i}) :=\left\{  q \big({t^k_i},\gamma_i (\pi_2(\mathbf{o}^k_i)), a \big); ~ a \in \mathbb{A} \right\}\subset \mathbb{R}^n
\] is convex for each $i=0,1,2, \ldots , e(k,T)-1$. Here, $q(t,f,a):= \alpha(t,f_t,a)$ in the SDEs (\ref{sde1}) and (\ref{sde2}) and $q(t,f,a):=f(t)\mu(a)$ in the SDE (\ref{sde3}).

\

\begin{theorem}\label{THdriftrate}
Let $X$ be the controlled SDEs given by (\ref{sde1}), (\ref{sde2}) and (\ref{sde3}), where only the correspondent drifts contain the control variables and let $\mathcal{X} = \big((X^k)_{k\ge 1},\mathscr{D}\big)$ be the controlled imbedded discrete structures associated with $X$ given by (\ref{Xeuler}), (\ref{eulerFBMd}) and (\ref{Seuler}), respectively. Suppose that Assumptions (C1-D1-F1) are in force. Let $(h_k)_{k\ge 1}$ be the rate associated with $\mathcal{X}$ given in Propositions \ref{mainPrSDEBM}, \ref{mainPrSDEFBM} and \ref{mainPrRV}. Let $\xi$ be a payoff functional satisfying Assumption (A1) with $\gamma$-H\"older regularity. Then,

$$ \Big|\sup_{\phi\in U^T_0}\mathbb{E}\big[\xi_X(\phi)\big] - V_k(0)\Big| \lesssim h_k^\gamma\rightarrow 0,~\text{as}~k\rightarrow \infty.$$
\end{theorem}

The proof of Theorem \ref{THdriftrate} relies on the following argument. To keep notation simple, starting with a given initial condition $X^{k,u}(0)=x_0$ and a control $u \in U^T_0$, we set

$$
\mathcal{V}(\Delta X^{k,u}(T^k_{j})):=\left\{
\begin{array}{rl}
\sigma \big(T^{k}_{j},X^{k,u}_{T^{k}_{j}}\big)\Delta A^k(T^k_{j+1}),
; & \hbox{for} \ (\ref{Xeuler}) \\
\sigma \Delta B^k_H(T^k_{j+1});& \hbox{for} \ (\ref{eulerFBMd}) \\
X^{k,u}(T^{k}_{j})\vartheta(Z^k(T^k_j))\Delta A^{k,1}(T^k_{j+1});& \hbox{for} \ (\ref{Seuler})
\end{array}
\right.
$$
for $0\le j\le e(k,T)-1$. Recall
\begin{eqnarray}
\nonumber X^{k,u}(T^k_{j+1})&=&X^{k,u}(T^k_{j}) + \int_{T^k_{j}}^{T^k_{j+1}}q \Big(T^k_{j},X^{k,u}_{T^k_{j}},u(s)\Big)ds\\
\label{XeulerR}& &\\
\nonumber&+& \mathcal{V}(\Delta X^{k,u}(T^k_{j}))
\end{eqnarray}
for $0\le j\le e(k,T)-1$, where the function $q$ in (\ref{XeulerR}) is defined in Assumption (F1).

The operator $\Xi^{k,u}_j$ given by (\ref{Xioperator}) can be extended to controls in $U^T_0$ by setting

$$
\Xi^{k,u}_j :=\Big((\Delta T^k_1,\Delta A^k(T^k_1), \Delta X^{k, u}(T^k_1 ) ), \ldots, (\Delta T^k_j, \Delta A^k(T^k_j), \Delta X^{k, u}(T^k_j) ) \Big),
$$
for $1\le j\le e(k,T)$. By construction, for each $u \in U^T_0$,
$$X^{k,u}_T = \gamma_{e(k,T)} \big(\Xi^{k,u}_{e(k,T)}\big).$$
Mean value theorem yields the following pathwise statement: For every $u \in U^T_0$ (with continuous paths $t\mapsto u(t)$) and a deterministic path $\mathbf{o}^k_{i}$, there exists a random variable $a^k_i$ (which obviously depends on $\mathbf{o}^k_i$ and $u$) satisfying

\begin{equation}\label{repl}
\int_{t^k_i}^{t^k_i + \Delta T^k_{i+1} (\omega)} q \big({t^k_i},\gamma_i(\pi_2(\mathbf{o}^k_i)),u(s, \omega)\big)ds = q \big({t^k_i},\gamma_i(\pi_2(\mathbf{o}^k_i)), a^k_i (\omega) )\big) \Delta T^k_{i+1} (\omega),
\end{equation}
for each $\omega \in \Omega$ and $0\le i\le e(k,T)-1$.

\

\begin{lemma} \label{LB3}
For every control $u \in U_0^T$ (with continuous paths) and an integer $k\ge 1$, there exists a control  $u^k \in U^{k,e(k,T)}_0$  satisfying

\begin{equation}\label{contrCONS}
\mathbb{E}\big[\xi_{X^k}(u)\big]  = \mathbb{E}\big[\xi_{X^k}(u^k)\big].
\end{equation}
\end{lemma}
\begin{proof}
In the sequel, we fix $u \in U^T_0$ with continuous paths $t\mapsto u(t)$. To keep notation simple, we set $m=e(k,T)$. We will construct by forward induction the required discrete control $u^k \in U_0^{k, m-1}$. At step $i=0$, for a given $ \mathbf{o}^k_{0}=(0,0,x_0)$, we apply (\ref{repl}) at the first jump $\Delta X^{k,u}(T^k_1)$ so that

$$\mathbb{E} \left[ \xi_{X^k}(u) \mid \Xi^{k,u}_{0}=\mathbf{o}^k_0  \right]
$$

$$ = \mathbb{E} \left[ \xi \left(\gamma_m (\pi_2(\Xi^{k,u}_0), (\Delta T^k_1 , \Delta X^{k,u} (T^k_{1}) ), \ldots ,(\Delta T^k_m , \Delta X^{k,u} (T^k_{m}) )) \right) \mid \Xi^{k,u}_{0}  = \mathbf{o}^k_0  \right]
$$

$$= \mathbb{E} \left[ \xi \left(\gamma_m (x_0, (\Delta T^k_1 , \Delta X^{k,u} (T^k_{1}) ), \ldots ,(\Delta T^k_m , \Delta X^{k,u} (T^k_{m}) )) \right) \mid \Xi^{k,u}_{0}  = \mathbf{o}^k_0   \right]
$$
\begin{equation}\label{repl1}
= \mathbb{E} \left[ \xi \left(\gamma_m (x_0, (\Delta T^k_1 , \Delta X^{k,a^k_0} (T^k_{1}) ), \ldots ,(\Delta T^k_m , \Delta X^{k,u} (T^k_{m}) )) \right) \mid \Xi^{k,u}_{0}  = \mathbf{o}^k_0   \right],
\end{equation}
where $a^k_0$ is a random variable. At this point, we define $T_0:\mathbb{D}_1(\mathbf{o}^k_0)\rightarrow \mathbb{R}$ as follows:

\begin{equation}\label{Toperator}
T_0 \big(q (0,x_0, a) \big)
\end{equation}
$$
:=\mathbb{E} \left[ \xi \left(\gamma_m (x_0, (\Delta T^k_1 , \Delta X^{k,a\otimes_1u} (T^k_{1}) ), \ldots ,(\Delta T^k_m , \Delta X^{k,a\otimes_1 u} (T^k_{m}) )) \right) \mid \Xi^{k,u}_{0}  = \mathbf{o}^k_0   \right],
$$
for $a \in \mathbb{A}$. In  (\ref{Toperator}), we recall

$$\gamma_m (x_0, (\Delta T^k_1 , \Delta X^{k,a\otimes_1 u} (T^k_{1}) ), \ldots ,(\Delta T^k_m , \Delta X^{k,a\otimes_1u} (T^k_{m}) ))$$
$$ = x_0 +  \sum_{i=1}^m \Delta X^{k,a\otimes_1u}(T^k_i)\mathds{1}_{\{T^k_i\le \cdot\}},$$
where $a\otimes_1 u = a\otimes_{T^k_1} u$ is a concatenation of the constant control $a \in \mathbb{A}$ with $u \in U^T_0$ (see (\ref{concatenation})),

$$\Delta X^{k,a\otimes_1 u}(T^k_1) =  q(0,x_0, a)\Delta T^k_1 + \mathcal{V}(\Delta X^{k,a}(T^k_0)),$$
with $\Delta X^{k,a}(T^k_0) = x_0$ and

%$$\Delta X^{k,u}(T^k_2) = \int_{T^k_1}^{T^k_2}  \alpha(T^k_1, X^{k,a}_{T^k_1},u(s))ds+   \sigma(T^k_1, X^{k,a}_{T^k_1})\Delta A^k(T^k_2),$$
%where $X^{k,a}(0)=x$ and, in general,

$$
\Delta X^{k,a\otimes_1 u}(T^k_{j+1}) = \int_{T^k_{j}}^{T^k_{j+1}}  q(T^k_{j}, X^{k,a\otimes_1 u}_{T^k_{j}},u(s))ds+ \mathcal{V}(\Delta X^{k,a\otimes_1u}(T^k_{j})),
$$
for $j=1,\ldots, m-1$. Assumptions (A1-C1-D1) imply that $e\mapsto T_0(e)$ is continuous on the compact set $ \mathbb{D}_1(\mathbf{o}^k_0)$. Let us consider

\[
r(\mathbf{o}^k_0) := \min_{a \in \mathbb{A}}  T_0(q(0,x_0;a)) \quad \text{and} \quad R(\mathbf{o}^k_0) := \max_{a \in \mathbb{A}} T_0(q(0,x_0;a)).
\]

By (\ref{repl}), (\ref{repl1}) and the definition of the map (\ref{Toperator}), we observe

\[
 r(\mathbf{o}^k_0)\leq \mathbb{E} \left[ \xi \left(\gamma_m (\mathbf{o}^k_0, (\Delta T^k_1 , \Delta X^{k,u} (T^k_{1}) ), \ldots ,(\Delta T^k_m , \Delta X^{k,u} (T^k_{m}) )) \right) \mid \Xi^{k,u}_{0}  = \mathbf{o}^k_0   \right] \leq R(\mathbf{o}^k_0).
\]
By Assumption (F1) and the continuity of $T_0:\mathbb{D}_1(\mathbf{o}^k_0)\rightarrow \mathbb{R}$, the image $T_0(\mathbb{D}_1(\mathbf{o}^k_0))$ is an interval on the real line. Therefore, for each $\mathbf{o}^k_0 = (0,0,x_0)$, there exists $a^k_{0} (\mathbf{o}^k_0) \in \mathbb{A}$ such that

\begin{equation}\label{Teq}
\mathbb{E} \left[ \xi \left(\gamma_m (\mathbf{o}^k_0, (\Delta T^k_1 , \Delta X^{k,u} (T^k_{1}) ), \ldots ,(\Delta T^k_m , \Delta X^{k,u} (T^k_{m}) )) \right) \mid \Xi^{k,u}_{0}  = \mathbf{o}^k_0   \right]= T_0 \big(q (0,x_0, a^k_0(\mathbf{o}^k_0)) \big).
\end{equation}

Next, we introduce the following Borel measurable function

\begin{eqnarray*}
y^k_0 (\mathbf{o}^k_0 , a) &:=& \Big| \mathbb{E} \left[ \xi \left(\gamma_m (\mathbf{o}^k_0, (\Delta T^k_1 , \Delta X^{k,u} (T^k_{1}) ), \ldots ,(\Delta T^k_m , \Delta X^{k,u} (T^k_{m}) )) \right) \mid \Xi^{k,u}_{0}  = \mathbf{o}^k_0   \right]  \\
&-&  T_0 \big(q (0,x_0,a) \big) \Big|^2.
\end{eqnarray*}
Relation (\ref{Teq}) yields

\[
y^k_0 (\mathbf{o}^k_0 , a^k_0(\mathbf{o}^k_0)) = \inf_{a \in \mathbb{A}} y^k_0 (\mathbf{o}^k_0 , a), \quad \mathbf{o}^k_0 \in \{0\}\times\{0\} \times \mathbb{R}^n.
\]

Therefore, Theorem 7.50 (b) in \cite{bertsekas} yields the existence of an universally measurable function $g^k_{0} :\{0\} \times \{0\}\times \mathbb{R}^n   \rightarrow   \mathbb{A}$ such that

\[
y^k_0 (\mathbf{o}^k_0 , g^k_0 (\mathbf{o}^k_0)) = y^k_0 (\mathbf{o}^k_0 , a^k_0(\mathbf{o}^k_0)) = 0, \quad \mathbf{o}^k_0 \in \{0\} \times \{0\}\times \mathbb{R}^n.
\] Then, we conclude that

\[\mathbb{E} \left[ \xi \left(\gamma_m (\mathbf{o}^k_0, (\Delta T^k_1 , \Delta X^{k,u} (T^k_{1}) ), \ldots ,(\Delta T^k_m , \Delta X^{k,u} (T^k_{m}) )) \right) \mid \Xi^{k,u}_{0}  = \mathbf{o}^k_0   \right]=
\]

\[
 \mathbb{E} \left[ \xi \left(\gamma_m (\mathbf{o}^k_0, (\Delta T^k_1 , \Delta X^{k,g^k_0(\mathbf{o}^k_0)\otimes_1u} (T^k_{1}) ), \ldots ,(\Delta T^k_m , \Delta X^{k,g^k_0(\mathbf{o}^k_0)\otimes_1u} (T^k_{m}) )) \right) \mid \Xi^{k,u}_{0}  = \mathbf{o}^k_0   \right].
\]
By construction, if $u^k_0 (t) := g^k_{0} (\mathbf{o}^k_0) 1\!\!1_{ \{0 < t \leq T^k_1 \} }$, then

\begin{eqnarray}
\nonumber\mathbb{E}\Big[\xi_{X^k}(u)\Big]  &=& \mathbb{E}\Big[\mathbb{E}\big[\xi_{X^k}(u)|\Xi^{k,u}_0\big]\Big]\\
\nonumber&=&  \mathbb{E}\Big[\mathbb{E}\big[\xi_{X^k}(u^k_0 \otimes_1 u)|\Xi^{k,u}_0\big]\Big]\\
\label{it0}& = &\mathbb{E}\Big[\xi_{X^k}(u^k_0 \otimes_1 u)\Big].
\end{eqnarray}
At step $i=1$ and for a deterministic path $\mathbf{o}^k_{1}=(0,x_0, (w^k_1, x^k_1))$, we apply (\ref{repl}) to the second jump $\Delta X^{k,u^k_0\otimes_1 u}(T^k_2)$ and we get

$$
\mathbb{E} \left[ \xi_{X^k}(u^k_0\otimes_1 u) \mid \Xi^{k,u^k_0}_{1} = \mathbf{o}^k_1  \right]$$
$$= \mathbb{E} \left[ \xi \left(\gamma_m (\pi_2(\Xi^{k,u^k_0}_1), (\Delta T^k_2 , \Delta X^{k,u^k_0\otimes_1 u} (T^k_{2}) ), \ldots ,(\Delta T^k_m , \Delta X^{k,u^k_0\otimes_1 u} (T^k_{m}) )) \right) \mid \Xi^{k,u^k_0}_{1}  = \mathbf{o}^k_1  \right]
$$
$$= \mathbb{E} \left[ \xi \left(\gamma_m (\pi_2(\mathbf{o}^k_1), (\Delta T^k_2 , \Delta X^{k,u^k_0\otimes_1 u} (T^k_{2}) ), \ldots ,(\Delta T^k_m , \Delta X^{k,u^k_0\otimes_1 u} (T^k_{m}) )) \right) \mid \Xi^{k,u^k_0}_{1}  = \mathbf{o}^k_1   \right]$$
\begin{equation}\label{repl2}
=\mathbb{E} \left[ \xi \left(\gamma_m (\pi_2(\mathbf{o}^k_1), (\Delta T^k_2 , \Delta X^{k,u^k_0\otimes_1 a^k_1\otimes_2 u} (T^k_{2}) ), \ldots ,(\Delta T^k_m , \Delta X^{k,u^k_0\otimes_1 u} (T^k_{m}) )) \right) \mid \Xi^{k,u^k_0}_{1}  = \mathbf{o}^k_1   \right],
\end{equation}
where the existence of the random variable $a^k_1$ in (\ref{repl2}) is due to (\ref{repl}). We define $T_1:\mathbb{D}_2(\mathbf{o}^k_1)\rightarrow \mathbb{R}$ as follows:

$$
T_1 \big(q (t^k_1,\gamma_1(\pi_2(\mathbf{o}^k_1), a) \big):=
$$
$$\mathbb{E} \Big[ \xi \Big(\gamma_m (\pi_2(\mathbf{o}^k_1), (\Delta T^k_2 , \Delta X^{k,u^k_0\otimes_1 a\otimes_2 u} (T^k_{2}) ), \ldots ,(\Delta T^k_m , \Delta X^{k,u^k_0\otimes_1 a\otimes_2 u} (T^k_{m}) ))\Big)\big|  \Xi^{k,u^k_0}_{1}  = \mathbf{o}^k_1 \Big],$$
for $a \in \mathbb{A}$. We then proceed in the same way as in the step $i=0$ to find a universally measurable function $g^k_{1} :\mathbb{H}_k^0\times \mathbb{H}_k \rightarrow   \mathbb{A}$ such that $u^k_{1} (t) := g^k_{1} (\Xi^{k,u^k_0}_1) 1\!\!1_{ \{T^k_{1} < t \leq T^k_{2} \} }$ realizes

\begin{eqnarray}
\nonumber\mathbb{E}\Big[\xi_{X^k}(u^k_0\otimes_1 u)\Big]  &=& \mathbb{E}\Big[\mathbb{E}\big[\xi_{X^k}(u^k_0\otimes_1 u)|\Xi^{k,u^k_0}_1\big]\Big]\\
\nonumber &=&  \mathbb{E}\Big[\mathbb{E}\big[\xi_{X^k}(u^k_0 \otimes_1 u^k_1\otimes_2 u)|\Xi^{k,u^k_0}_1\big]\Big]\\
\label{it1}  &=& \mathbb{E}\Big[\xi_{X^k}(u^k_0 \otimes_1u^k_1\otimes_2 u)\Big].
\end{eqnarray}
By iterating this procedure via (\ref{it0}) and (\ref{it1}), we construct $u^k=u^k_0 \otimes_1 u^k_1 \otimes_2 \ldots\otimes_{m-1} u^k_{m-1}$ realizing (\ref{contrCONS}). By recalling that compositions of universally measurable functions are universally measurable, we also conclude $u^k = u^k_0 \otimes_1 u^k_1 \otimes_2 \ldots\otimes_{m-1} u^k_{m-1} \in U^{k,m}_0$. This concludes the proof.

\end{proof}
\noindent \textbf{Proof of Theorem \ref{THdriftrate}:} In the sequel, $C$ is a constant which may differ from line to line. For a given $\epsilon >0$, from Lemma \ref{Liplemma}, we may choose $\phi^{\epsilon}  \in U^T_0(Lip)$ such that
\begin{eqnarray}
\nonumber V(0) &\le& \mathbb{E}[\xi_X(\phi^\epsilon)] + \epsilon
= \mathbb{E}[\xi_X(\phi^\epsilon)] - \mathbb{E}[\xi_{X^k}(\phi^\epsilon)] + \mathbb{E}[\xi_{X^k}(\phi^\epsilon)] +\epsilon\\
\label{cho1}&\le& C h^\gamma_k + \mathbb{E}[\xi_{X^k}(\phi^\epsilon)] +\epsilon\\
\label{cho2}&=& C h^\gamma_k + \mathbb{E}[\xi_{X^k}(\phi^{k,\epsilon})] +\epsilon\\
\nonumber&\le&  C h^\gamma_k + V_k(0) +\epsilon
\end{eqnarray}
where (\ref{cho1}) is due to (\ref{sstrong1}), (\ref{sstrong2}) and (\ref{sstrong3}) given in Propositions \ref{mainPrSDEBM}, \ref{mainPrSDEFBM} and \ref{mainPrRV}, respectively. The equality (\ref{cho2}) is due to Lemma \ref{LB3} so that we may choose $\phi^{k,\epsilon} \in U^{k,e(k,T)}_0$ in such way that $\mathbb{E}[\xi_{X^k}(\phi^\epsilon)] = \mathbb{E}[\xi_{X^k}(\phi^{k,\epsilon})]$. By repeating the same argument as in (\ref{oth}), we also have $V_k(0)\le Ch_k^\gamma + V(0).$ The constant $C$ which appears in these estimates does not on $\epsilon$ and hence we conclude the proof.

\section{Proof of Propositions \ref{mainPrSDEBM}, \ref{mainPrSDEFBM} and \ref{mainPrRV}}\label{Eulerproofs}
Throughout this section, $C$ is a generic positive constant which may differ from line to line.

\subsection{Proof of Proposition \ref{mainPrSDEBM}}
Similar to the classical Euler scheme, we need to define a continuous version of $\mathbb{X}^{k,\phi}$ for $\phi \in U^T_0$. We set

$$\Sigma^{ij,k,\phi}(t):= 0 \mathds{1}_{\{t=0\}} + \sum_{n=1}^\infty \sigma^{ij}\big(T^{k}_{n-1}, \mathbb{X}^{k,\phi}_{T^{k}_{n-1}},\phi(t)\big)\mathds{1}_{\{T^{k}_{n-1} < t \le T^{k}_n\}},
$$
for $0\le t\le T, 1\le i\le n, 1\le j\le d$. We observe for each $\phi \in U^T_0$, $\Sigma^{k,\phi}$ is $\mathbb{F}$-progressively measurable.

We define

$$
\widehat{\mathbb{X}}^{k,\phi}(t):=x_0 + \int_0^t\alpha (\bar{s}_k, \mathbb{X}^{k,\phi}_{\bar{s}_k}, \phi(s))ds + \int_0^t{\Sigma}^{k,\phi}(s)dB(s),
$$
for $0\le t\le T$. It is important to notice that

%\begin{equation}\label{r1mult}
$$\mathbb{X}^{k,\phi}(t) =\mathbb{X}^{k,\phi}(\bar{t}_k) = \widehat{\mathbb{X}}^{k,\phi}(\bar{t}_k),~\mathbb{X}^{k,\phi}_t = \mathbb{X}^{k,\phi}_{\bar{t}_k},$$
%\end{equation}
for every $t\in[0,T]$. The idea is to analyse

$$\mathbb{E}\|\mathbb{X}^{k,\phi}_T- X^{\phi}_T \|_\infty\le \mathbb{E}\|\mathbb{X}^{k,\phi}_T - \widehat{\mathbb{X}}^{k,\phi}_T \|_\infty + \mathbb{E}\|\widehat{\mathbb{X}}^{k,\phi}_T- X^{\phi}_T\|_\infty. $$

\begin{lemma}\label{passo1mult}
Under assumption (C1),
$$
\sup_{k\ge 1}\sup_{\phi\in U^{T}_0}\mathbb{E}\|\mathbb{X}^{k,\phi}_T\|^p_\infty < \infty, \quad \forall p> 1.
$$
\end{lemma}
\begin{proof}
From Assumption (C1), there exists a constant $C$ such that

\begin{equation}\label{r2mult}
|\alpha^i(t,\omega,a)| + |\sigma^{ij}(t,\omega,a)|\le C(1+ \|\omega_t\|_\infty),
\end{equation}
for every $(t,\omega,a)\in [0,T]\times \mathbf{D}_{n,T}\times \mathbb{A}$ and $1\le i\le n,1\le j\le d$, where $C$ only depends on $T, \alpha(0,0,0)$, $\sigma(0,0,0)$, the compact set $\mathbb{A}$ and the Lipschitz constants $(K_{1,Lip}, K_{2,Lip})$. The proof consists on routine arguments based on Burkholder-Davis-Gundy and Jensen inequalities jointly with Gr\"{o}nwall's inequality on the function $s\mapsto \mathbb{E}\|\mathbb{X}^{k,\phi}_s\|^p_\infty$. Indeed, recall

$$\mathbb{X}^{k,\phi}(t) = x_0 + \int_0^{\bar{t}_k}\alpha(\bar{s}_k,\mathbb{X}^{k,\phi}_{\bar{s}_k},\phi(s))ds +\int_0^{\bar{t}^k}\Sigma^{k,\phi}(s)dB(s).$$

Fix $p \ge 2$. By applying Jensen's ineguality and using (\ref{r2mult}), we get

$$
\mathbb{E}\sup_{0\le t\le T}\Bigg|\int_0^{\bar{t}_k}\alpha (\bar{s}_k,\mathbb{X}^{k,\phi}_{\bar{s}_k},\phi(s))ds\Bigg|^p\le TC\Bigg(1+\int_0^T\mathbb{E}\|\mathbb{X}^{k,\phi}_{\bar{s}_k}\|^p_\infty ds\Bigg).
$$
Burkholder-Davis-Gundy and Jensen inequalities jointly with (\ref{r2mult}) yield

$$
\mathbb{E}\sup_{0\le t\le T}\Bigg|\int_0^{t}\Sigma^{ij,k,\phi}(s)dB^j(s)\Bigg|^p\le\mathbb{E}\Bigg(\int_0^T|\Sigma^{ij,k,\phi}(s)|^2ds\Bigg)^{\frac{p}{2}}
$$
$$\le C\mathbb{E}\int_0^T|\Sigma^{ij,k,\phi}(s)|^pds\le C \Big(1+\int_0^T\mathbb{E}\|\mathbb{X}^{k,\phi}_{\bar{s}_k}\|^p_\infty ds\Big).$$

Summing up the above estimates, we have

$$\mathbb{E}\|\mathbb{X}^{k,\phi}_T\|^p_\infty\le C\Bigg(1+ \int_0^T\mathbb{E}\|\mathbb{X}^{k,\phi}_{s}\|^p_\infty ds \Bigg).$$
Grownall's inequality and the fact that $C$ does not depends on on $\phi$ and $k\ge 1$ allow us to conclude the result.
\end{proof}

%%%%%%%%%%%%%%%%%%%%%%%%%%%%%%%%% SM %%%%%%%%%%%%%%%%%%%%%%%%%%%%%%%%%%%%%%%%%%%%%
\begin{lemma}\label{passo2mult}
Fix $p\ge 2$ such that $\theta p \ge 1$, where $\theta \in (0,1]$ is the H\"older regularity of the time variable in Assumption (C1). For a given $0 < \beta< 1$, there exists a constant $C$ which only depends on $\alpha$, $\sigma, \beta, p, \theta$ and $T$  such that

\begin{equation*}
\begin{split}
\mathbb{E}\|\widehat{\mathbb{X}}^{k,\phi}_T- X^{\phi}_T\|^p_\infty \le C\Bigg\{ \epsilon_k^{2\theta p - 2(1-\beta)}   + \int_0^T\mathbb{E}\|\mathbb{X}^{k,\phi}_{s}- X^{\phi}_s\|^p_\infty ds \Bigg\},\\
\end{split}
\end{equation*}
for every $\phi \in U^{T}_0$ and $k\ge 1$.
\end{lemma}
\begin{proof}
Let us fix $1\le i\le n$, $\phi\in U^{T}_0$, $\beta \in (0,1)$, $\theta \in (0,1]$. Let $p\ge 2$ such that $p\theta\ge 1$. We shall write

\begin{equation*}\label{eu1}
\begin{split}
\widehat{\mathbb{X}}^{i,k,\phi}(t) - X^{i,\phi}(t) &= \int_0^t\Big[\alpha^i(\bar{s}_k,\mathbb{X}^{k,\phi}_{\bar{s}_k},\phi(s)) - \alpha^i(s,X^{\phi}_{s},\phi(s))\big]ds\\
&+\sum_{j=1}^d\int_0^t\Big[\Sigma^{ij,k,\phi}(s) - \sigma^{ij}(s,X^{\phi}_s,\phi(s))\Big]dB^j(s).
\end{split}
\end{equation*}
Jensen's inequality yields

$$\sup_{0\le t\le T}\Bigg|\int_0^t\Big[\alpha^i(\bar{s}_k,\mathbb{X}^{k,\phi}_{\bar{s}_k},\phi(s)) - \alpha^i(s,X^{\phi}_{s},\phi(s))\big]ds\Bigg|^p$$
$$\le T \int_0^T\Big|\alpha^i(\bar{s}_k,\mathbb{X}^{k,\phi}_{\bar{s}_k},\phi(s)) - \alpha^i(s,X^{\phi}_{s},\phi(s))\big|^pds. $$
By using Assumption (C1), there exists a constant $C$ which only depends on $K_{Lip}$ such that

\begin{eqnarray*}
|\alpha^i(\bar{s}_k,\mathbb{X}^{k,\phi}_{\bar{s}_k},\phi(s)) - \alpha^i(s,X^{\phi}_{s},\phi(s))\big|^p&\le& C \Big\{|\bar{s}_k - s|^{p\theta} + \|\mathbb{X}^{k,\phi}_{\bar{s}_k} - X^{\phi}_{s}\|^p_\infty\Big\}\\
&\le & C \Big(\max_{j\ge 0} \Delta T^k_{j+1} \mathds{1}_{\{T^k_j\le T\}}\Big)^{p\theta}\\
&+& C \|\mathbb{X}^{k,\phi}_{\bar{s}_k} - X^{\phi}_{s}\|^p_\infty~a.s,
\end{eqnarray*}
for every $s \in [0,T]$. Then, Lemma \ref{meshlemma} yields

$$\mathbb{E}\sup_{0\le t\le T}\Bigg|\int_0^t\Big[\alpha^i(\bar{s}_k,\mathbb{X}^{k,\phi}_{\bar{s}_k},\phi(s)) - \alpha^i(s,X^{\phi}_{s},\phi(s))\big]ds\Bigg|^p$$
$$\le C \epsilon_k^{2p\theta - 2(1-\beta)}  + C \int_0^T \mathbb{E}\|\mathbb{X}^{k,\phi}_{\bar{s}_k} - X^{\phi}_{s}\|^p_\infty ds,$$
where $C$ above depends on $K_{Lip}$, $\beta$, $\theta$, $p$ and $T$. Next, we consider the stochastic integral. By Burkholder-Davis-Gundy's inequality and Assumption (C1), there exists a constant $C$ such that

$$\mathbb{E}\sup_{0\le t\le T}\Bigg|\int_0^t\Big[\Sigma^{ij,k,\phi}(s) - \sigma^{ij}(s,X^{\phi}_s,\phi(s))\Big]dB^j(s)\Bigg|^p$$
$$\le C \mathbb{E}\Bigg(\int_0^T \Big|\Sigma^{ij,k,\phi}(s) - \sigma^{ij}(s,X^{\phi}_s,\phi(s))\Big|^2ds\Bigg)^{\frac{p}{2}},$$
where

\begin{eqnarray*}
\Big|\Sigma^{ij,k,\phi}(s) - \sigma^{ij}(s,X^{\phi}_s,\phi(s))\Big|^2 &=& \Big|\sigma^{ij}(T^k_{n-1}, \mathbb{X}^{k,\phi}_{T^k_{n-1}}, \phi(s)) - \sigma^{ij}(s, X^{\phi}_s, \phi(s))\Big|^2\\
&\le & C \Big \{ |s-T^k_{n-1}|^{2\theta} + \| \mathbb{X}^{k,\phi}_{\bar{s}_k} - X^\phi_s\|^2_\infty \Big\}\quad \text{on}~\{T^k_{n-1} < s < T^k_n\}\\
&\le& C \Big\{ \big(\max_{j\ge 0} \Delta T^k_{j+1} \mathds{1}_{\{T^k_j\le T\}}\big)^{2\theta}+ \| \mathbb{X}^{k,\phi}_{\bar{s}_k} - X^\phi_s\|^2_\infty \Big\}.
\end{eqnarray*}
Therefore, Lemma \ref{meshlemma} yields

$$\mathbb{E}\sup_{0\le t\le T}\Bigg|\int_0^t\Big[\Sigma^{k,\phi}(s) - \sigma(s,X^{\phi}_s,\phi(s))\Big]dB(s)\Bigg|^p$$
$$\le C \epsilon_k^{2\theta p - 2(1-\beta)}  + C \int_0^T \mathbb{E}\|\mathbb{X}^{k,\phi}_{\bar{s}_k} - X^{\phi}_{s}\|^p_\infty ds,$$
where the constant $C$ above depends on Assumption (C1), $\beta$, $\theta$ and $p$. This concludes the proof.
\end{proof}

\begin{lemma}\label{passo3mult}
For every $p\ge 1$, there exists a constant $C$ which only depends on $T, \alpha,\sigma$ and $p$ such that
$$\sup_{\phi \in U^{k,\infty}_0}\mathbb{E}\|\mathbb{X}^{k,\phi}_T - \widehat{\mathbb{X}}^{k,\phi}_T\|^p_\infty \le C \epsilon_k^p,$$
for every $k\ge 1$. In case the control affects only the drift $\alpha$, then

$$\sup_{\phi \in U^{T}_0}\mathbb{E}\|\mathbb{X}^{k,\phi}_T - \widehat{\mathbb{X}}^{k,\phi}_T\|^p_\infty \le C  \epsilon_k^p,$$
for every $k\ge 1$.

\end{lemma}
\begin{proof}
Fix a control $\phi = \sum_{n=1}^\infty v^k_{n-1}\mathds{1}_{\{T^k_{n-1} < t \le T^k_n\}} \in U^{k,\infty}_0$ and $p\ge 1$. By definition,
$$\widehat{\mathbb{X}}^{k,\phi}(t) -  \mathbb{X}^{k,\phi}(t)= \alpha \Big(\bar{t}_k,\mathbb{X}^{k,\phi}_{\bar{t}_k},\phi(\bar{t}_k)\Big)(t-\bar{t}_k) + \sigma\Big(\bar{t}_k, \mathbb{X}^{k,\phi}_{\bar{t}_k}, \phi(\bar{t}_k)\Big)\Big(B(t) - B(\bar{t}_k)\Big)
$$
Then (\ref{r2mult}) yields

$$\|\widehat{\mathbb{X}}^{k,\phi}_T -  \mathbb{X}^{k,\phi}_T\|^p_\infty\le C \Big(1 + \|\mathbb{X}^{k,\phi}_T\|^p_\infty\Big)\Big\{ \big(\max_{n\ge 0}\Delta T^k_{n+1}\mathds{1}_{\{T^k_n\le T\}}\big)^{p} + \epsilon_k^p  \Big\}~a.s.$$
For a given $1 < q < \infty$, the H\"older's inequality jointly with Lemmas \ref{passo1mult} and \ref{meshlemma} in Appendix \ref{LDappendix} yield

\begin{eqnarray}
\nonumber\mathbb{E}\|\widehat{\mathbb{X}}^{k,\phi}_T -  \mathbb{X}^{k,\phi}_T\|^p_\infty\nonumber&\le & C\epsilon_k^{2p-\frac{2}{q}(1-\beta)} + \epsilon_k^p\\
\label{le1}&\le & C\epsilon_k^p,
\end{eqnarray}
for every $k\ge 1$. The constant $C$ in (\ref{le1}) does not depend on the controls of $U^{k,\infty}_0$. Now, we fix a control $\phi \in U^T_0$. We know that $\widehat{\mathbb{X}}^{k,\phi}(\bar{t}_k) =  \mathbb{X}^{k,\phi}(t)$ and hence

$$
\widehat{\mathbb{X}}^{k,\phi}(t) -  \mathbb{X}^{k,\phi}(t)= \int_{\bar{t}_k}^t \alpha (\bar{s}_k,\mathbb{X}^{k,\phi}_{\bar{s}_k},\phi(s))ds + \int_{\bar{t}_k}^t \Sigma^{k}(s)dB(s)
$$
where
\begin{equation}\label{Sigmawc}
\Sigma^{k}(t):= 0 \mathds{1}_{\{t=0\}} + \sum_{n=1}^\infty \sigma\big(T^{k}_{n-1}, \mathbb{X}^{k,\phi}_{T^{k}_{n-1}}\big)\mathds{1}_{\{T^{k}_{n-1} < t \le T^{k}_n\}},
\end{equation}
for $0\le t\le T$. It is important to notice that

$$
\widehat{\mathbb{X}}^{k,\phi}(t) -  \mathbb{X}^{k,\phi}(t)= \int_{\bar{t}_k}^t \alpha (\bar{s}_k,\mathbb{X}^{k,\phi}_{\bar{s}_k},\phi(s))ds + \sigma\Big(\bar{t}_k, \mathbb{X}^{k,\phi}_{\bar{t}_k}\Big) \big(B(t) - B(\bar{t}_k) \big),
$$
for $0\le t\le T$. Therefore, (\ref{r2mult}) yields

$$
\|\widehat{\mathbb{X}}^{k,\phi}_T -  \mathbb{X}^{k,\phi}_T\|^p_\infty\le C \big(1+ \|\mathbb{X}^{k,\phi}_T\|^p_\infty\big)\Big\{\big(\max_{n\ge 0}\Delta T^k_{n+1}\mathds{1}_{\{T^k_n\le T\}}\big)^{p} + \epsilon_k^p\Big\}~a.s,
$$
for $k\ge 1$. Hence, H\"older's inequality jointly with Lemmas \ref{passo1mult} and  \ref{meshlemma} yield (for a given $q>1$)

\begin{equation}\label{indc}
\mathbb{E}\|\widehat{\mathbb{X}}^{k,\phi}_T -  \mathbb{X}^{k,\phi}_T\|^p_\infty \le C\epsilon_k^{2p -\frac{2}{q}(1-\beta) } + \epsilon_k^p \le C \epsilon_k^p,
\end{equation}
for $k\ge 1$. The constant $C$ in (\ref{indc}) does not depend on controls. This concludes the proof.
\end{proof}

By Lemmas \ref{passo2mult}, \ref{passo3mult} and applying H\"older and Grownall's inequality, we then conclude the following result.

\begin{proposition}\label{semekT}
Fix $p\ge 2$ such that $\theta p \ge 1$, where $\theta \in (0,1]$ is the H\"older regularity of the time variable in Assumption (C1). For a given $0 < \beta< 1$, there exists a constant $C$ which only depends on $\alpha$, $\sigma, \beta, p, \theta$ and $T$  such that

\begin{equation*}
\begin{split}
\sup_{\phi \in U^{k,\infty}_0}\mathbb{E}\|\mathbb{X}^{k,\phi}_T - X^{\phi}_T\|_\infty&\le C \Big\{ \epsilon_k + \epsilon_k^{2\theta-\frac{2}{p}(1-\beta)}\Big\},
\end{split}
\end{equation*}
for every $k\ge 1$. In case the controls affect only the drift, we have
\begin{equation*}
\begin{split}
\sup_{\phi \in U^{T}_0}\mathbb{E}\|\mathbb{X}^{k,\phi}_T - X^{\phi}_T\|_\infty&\le  C \Big\{ \epsilon_k + \epsilon_k^{2\theta-\frac{2}{p}(1-\beta)}\Big\},
\end{split}
\end{equation*}
for $k\ge 1$.

\end{proposition}

Now we are able to present the proof of Proposition \ref{mainPrSDEBM}.

\

\noindent \textbf{Proof of Proposition \ref{mainPrSDEBM}}: Fix $p\ge 2$ such that $\theta p \ge 1$, where $\theta \in (0,1]$ is the H\"older regularity of the time variable in Assumption (C1). Fix an arbitrary $0 < \beta < 1$. For each $\phi\in U^{k,\infty}_0$, we shall  write
$$X^{k,\phi} - X^\phi = X^{k,\phi} - \mathbb{X}^{k,\phi} + \mathbb{X}^{k,\phi} - X^\phi,$$
and hence, by applying Proposition \ref{semekT}, we get

\begin{equation}\label{INTT1}
\sup_{\phi \in U^{k,\infty}_0}\mathbb{E}\| X^{k,\phi} - X^\phi\|_\infty\le C \Big\{ \epsilon_k + \epsilon_k^{2\theta-\frac{2}{p}(1-\beta)}\Big\} + C \sup_{\phi \in U^{k,\infty}_0}\mathbb{E}\| X^{k,\phi} - \mathbb{X}^{k,\phi}\|_\infty,
\end{equation}
for every $k\ge 1$. In the sequel, we fix $\phi \in U^{k,\infty}_0$. By the very definition, $\mathbb{X}^{k,\phi}(t) - X^{k,\phi}(t) = 0$ whenever $0\le t < T^k_{e(k,T)+1}$ and

\begin{eqnarray*}
\mathbb{X}^{k,\phi}(t) - X^{k,\phi}(t)&=& \int_{T^k_{e(k,T)}}^{\bar{t}_k} \alpha \big( \bar{s}_k, \mathbb{X}^{k,\phi}_{\bar{s}_k},\phi(s)\big)ds\\
&+& \int_{T^k_{e(k,T)}}^{\bar{t}_k}
\Sigma^{k,\phi}(s)dB(s),
\end{eqnarray*}
 whenever $T^k_{e(k,T)+1}\le t \le T$. By the additivity of the It\^o's integral, we shall write

 \begin{equation}\label{spd1}
 \int_{T^k_{e(k,T)}}^{\bar{t}_k}
\Sigma^{k,\phi}(s)dB(s) = \int_{\big(T^k_{e(k,T)},t\big]}
\Sigma^{k,\phi}(s)dB(s) - \int_{(\bar{t}_k,t]}
\Sigma^{k,\phi}(s)dB(s)
\end{equation}
for $T^k_{e(k,T)+1}\le t \le T$. We notice that $\phi \in U^{k,\infty}_0$ implies

\begin{eqnarray}
\nonumber\sup_{0\le t\le T}\Bigg|\int_{(\bar{t}_k,t]}
\Sigma^{k,\phi}(s)dB(s)\Bigg| &=& \sup_{0\le t\le T}\Big|\sigma \big(\bar{t}_k, \mathbb{X}^{k,\phi}_{\bar{t}_k},\phi(\bar{t}_k) \big) \big( B(t) - B(\bar{t}_k) \big)\Big|\\
\label{spd2}&\le& C (1+ \|\mathbb{X}^{k,\phi}_T\|_\infty)\epsilon_k~a.s,
\end{eqnarray}
for $k\ge 1$. By applying Burkholder-Davis-Gundy's inequality, H\"older's inequality, (\ref{r2mult}), Lemma \ref{passo1mult} and (\ref{astk2}) in Lemma \ref{Tkrate}, we get

\begin{eqnarray}
\nonumber\mathbb{E}\sup_{0\le t\le T}\Bigg|\int_{(T^k_{e(k,T)},t]}
\Sigma^{k,\phi}(s)dB(s)\Bigg| &\le& C \mathbb{E}\Big[(1+ \|\mathbb{X}^{k,\phi}_T\|_\infty)|T^k_{e(k,T)}-T|^{\frac{1}{2}}\mathds{1}_{\{T^k_{e(k,T)} < T\}}\Big]\\
\label{spd2o}&\le& C \epsilon_k^{\frac{\beta}{2}},
\end{eqnarray}
for $k\ge 1$. By (\ref{spd1}), (\ref{spd2}) and (\ref{spd2o}), we have

\begin{equation}\label{spd2oo}
 \mathbb{E}\sup_{0\le t\le T}\Bigg| \int_{T^k_{e(k,T)}}^{\bar{t}_k}
\Sigma^{k,\phi}(s)dB(s)\Bigg|\le C \epsilon_k^{\frac{\beta}{2}},
\end{equation}
for $k\ge 1$. The constant $C$ in the right-hand side of (\ref{spd2oo}) does not depend on $k\ge 1$ and controls. Again, by applying H\"older's inequality, (\ref{r2mult}), Lemma \ref{passo1mult} and (\ref{astk2}) in Lemma \ref{Tkrate}, we get (for $q>1$)

\begin{eqnarray}
\nonumber\mathbb{E}\sup_{0\le t\le T}\Bigg|\int_{T^k_{e(k,T)}}^{\bar{t}_k} \alpha \big( \bar{s}_k, \mathbb{X}^{k,\phi}_{\bar{s}_k},\phi(s)\big)ds\Bigg|&\le& C \Big(\mathbb{E}|T-T^k_{e(k,T)}|^{q}\mathds{1}_{\{T^k_{e(k,T)} < T\}}\Big)^{\frac{1}{q}}\\
\label{spd2ooo}&\le& C \epsilon_k^{\beta},
\end{eqnarray}
for $k\ge 1$. From (\ref{INTT1}), (\ref{spd2oo}) and (\ref{spd2ooo}), we get

$$
\sup_{\phi \in U^{k,\infty}_0}\mathbb{E}\| X^{k,\phi}_T - X^\phi_T\|_\infty \le C \Big\{ \epsilon^{\frac{\beta}{2}}_k + \epsilon_k^{2\theta-\frac{2}{p}(1-\beta)}\Big\},
$$
for $k\ge 1$. Next, we treat the case when $\sigma$ is independent from the control variable. Proposition \ref{semekT} yields

$$\sup_{\phi \in U^{T}_0}\mathbb{E}\| X^{k,\phi}_T - X^\phi_T\|_\infty\le C \Big\{ \epsilon_k + \epsilon_k^{2\theta-\frac{2}{p}(1-\beta)}\Big\} + C \sup_{\phi \in U^{T}_0}\mathbb{E}\| X^{k,\phi}_T - \mathbb{X}^{k,\phi}_T\|_\infty.$$
Similar to (\ref{spd1}), we have
\begin{equation}\label{spd3}
\int_{T^k_{e(k,T)}}^{\bar{t}_k}
\Sigma^{k}(s)dB(s) = \int_{\big(T^k_{e(k,T)},t\big]}
\Sigma^{k}(s)dB(s) - \int_{(\bar{t}_k,t]}
\Sigma^{k}(s)dB(s),
\end{equation}
for $T^k_{e(k,T)+1}\le t \le T$. We notice that since $\sigma$ is not a function of the control and $\Sigma^k$ (recall (\ref{Sigmawc})) is a stepwise constant $\mathbb{F}$-predictable process which jumps only at the stopping times $\{T^k_n; n\ge 1\}$, we have

\begin{equation}\label{spd4}
\int_{(\bar{t}_k,t]}
\Sigma^{k}(s)dB(s) = \sigma \big(\bar{t}_k, \mathbb{X}^{k,\phi}_{\bar{t}_k} \big) \big( B(t) - B(\bar{t}_k) \big),
\end{equation}
for every $t \in [0,T]$. By (\ref{spd3}) and (\ref{spd4}), we observe we can apply the same arguments given in the first case to state

 $$\sup_{\phi \in U^{T}_0}\mathbb{E}\| X^{k,\phi}_T - X^\phi_T\|_\infty\le C \epsilon_k^{\frac{\beta}{2}},$$
 for $k\ge 1$. This concludes the proof.

\subsection{Proof of Proposition \ref{mainPrSDEFBM}}

Recall that the controlled structure $\mathcal{X} = \big((X^k)_{k\ge 1},\mathscr{D}\big)$ associated with (\ref{limsdefbm}) is given as follows: For a given control $\phi\in U^{T}_0$,

\begin{eqnarray*}
\nonumber\mathbb{X}^{k,\phi}(T^k_{n})&=&\mathbb{X}^{k,\phi}(T^k_{n-1}) + \int_{T^k_{n-1}}^{T^k_n}\alpha\big(T^k_{n-1},\mathbb{X}^{k,\phi}_{T^k_{n-1}},\phi(s)\big)ds\\
& &\\
\nonumber&+& \sigma\Delta B^k_{H}(T^k_{n}),
\end{eqnarray*}
where  $\mathbb{X}^{k,\phi}(t) = \sum_{n=0}^\infty \mathbb{X}^{k,\phi}(T^k_n)\mathds{1}_{\{T^k_n \le t < T^k_{n+1}\}}; t\ge 0$. Recall
$$
X^{k,\phi}(t)=\mathbb{X}^{k,\phi}(t\wedge T^k_{e(k,T)}); 0\le t\le T.
$$
Let us denote

$$\widehat{\mathbb{X}}^{k,\phi}(t):= x + \int_0^t \alpha (\bar{s}_k, \mathbb{X}^{k,\phi}_{\bar{s}_k},\phi(s))ds + \sigma B^k_H(t); t\ge 0.$$
By construction,

%\begin{equation}\label{r1mult}
$$\mathbb{X}^{k,\phi}(t) =\mathbb{X}^{k,\phi}(\bar{t}_k) = \widehat{\mathbb{X}}^{k,\phi}(\bar{t}_k),~\mathbb{X}^{k,\phi}_t = \mathbb{X}^{k,\phi}_{\bar{t}_k},$$
%\end{equation}
for every $t\in[0,T]$. Recall the  notation (\ref{lessNOTA}). For $p\ge 1$, Theorem 3.5 in \cite{ohashifrancys} yields

\begin{equation}\label{raFBM}
\mathbb{E}\| B^k_H - B_H\|^p_\infty\lesssim \epsilon_k^{(pH)^-},
\end{equation}
for $k\ge 1$.

\

\noindent \textbf{Proof of Proposition \ref{mainPrSDEFBM}}: We fix $H \in (0,\frac{1}{2})$, $0 < \varepsilon < H$, $\phi\in U^T_0$ and the parameter $0< \theta\le 1$ from Assumption (C1). By invoking (\ref{raFBM}), the proof is similar and even simpler than Proposition \ref{mainPrSDEBM}. The idea is to analyse

$$\mathbb{E}\|\mathbb{X}^{k,\phi}_T- X^{\phi}_T \|_\infty\le \mathbb{E}\|\mathbb{X}^{k,\phi}_T - \widehat{\mathbb{X}}^{k,\phi}_T \|_\infty + \mathbb{E}\|\widehat{\mathbb{X}}^{k,\phi}_T- X^{\phi}_T\|_\infty, $$
and

\begin{eqnarray*}
\|X^{k,\phi}_T -  X^{\phi}_T\|_\infty&=&\|\mathbb{X}^{k,\phi} (\cdot \wedge T^k_{e(k,T)}) - X^{\phi} (\cdot \wedge T^k_{e(k,T)}) + X^{\phi}(\cdot\wedge T^k_{e(k,T)}) - X^{\phi} (\cdot)\|_\infty\\
&\le& \|\mathbb{X}^{k,\phi}_T - X^{\phi}_T\|_\infty \\
&+& \sup_{0\le t\le T}|X^{\phi}(t\wedge T^k_{e(k,T)}) - X^{\phi}(t)|.
\end{eqnarray*}
First, we claim
\begin{equation}\label{fbmest1}
\mathbb{E}\|\mathbb{X}^{k,\phi}_T - X^{\phi}_T\|_\infty \lesssim \epsilon_k^{(2\theta \wedge H)^-},
\end{equation}
for $k\ge 1$. One can check (\ref{fbmest1}) by using (\ref{raFBM}) and repeating exactly the same steps (with the obvious modification by replacing $A^k$ by $B^k_H$) as in the proofs of Lemmas \ref{passo1mult}, \ref{passo2mult}, \ref{passo3mult} and Proposition \ref{semekT}. Then, we omit the details. We observe

\begin{eqnarray}
\nonumber\sup_{0\le t\le T}|X^{\phi}(t\wedge T^k_{e(k,T)}) - X^{\phi}(t)|&\le& C\big(1+ \|X^{\phi}_T\|_\infty\big) |T-T^k_{e(k,T)}|\mathds{1}_{\{T^k_{e(k,T)} < T\}}\\
\label{fbmest2} &+& \sigma \|B_H(\cdot \wedge T^k_{e(k,T)}) - B_H\|_\infty =: I^k_1 + I^k_2.
\end{eqnarray}
One can easily check $\sup_{u\in U^T_0}\mathbb{E}\|X^u\|^{2p}_{\infty}\le C(1+|x_0|^{2p}\exp(CT))$ for every $p\ge 1$ where $C$ is a constant which depends on $\bar{a}$ and $B_H$. We shall invoke (\ref{astk2}) in Lemma \ref{Tkrate} and apply H\"older's inequality to get

\begin{equation}\label{fbmest3}
\mathbb{E}I^k_1\lesssim \epsilon_k^{1^-},
\end{equation}
for $k\ge 1$. Now, from  Garsia-Rodemich-Rumsey's inequality (see e.g Lemma A.3.1 in \cite{nualart}), for all $\varepsilon\in (0,H)$ there exists a nonnegative random variable $G_{\varepsilon,T}$ with $\mathbb{E}|G_{\varepsilon,T}|^p < + \infty$ for all $p\ge 1$ such that

$$|B_H(t) - B_H(s)|\le G_{\varepsilon,T}|t-s|^{H-\epsilon}~a.s,$$
for all $s,t\in [0,T]$. Therefore,

\begin{equation}\label{fbmest4}
\sup_{0\le t\le T}|B_H(t) - B_H(t\wedge T^k_{e(k,T)})|\le |T - T^k_{e(k,T)}|^{H-\epsilon }G_{\varepsilon,T}\mathds{1}_{\{T^k_{e(k,T)} < T\}}.
\end{equation}
From (\ref{fbmest4}), Lemma \ref{Tkrate} and H\"older's inequality, we get

\begin{equation}\label{fbmest5}
\mathbb{E}I^k_2\lesssim  \epsilon_k^{(H-\varepsilon)^-},
\end{equation}
for $k\ge 1$.

Summing up (\ref{fbmest1}), (\ref{fbmest2}), (\ref{fbmest3}) and (\ref{fbmest5}), we conclude the proof.

\subsection{Proof of Proposition \ref{mainPrRV}}
In the sequel, we recall the processes $Z^k$, $W^k_H$ and $\mathbb{X}^{k,\phi}$ given by (\ref{discreteroughOU}), (\ref{discreterough1}) and (\ref{Seuler}), respectively. Clearly, there exists a constant $C$ which only depends on $H$ such that

\begin{eqnarray}
\nonumber\| Z_T- Z^k_T\|_{\infty}&\le& C \Bigg\{|\beta \nu|\big(\max_{n\ge 0} \Delta T^k_{n+1}\mathds{1}_{\{T^k_n\le T\}}\big) + \nu \|W^k_H  - W_H\|_{\infty}\\
\nonumber& &\\
\nonumber &+& |2\beta^2\nu| \big(\max_{n\ge 0} \Delta T^k_{n+1}\mathds{1}_{\{T^k_n\le T\}}\big) \|W^k_H\|_{\infty}e^{\beta T}\\
\label{pathwiseZ_k}& &\\
\nonumber&+& e^{\beta T}\|W^k_H - W_H\|_{\infty} + \beta \|W_H\|_{\infty}e^{\beta  T}\big(\max_{n\ge 0} \Delta T^k_{n+1}\mathds{1}_{\{T^k_n\le T\}}\big)\Bigg\}~a.s.
\end{eqnarray}
Lemma \ref{meshlemma} and Theorem 3.5 in \cite{ohashifrancys} applied to (\ref{pathwiseZ_k}) allow us to state the following result.

\begin{lemma}\label{Zkcon}
Fix $0 < H < \frac{1}{2}$ and $p\ge 1$. Then,

$$
\mathbb{E}\|Z^k_T-Z_T\|^p_\infty \lesssim
\epsilon_k^{(pH)^-}
$$
for $k\ge 1$.
\end{lemma}

For a given control $\phi \in U^{T}_0$, we set

$$\overline{\Sigma}^{k,\phi}(t):= 0 \mathds{1}_{\{t=0\}} + \sum_{n=1}^\infty \mathbb{X}^{k,\phi}(T^{k}_{n-1})\vartheta\big(Z^k(T^k_{n-1}),\phi(t)\big)\mathds{1}_{\{T^{k}_{n-1} < t \le T^{k}_n\}},
$$
for $0\le t\le T$. We define

$$
\widehat{\mathbb{X}}^{k,\phi}(t):=x_0 + \int_0^t \mathbb{X}^{k,\phi}(\bar{s}_k)\mu(\phi(s))ds + \int_0^t \overline{\Sigma}^{k,\phi}(s)dB^{1}(s),
$$
for $0\le t\le T$. One should notice that

$$
\mathbb{X}^{k,\phi}(t) =\mathbb{X}^{k,\phi}(\bar{t}_k) = \widehat{\mathbb{X}}^{k,\phi}(\bar{t}_k),~\mathbb{X}^{k,\phi}_t = \mathbb{X}^{k,\phi}_{\bar{t}_k},
$$
for every $t\in[0,T]$. The idea is to analyse

$$\mathbb{E}\|\mathbb{X}^{k,\phi} - X^{\phi}\|_\infty\lesssim \mathbb{E}\|\mathbb{X}^{k,\phi}- \widehat{\mathbb{X}}^{k,\phi}\|^2_\infty + \mathbb{E}\|\widehat{\mathbb{X}}^{k,\phi}- X^{\phi} \|_\infty. $$
By using Assumption (D1), one can follow the same arguments as described in Lemmas \ref{passo1mult}, \ref{passo2mult}, \ref{passo3mult}, Proposition \ref{semekT} and the proof of Proposition \ref{mainPrSDEBM}.

\begin{lemma}\label{bsvol1}
Under Assumption (D1), we have

$$
\sup_{k\ge 1}\sup_{\phi \in U^T_0}\mathbb{E}\| \mathbb{X}^{k,\phi}_T\|^p_\infty < \infty,
$$
for every $p\ge 1$.
\end{lemma}

%\begin{proof}
%We know that $\mathbb{X}^{k,\phi}(t) = \widehat{\mathbb{X}}^{k,\phi}(\bar{t}_k)$ for $0\le t\le T$. The proof consists on routine arguments based on Burkholder-Davis-Gundy and Jensen inequalities jointly with Gr\"{o}nwall's inequality on the function $s\mapsto \mathbb{E}\|\mathbb{X}^{k,\phi}_s\|^p_\infty$. Indeed, recall

%$$\mathbb{X}^{k,\phi}(t) = x_0 + \int_0^{\bar{t}_k}\mathbb{X}^{k,\phi}_{\bar{s}_k} \mu(\phi(s))ds +\int_0^{\bar{t}^k}\overline{\Sigma}^{k,\phi}(s)dB^1(s).$$

%Fix $p \ge 2$. By applying Jensen's ineguality and using Assumption (D1), we get

%$$
%\mathbb{E}\sup_{0\le t\le T}\Bigg|\int_0^{\bar{t}_k} \mathbb{X}^{k,\phi}(s) \mu(\phi(s))ds\Bigg|^p\le C\Bigg(1+\int_0^T\mathbb{E}\|\mathbb{X}^{k,\phi}_{s}\|^p_\infty ds\Bigg).
%$$
%Burkholder-Davis-Gundy and Jensen inequalities and Assumption (D1) yield

%$$
%\mathbb{E}\sup_{0\le t\le T}\Bigg|\int_0^{t}\overline{\Sigma}^{k,\phi}(s)dB^1(s)\Bigg|^p\le C \mathbb{E}\Bigg(\int_0^T|\overline{\Sigma}^{k,\phi}(s)|^2ds\Bigg)^{\frac{p}{2}}
%$$
%$$\le C\mathbb{E}\int_0^T|\overline{\Sigma}^{k,\phi}(s)|^pds\le C \Big(1+\int_0^T\mathbb{E}\|\mathbb{X}^{k,\phi}_s\|^p_\infty ds\Big).$$

%Summing up the above estimates, we have

%$$\mathbb{E}\|\mathbb{X}^{k,\phi}_T\|^p_\infty\le C\Bigg(1+ \int_0^T\mathbb{E}\|\mathbb{X}^{k,\phi}_{s}\|^p_\infty ds \Bigg).$$
%Grownall's inequality and the fact that $C$ does not depends on on $\phi$ and $k\ge 1$ allow us to conclude the result.
%\end{proof}

The proof of Lemma \ref{bsvol1} is identical to the proof of Lemma \ref{passo1mult}, so we omit the details.

\begin{lemma}\label{p2svol}
Fix $0 < H < \frac{1}{2}$. Under Assumption (D1), we have

%there exists a constant $C$ which only depends on $H$, $\mu, \vartheta$ and $T$ such that

\begin{equation*}
\begin{split}
\mathbb{E}\|\widehat{\mathbb{X}}^{k,\phi}_T- X^{\phi}_T\|^2_\infty \lesssim \Bigg\{ \epsilon^{(2H)^-}_k  + \int_0^T\mathbb{E}\|\mathbb{X}^{k,\phi}_{s}- X^{\phi}_s\|^2_\infty ds \Bigg\},\\
\end{split}
\end{equation*}
for every $\phi \in U^{T}_0$ and $k\ge 1$.
\end{lemma}
\begin{proof}
By definition,

\begin{eqnarray*}
\widehat{\mathbb{X}}^{k,\phi}(t) - X^\phi(t) &=& \int_0^t \Big[ (\mathbb{X}^{k,\phi}(s)-X^\phi(s))\mu(\phi(s))\Big]ds\\
&+& \int_0^t \Big[\overline{\Sigma}^{k,\phi}(s)-X^\phi(s)\vartheta(Z(s),\phi(s))\Big]dB^1(s)
\end{eqnarray*}
Assumption (D1) yields the existence of a constant $C$ such that $\sup_{u \in U^T_0}\|\mu(u)\|_\infty \le C$ a.s. Therefore,

$$\mathbb{E}\sup_{0\le t\le T}\Bigg|\int_0^t \big(\mathbb{X}^{k,\phi}(s)-X^\phi(s)\big)\mu(\phi(s))ds\Bigg|^2$$
$$\le C \mathbb{E}\int_0^T \big|\big(\mathbb{X}^{k,\phi}(s)-X^\phi(s)\big)\mu(\phi(s))\big|^2ds$$
$$\le C \mathbb{E}\int_0^T \|\mathbb{X}^{k,\phi}_s-X^\phi_s\|^2_\infty ds,$$
for every $k\ge 1$. By writing

\begin{eqnarray*}
\overline{\Sigma}^{k,\phi}(s) - X^\phi(s) \vartheta \big(Z(s),\phi(s)\big)&=& \mathbb{X}^{k,\phi}(T^k_{n-1}) \Big[\vartheta(Z^k(T^k_{n-1}),\phi(s))  - \vartheta(Z(s),\phi(s))\Big]\\
&+& \big[\mathbb{X}^{k,\phi}(T^k_{n-1}) - X^\phi(s)\big]\vartheta \big(Z(s),\phi(s)\big)\\
\end{eqnarray*}
on each stochastic interval of form $T^k_{n-1} < s < T^k_n$, we can apply Assumption (D1) to get

\begin{eqnarray}
\nonumber\Big|\overline{\Sigma}^{k,\phi}(s) - X^\phi(s) \vartheta \big(Z(s),\phi(s)\big)\Big|^2&\le & C \|\mathbb{X}^{k,\phi}_T\|_\infty^2 \| Z^k_T - Z_T\|_\infty^2 \\
\label{Zkop1}&+&C \|\mathbb{X}^{k,\phi}_s - X^\phi_s\|^2_\infty,
\end{eqnarray}
on each stochastic interval of form $T^k_{n-1} < s < T^k_n$. From Lemma \ref{Zkcon}, we know that for any $p>1$, we have

\begin{equation}\label{Zkop2}
\big(\mathbb{E}\|Z^k_T-Z_T\|^{2p}_\infty\big)^{\frac{1}{p}} \lesssim
\epsilon_k^{(2H)^-},
\end{equation}
for every $k\ge 1$. By applying Burkholder-Davis-Gundy and H\"older's inequalities and using (\ref{Zkop1}), (\ref{Zkop2}) and Lemma \ref{bsvol1}, we then get

$$\mathbb{E}\sup_{0\le t\le T}\Bigg|\int_0^t \Big[\overline{\Sigma}^{k,\phi}(s)-X^\phi(s)\vartheta(Z(s),\phi(s))\Big]dB^1(s)\Bigg|^2$$
$$\le C  \mathbb{E}\int_0^T \Big|\overline{\Sigma}^{k,\phi}(s) - X^\phi(s) \vartheta \big(Z(s),\phi(s)\big)\Big|^2ds
$$
$$\lesssim \epsilon_k^{(2H)^-} + \int_0^T \mathbb{E}\| \mathbb{X}^{k,\phi}_s-X^\phi_s\|_\infty^2ds,$$
for every $k\ge 1$. All the constants in the above inequalities do not depend on $k\ge 1$ and $\phi \in U^T_0$. This concludes the proof.

%Assumption (D1) and Lemma yield$$ \sup_{k\ge 1}\sup_{u \in U^T_0}\mathbb{E}\| \overline{\Sigma}^{k,u}_T\|^p_\infty\le C \sup_{k\ge 1}\sup_{u \in U^T_0}\mathbb{E}\| \mathbb{X}^{k,u}_T\|^p_\infty< \infty.$$
%We can then apply Th 1 in \cite{fischer} and repeat the same arguments given in the proof of Lemma \ref{passo3mult}

\end{proof}

\begin{lemma}\label{p3svol}
Under Assumption (D1), there exists a constant $C$ such that

$$\sup_{\phi \in U^{k,\infty}_0}\mathbb{E}\|\mathbb{X}^{k,\phi}_T - \widehat{\mathbb{X}}^{k,\phi}_T\|^2_\infty \le C \epsilon_k^2, $$
for every $k\ge 1$. In case $\vartheta$ is independent from the control variable, there exists a constant $C$ such that

$$\sup_{\phi \in U^{T}_0}\mathbb{E}\|\mathbb{X}^{k,\phi}_T - \widehat{\mathbb{X}}^{k,\phi}_T\|^2_\infty \le C \epsilon_k^2, $$
for every $k\ge 1$.
\end{lemma}
The proof of Lemma \ref{p3svol} is identical to the proof of Lemma \ref{passo3mult}, so we omit the details. By combining Lemmas \ref{p2svol} and \ref{p3svol} and applying H\"older and Grownall's inequality, we get

%\begin{proof}
%Fix a control $\phi \in U^{k,\infty}_0$. We observe
%\begin{eqnarray*}
%\widehat{\mathbb{X}}^{k,\phi}(t) - \mathbb{X}^{k,\phi}(t) &=& \int_{\bar{t}_k}^t \mathbb{X}^{k,\phi}(s)\mu (\phi(s))ds + \int_{\bar{t}_k}^t \overline{\Sigma}^{k,\phi}(s)dB^1(s)\\
%&=& \mathbb{X}^{k,\phi}(\bar{t}_k)\mu (\phi(\bar{t}_k)) (t-\bar{t}_k) + \mathbb{X}^{k,\phi}(\bar{t}_k)\vartheta \big(Z^k(\bar{t}_k),\phi(\bar{t}_k)\big) \big(B^1(t) - B^1(\bar{t}_k)\big),
%\end{eqnarray*}
%for $0\le t\le T$ and $k\ge 1$. Hence, Assumption (D1), H\"older's inequality, Lemmas \ref{bsvol1} and \ref{meshlemma} yield

%$$\sup_{\phi \in U^{k,\infty}_0}\mathbb{E}\|\widehat{\mathbb{X}}^{k,\phi}_T - \mathbb{X}^{k,\phi}_T\|^2_\infty\le C \epsilon^2_k,$$
%for $k\ge 1$. Similar analysis can be made if $\vartheta$ does not depend on controls and we may also conclude
%$$\sup_{\phi \in U^{T}_0}\mathbb{E}\|\widehat{\mathbb{X}}^{k,\phi}_T - \mathbb{X}^{k,\phi}_T\|^2_\infty\le C \epsilon^2_k,$$
%for $k\ge 1$.
%\end{proof}

\begin{equation}\label{rvdx1}
\sup_{\phi \in U^{k,\infty}_0}\mathbb{E}\|\mathbb{X}^{k,\phi}_T - X^{\phi}_T\|_\infty\lesssim \epsilon^{H^-}_k,
\end{equation}
for $k\ge 1$. In case $\vartheta$ does not depend on controls, we also have

\begin{equation}\label{rvdx2}
\sup_{\phi \in U^{T}_0}\mathbb{E}\|\mathbb{X}^{k,\phi}_T - X^{\phi}_T\|_\infty\lesssim \epsilon^{H^-}_k,
\end{equation}
for $k\ge 1$.

We are now able to present the proof of Proposition \ref{mainPrRV}.

\

\noindent \textbf{Proof of Proposition \ref{mainPrRV}}: The argument of the proof is identical to the proof of Proposition \ref{mainPrSDEBM}. Indeed, for each $\phi\in U^{k,\infty}_0$, we shall write
$$X^{k,\phi} - X^\phi = X^{k,\phi} - \mathbb{X}^{k,\phi} + \mathbb{X}^{k,\phi} - X^\phi.$$
We can apply (\ref{rvdx1}) to state that

$$\sup_{\phi \in U^{k,\infty}_0}\mathbb{E}\| X^{k,\phi}_T - X^\phi_T\|_\infty\lesssim \epsilon_k^{H^-} + \sup_{\phi \in U^{k,\infty}_0}\mathbb{E}\| X^{k,\phi}_T - \mathbb{X}^{k,\phi}_T\|_\infty,$$
for $k\ge 1$. For each $\phi \in U^{k,\infty}_0$, $\mathbb{X}^{k,\phi}(t) - X^{k,\phi}(t) = 0$ whenever $0\le t < T^k_{e(k,T)+1}$ and

\begin{eqnarray*}
\mathbb{X}^{k,\phi}(t) - X^{k,\phi}(t)&=& \int_{T^k_{e(k,T)}}^{\bar{t}_k} \mathbb{X}^{k,\phi}(\bar{s}_k)\mu(\phi(s))ds\\
&+& \int_{T^k_{e(k,T)}}^{\bar{t}_k}
\overline{\Sigma}^{k,\phi}(s)dB^1(s),
\end{eqnarray*}
 whenever $T^k_{e(k,T)+1}\le t \le T$. By using Assumption (D1) and repeating the same steps given in the proof of Proposition \ref{mainPrSDEBM}, one can check

$$
\sup_{\phi \in U^{k,\infty}_0}\mathbb{E}\| X^{k,\phi}_T - X^\phi_T\|_\infty \lesssim  \epsilon_k^{H^-} + \epsilon_k^{\frac{1}{2}^-} \lesssim \epsilon_k^{H^-},
$$
for $k\ge 1$. In case $\vartheta$ does not depend on controls, we use (\ref{rvdx2}) and the same argument given in the proof of Proposition \ref{mainPrSDEBM} yields

$$
\sup_{\phi \in U^{T}_0}\mathbb{E}\| X^{k,\phi}_T - X^\phi_T\|_\infty \lesssim \epsilon_k^{H^-},
$$
for $k\ge 1$. This completes the proof.

\section{Numerical example and a pseudo code}\label{numersection}
%At first,
%we present a pseudo-code for a given stochastic control problem. For simplicity, we set the dimension of the Brownian motion equals to $2$.
%Let $m:=\Big\lceil\frac{e(k,T)}{\mathbb{E}(\tau_1\wedge \tau_2)}\Big\rceil$, where $\tau_i=\inf\{t:|B^i(t)|=1\}$.

In this section, we explain how the methodology can be implemented in a concrete simple example. For this purpose, we choose the example of hedging in a two-dimensional Black-Scholes model. This is a classical problem in Finance which can be briefly described as follows. For a given $c\in \mathbb
{R}$ and a Lipschitz function $\varphi: \mathbb{R}^2\rightarrow\mathbb{R}$, we define $\varrho_c:\mathbb{R}^3\rightarrow \mathbb{R}$ by

$$\varrho_c(x,y,z):=(c+ x -\varphi(y,z))^2; (x,y,z)\in \mathbb{R}^3.$$
Let us consider
\begin{equation}\label{Seq}
\begin{split}
dP^1(t)&=P^1(t)\left(\mu_1dt+\sigma_1dB^1(t)\right)\\
dP^2(t)&=P^2(t)\left(\mu_2dt+\sigma_2dB^2(t)\right)
\end{split}
\end{equation}
where, for simplicity, we assume $[B^1,B^2]=0, \mu_1=\mu_2=0, T=1$ and the riskless rate equals zero. The problem is

$$
\text{minimize}\quad \mathbb{E}\Big[\varrho_c\big(Y^\phi(T),P^1(T), P^2(T)\big)\Big]\quad \text{over all}~\phi\in U^T_0,~c\in \mathbb{R},
$$
where

$$
Y^\phi(t)=\sum_{j=1}^2\int_0^t \phi_j(r)dP^j(r); \phi\in U^T_0, 0\le t\le T,
$$
and the controls $\phi(t) = (\phi_1(t),\phi_2(t))$ represent the absolute percentages of the securities $(P^1,P^2)$ which an investor holds at time $t\in [0,T]$. In this example, we choose $\varphi(y,z):=\max{\left(y-z,0\right)}$ and $\bar{a}=1$. It is well-known there exists a unique choice of $(c^*,\phi^*) \in \mathbb{R}\times U^T_0$ such that

$$\inf_{(c,\phi)\in \mathbb{R}\times U^T_0}\mathbb{E}\Big[\varrho_c\big(Y^\phi(T),P^1(T), P^2(T)\big)\Big] = \mathbb{E}\Big[\varrho_{c^*}\big(Y^{\phi^*}(T),P^1(T), P^2(T)\big)\Big]=0,$$
where by Margrabe's formula, we have

$$c^*=P^1_0\Phi(d_1)-P^2_0\Phi(d_2),$$
where

$$\sigma=\sqrt{\sigma_1^2+\sigma_2^2},\quad d_1=\frac{\log\left(\frac{P^1(0)}{P^2(0)}\right)+\frac{\sigma^2}{2}T}{\sigma\sqrt{T}},\quad d_2=d_1-\sigma\sqrt{T},$$
and $\Phi$ is the cumulative distribution function of the standard Gaussian variable. We recall $c^*$ is the price of the option and $\phi^*$ is the so-called delta hedging which can be computed by means of the classical PDE Black-Scholes as a function of $\Phi$. We set $P^1(0)=49, P^2(0)=52, \sigma_1=0.2$, $\sigma_2=0.3$ and $\epsilon_k = 2^{-k}$. The controlled process $X^\phi:= (P, Y^\phi)$ is the 3-dimensional SDE given by

$$dX^{\phi,i}(t) = \sum_{j=1}^2 \sigma^{ij}(t,X^\phi_t,\phi(t))dB^j(t),~i=1,2,3$$
where

  \begin{equation*}
  \sigma^{ij}(t,f_t,a):=\left\{
\begin{array}{rl}
\sigma_i f_i(t)\delta_{i=j}; & \hbox{if} \ i=1,2; j=1,2 \\
a_j \sigma_jf_j(t);& \hbox{if} \ i=3; j=1,2.
\end{array}
\right.
\end{equation*}
Here, $\delta_{i=j}$ is the delta Dirac function concentrated at $i=j$. In order to keep notation simple, we set $m=e(k,T)$. The controlled imbedded discrete structure for $X^\phi$ is given by $X^{k,u^k} = \big(P^{k}, Y^{k,u^k}\big)$, where
$$Y^{k,u^k}(\bar{t}_k) = \sum_{j=1}^2 \int_0^{\bar{t}_k} u^{k,j}(s)dP^{k,j}(s)$$
for $u^k = (u^k_0, u^k_1, \ldots, u^k_{m-1}) \in U^{k,m}_0$. Here, $P^k=(P^{k,1}, P^{k,2})$ is the Euler scheme of (\ref{Seq}). By definition,

\begin{eqnarray*}
\Delta Y^{k,u^k}(T^k_j)&=& u^{k,1}_{j-1} \Delta P^{k,1}(T^k_j) + u^{k,2}_{j-1} \Delta P^{k,2}(T^k_j)\\
&=&  u^{k,1}_{j-1}\sigma_1 P^{k,1}(T^k_{j-1})\Delta A^{k,1}(T^k_j) + u^{k,2}_{j-1} \sigma_2 P^{k,2}(T^k_{j-1}) \Delta A^{k,2}(T^k_j)
\end{eqnarray*}
for $j=m,\ldots, 1$. In the sequel, in order to keep notation simple, we write $Y^{k,u^k}_j = Y^{k,u^k}(T^k_j), P^k_j = P^k(T^k_j)$, $\Delta P^{k,i}(T^k_j) = \Delta P^{k,i}_j = P^{k,i}_j - P^{k,i}_{j-1}$ and $\Delta Y^{k,u^k}(T^k_j) = \Delta Y^{k,u^k}_j = Y^{k,u^k}_j - Y^{k,u^k}_{j-1}$, for $j=m,\ldots, 1$ and $i=1,2$. We write

\begin{eqnarray*}
H &=& \max \big\{P^{k,1}_m - P^{k,2}_m; 0\big\}\\
&=& \max \{P^{k,1}_{m-1} - P^{k,2}_{m-1} + \sigma_1 P^{k,1}_{m-1} \Delta A^{k,1}(T^k_m) - \sigma_2 P^{k,2}_{m-1} \Delta A^{k,2}(T^k_m);0  \}.
\end{eqnarray*}

Fix $c \in \mathbb{R}$. By writing the value function in terms of conditional expectation $\mathbb{E}_{j}$ w.r.t $\mathcal{F}^k_{T^k_{j}}$, for $j=m-1, \ldots, 0$, we have

\begin{eqnarray*}
V^k(T^k_{m-1},u^k)&=&\inf_{a_{m-1}\in \mathbb{A}}\mathbb{E}_{m-1}\big[\big(c+Y^{k,u^k}_{m-1} +\Delta Y^{k,a_{m-1}}_m-H\big)^2\big]\\
&=& \inf_{a_{m-1}\in \mathbb{A}}\Bigg\{\mathbb{E}_{m-1}\big[(c+Y^{k,u^k}_{m-1}-H\big)^2\big]+2a_{1,m-1}\mathbb{E}_{m-1}\big[ \big(c+Y^{k,u^k}_{m-1}-H\big)\Delta P^{k,1}_m\big]\\
&+& a^2_{1,m-1}\mathbb{E}_{m-1} \big[ \big( \Delta P^{k,1}_m \big)^2\big] +2a_{2,m-1}\mathbb{E}_{m-1}\big[ \big( c+Y^{k,u^k}_{m-1}-H\big)\Delta P^{k,2}_m\big]\\
&+& a^2_{2,m-1}\mathbb{E}_{m-1} \big[\big(\Delta P^{k,2}_m\big)^2\big] + 2a_{1,m-1}a_{2,m-1}\mathbb{E}_{m-1}\big[\Delta P^{k,2}_m\Delta P^{k,1}_m\big]  \Bigg\}.
\end{eqnarray*}

By computing the gradient $\nabla$ of $ a_{m-1}\mapsto \mathbb{E}_{m-1}\big[\big(c+Y^{k,u^k}_{m-1} +\Delta Y^{k,a_{m-1}}_m-H\big)^2\big]$, solving the system $\nabla \mathbb{E}_{m-1}\big[\big(c+Y^{k,u^k}_{m-1} +\Delta Y^{k,a_{m-1}}_m-H\big)^2\big] (a^\star_{m-1}) = (0,0)$ and using the strong Markov property $\mathbb{E}_{m-1}\big[\Delta A^{k,1}(T^k_m)\Delta A^{k,2}(T^k_m) \big] = 0$, we get

$$
a^\star_{1,m-1}= \frac{\mathbb{E}_{m-1}\big[H\Delta P^{k,1}_m\big]}{ \big(P^{k,1}_{m-1}\big)^2\sigma_1^2\epsilon^2_k},\quad  a^\star_{2,m-1}=\frac{\mathbb{E}_{m-1}\big[H\Delta P^{k,2}_m\big]}{ \big(P^{k,2}_{m-1}\big)^2\sigma_2^2\epsilon^2_k}.$$

Then, by using the optimal $a^\star_{m-1} = ( a^\star_{1,m-1}, a^\star_{2,m-1})$, the tower property of the conditional expectation, the dynamic programming equation (\ref{DPEexp}) in Theorem \ref{absrate1} and Proposition \ref{AGREGATION}, we have

$$V^k(T^k_{m-2},u^k) = \inf_{a_{m-2} \in \mathbb{A}} \mathbb{E}_{m-2}\Big[ \Big( c + Y^{k,u^k}_{m-2} + \Delta Y^{k,a_{m-2}}_{m-1} + \Delta Y^{k,a^\star_{m-1}}_m -H\Big)^2\Big].$$
In the second step, we proceed in the same way as in the first step. By iterating this procedure, we get

\begin{equation*}
\begin{split}
a^\star_{1,m-i}&=-\frac{\sum_{j=0}^{i-1}\mathbb{E}_{m-i}\left[\left(a^\star_{1,m-j}\Delta P^{k,1}_{m+1-j}+ a^\star_{2,m-j}\Delta P^{k,2}_{m+1-j}-H\right)\Delta P^{k,1}_{m-i+1} \right]}{\big(P^{k,1}_{m-i}\big)^2\sigma_1^2\epsilon^2_k}\\
&\\
a^\star_{2,m-i}&= -\frac{\sum_{j=0}^{i-1}\mathbb{E}_{m-i}\left[\left(a^\star_{1,m-j}\Delta P^{k,1}_{m+1-j}+ a^\star_{2,m-j}\Delta P^{k,2}_{m+1-j}-H\right)\Delta P^{k,2}_{m-i+1} \right] }{\big(P^{k,2}_{m-i}\big)^2\sigma_2^2\epsilon^2_k},\\
\end{split}
\end{equation*}
for $i=1, \ldots, m$, with $a^\star_{1,m}:=a^\star_{2,m}:=0$. By the dynamic programming principle (see Proposition \ref{epsiloncTH}), $u^\star:= \big( a^\star_0, a^\star_1, \ldots, a^{\star}_{m-1}\big)$ realizes

$$\inf_{\phi \in U^{k,m}_0} \mathbb{E}\Big[\varrho_c \big(Y^{k,\phi}_m, P^{k,1}_m,P^{k,2}_m\big)  \Big] = \mathbb{E}\Big[\varrho_c \big(Y^{k,u^\star}_m, P^{k,1}_m,P^{k,2}_m\big)  \Big].$$
The estimated value $c^{k,\star}$ is computed according to

\begin{equation*}
\begin{split}
c^{k,\star}\in \argmin_{c\in \mathbb{R}}\mathbb{E}\left[\varrho_c\left(Y^{k,u^\star}_m,P^{k,1}_m, P^{k,2}_m\right)\right].
\end{split}
\end{equation*}
In other words, $\mathbb{E}\left[\varrho_{c^{k,\star}}\left(Y^{k,u^\star}_m,P^{k,1}_m, P^{k,2}_m\right)\right]=0$. Table \ref{001tab} presents a comparison between the true call option price $c^\star$ and the associated Monte Carlo price $c^{k,\star}$. Figure 1 presents the Monte Carlo experiments for $c^{k,\star}$ with $k=1,2$ and $3$. The number of Monte Carlo iterations in the experiment is $3 \times 10^4$.

%At first, for a given $c\in \mathbb{R}$, we apply the algorithm described below to get a Monte Carlo feedback (see Proposition \ref{openclosed}) optimal control approximation $\phi^{k,*}$ related to (\ref{numexfinal}). In this particular case, we stress one can analytically solve (\ref{pseudoeq1}) and the estimated value $c^{k,*}$ is computed according to
%\begin{equation*}
%\begin{split}
%c^{k,*}\in \argmin_{c\in \mathbb{R}}\mathbb{E}\left[\varrho_c\left(X^k(T,\phi^{k,*}),S^{k,1}(T\wedge T^k_{e(k,T)}), S^{k,2}(T\wedge T^k_{e(k,T)})\right)\right].
%\end{split}
%\end{equation*}
%In other words, $\mathbb{E}\left[\varrho_{c^{k,*}}\left(X^k(T,\phi^{k,*}),S^{k,1}(T\wedge T^k_{e(k,T)}), S^{k,2}(T\wedge T^k_{e(k,T)})\right)\right]=0$. Here

%$$X^{k}(t,\phi^k) = \mathbb{X}^k(t\wedge T^k_{e(k,T)},\phi^k),\quad \mathbb{X}^{k}(t,\phi^k)= \sum_{j=1}^2\int_0^{\bar{t}_k}\phi^k_j(r)dS^{k,j}(r); \phi^k\in U^{k,e(k,T)}_0,$$
%and $S^{k,i}(t)  = \mathbb{S}^{k,i}(t\wedge T^k_{e(k,T)})$, where $\mathbb{S}^k = (\mathbb{S}^{k,1}, \mathbb{S}^{k,2})$ follows (\ref{Xeuler}) (without the presence of controls) with the coefficients $\alpha^i(t,f) = \mu_i f(t), \sigma^{ij}(t,f)= \sigma_i f(t)\delta_{i=j}$ for $i,j=1,2$, where $\delta_{i=j}$ is the delta Dirac function concentrated at $i=j$. Table \ref{001tab} presents a comparison between the true call option price $c^*$ and the associated Monte Carlo price $c^{k,*}$. Figure 1 presents the Monte Carlo experiments for $c^{k,*}$ with $k=1,2$ and $3$. The number of Monte Carlo iterations in the experiment is $3 \times 10^4$.

\begin{table}[!h]
\caption{Comparison between $c^{\star}$ and $c^{k,\star}$ for $\epsilon_k = 2^{-k}$}
\begin{tabular}{lccccc}
  \hline
  % after \\: \hline or \cline{col1-col2} \cline{col3-col4} ...
  \textbf{k} & \textbf{Result} & \textbf{Mean Square Error} & \textbf{True Value} & \textbf{Difference} & \textbf{\% Error}\\
  \hline
  1 & 5.9740 & 0.01689567 & 5.821608 & 0.152458 & 0.0261\%\\
  2 & 5.8622 & 0.01158859 & 5.821608 & 0.04059157 & 0.0069\%\\
  3 & 5.7871 & 0.00821813 & 5.821608 & 0.03441365 & 0.0059\%\\
  \hline
\end{tabular}
\label{001tab}
\end{table}

\begin{figure}[!h]\label{1fig}
\center
\subfigure[ref1][]{\includegraphics[width=0.45\textwidth]{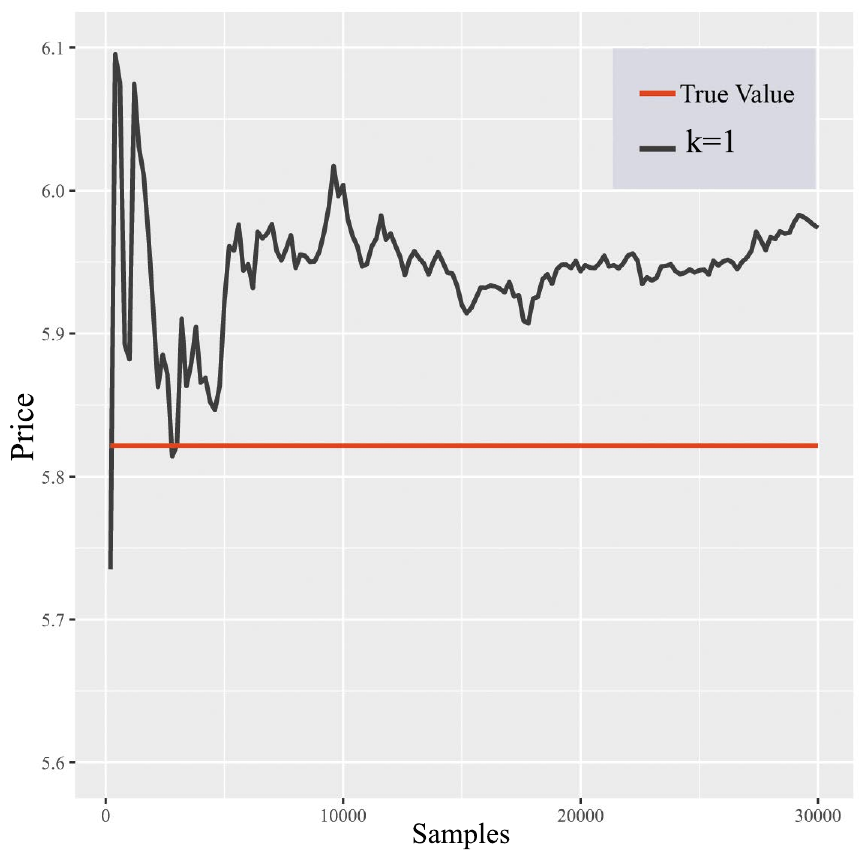}}
\qquad
\subfigure[ref2][]{\includegraphics[width=0.45\textwidth]{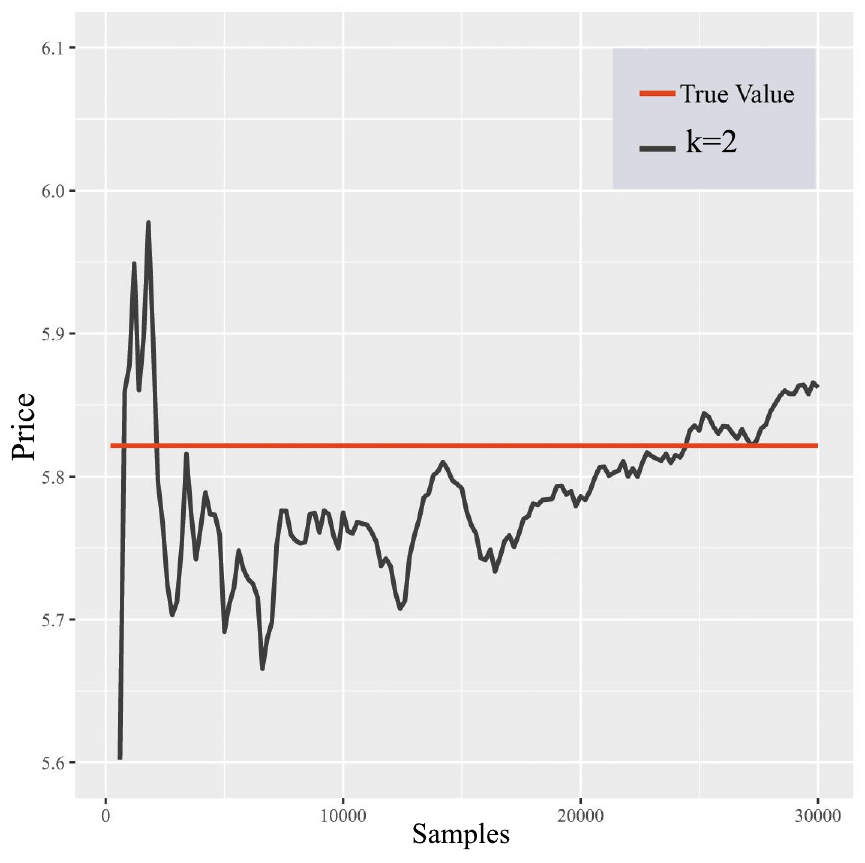}}
\qquad
\subfigure[ref3][]{\includegraphics[width=0.45\textwidth]{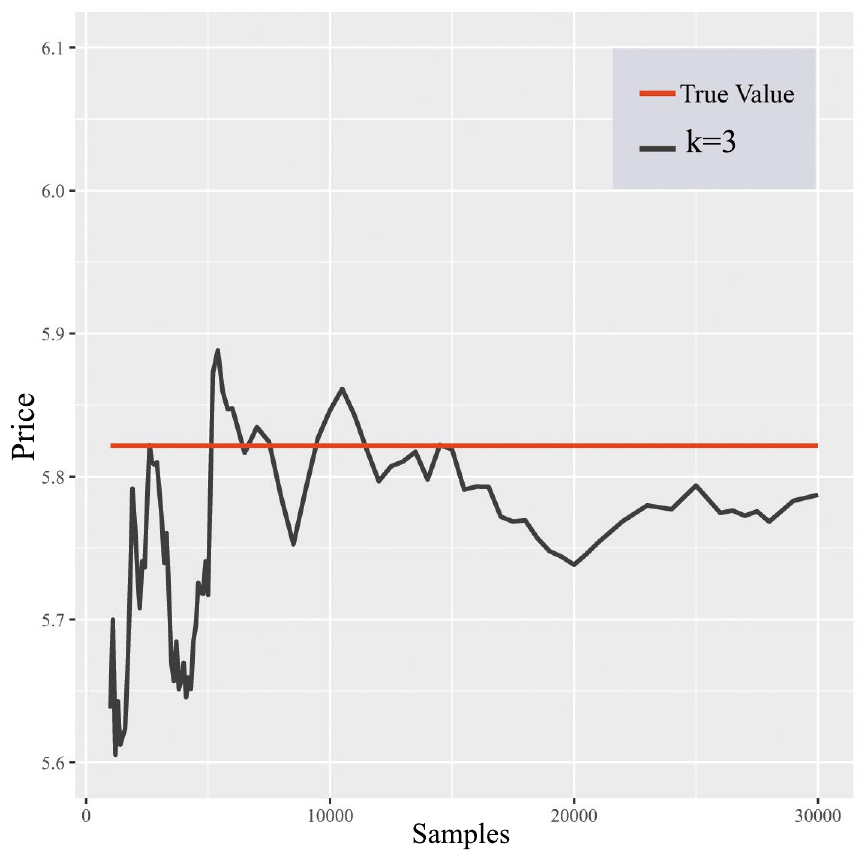}}
\caption{Monte Carlo experiments for $c^{k,\star}$ with $\epsilon_k = 2^{-k}$}
\end{figure}
%where $z$ is a truncated normal distribution with parameters $(0,\Delta T^k_\ell,-\epsilon_k,\epsilon_k)$;
\begin{algorithm}[H]
\label{alg1}
 \KwData{ Level of discretization $\epsilon_k$, number of periods $m:=\Big\lceil\frac{\epsilon^{-2}_kT}{\chi_2}\Big\rceil$, $\bar{a}>0$ and a Bernoulli-(1/2) distribution $Q$ with support in $\{-1,1\}$. }
 \KwResult{Vector of optimal control $u^{k,\star}$ given by (\ref{explicitcontrolTEXT}).}
 initialization\;
 \For{$\ell\leftarrow 1$ \KwTo $m$}{

 Generate $\Delta^{k,1}_{\ell}$ and $\Delta^{k,2}_{\ell}$ according to the algorithm described in \cite{Burq_Jones2008}.

 Compute $\Delta T^{k}_\ell=\min\{\Delta^{k,1}_\ell, \Delta^{k,2}_\ell\}$.

 \eIf{$\min\{ \Delta^{k,1}_{\ell},\Delta^{k,2}_{\ell}\}=\Delta^{k,1}_\ell$}{generate $\Delta A^{k,1}(T^k_\ell) \stackrel{d}{=} Q\epsilon_k$ and $\Delta A^{k,2}(T^k_\ell) \stackrel{d}{=} z$, where $z$ is the conditional distribution of $\Delta A^{k,2}(T^k_\ell)$ knowing that $\Delta T^k_\ell=\Delta^{k,1}_\ell$ (see Section \ref{explanation} for details) ;}{ generate $\Delta A^{k,2}(T^k_\ell) \stackrel{d}{=} Q\epsilon_k$ and $\Delta A^{k,1}(T^k_\ell) \stackrel{d}{=} z$, where $z$ is the conditional distribution of $\Delta A^{k,1}(T^k_\ell)$ knowing that $\Delta T^k_\ell = \Delta^{k,2}_\ell$;}

 Generate $\phi^k_\ell$ according to a uniform distribution on $[-\bar{a},\bar{a}]$.

 Store the information $\mathcal{T}=\{T^k_\ell, \ell=1\dots m\}$ $\Delta A^k(T^k_\ell)=(\Delta A^{k,1}(T^k_\ell),\Delta A^{k,2}(T^k_\ell))$ and $\phi^k_\ell$.
 }

 With the information set generated in previous step, calculate
  $\{X^{k,\phi^k}(T^k_\ell); \ell=1, \ldots, m\}$ as a function of $\mathcal{T}$, $\Delta A^k(T^k_\cdot)$ and $\phi^{k}$.

  Store $\mathbf{o}^{k}_\ell:=(\Delta T^k_1, \Delta A^k(T^k_1),\Delta X^{k,\phi^k}(T^k_1),\dots,\Delta T^k_{\ell},
  \Delta A^k(T^k_\ell),\Delta X^{k,\phi^k}(T^k_\ell))$ for $\ell=1,\ldots, m$.

\For{$\ell\leftarrow m$ \KwTo $1$}{
  Starting with $\mathbb{V}^k_m(\mathbf{o}^k_m):=\xi \big( \{X^{k,\phi^k}(T^k_\ell); \ell=1, \ldots, m\}\big)$, solve backwards
      \begin{equation}\label{pseudoeq1}
      C^\star_{k,\ell-1}(\mathbf{o}^k_{\ell-1})\in  \arg \min_{a^k_{\ell-1}\in \mathbb{A}} \int_{\mathbb{W}_k}\mathbb{V}^{k}_{\ell}\Big(\pi_2(\mathbf{o}^{k}_{\ell-1}), \mathfrak{X}^k_{\ell}(a^k_{\ell-1},\pi_2( \mathbf{o}^k_{\ell -1}), w^k)\Big)\nu^k(dw^k),
      \end{equation}
for $\ell=m,\ldots, 1$, where

$$\mathbb{V}^{k}_{\ell-1}(\mathbf{o}^{k}_{\ell-1})= \int_{\mathbb{W}_k}\mathbb{V}^{k}_{\ell}\Big(\pi_2(\mathbf{o}^{k}_{\ell-1}), \mathfrak{X}^k_{\ell}(C^*_{k,\ell-1}(\mathbf{o}^{k}_{\ell-1}),\pi_2( \mathbf{o}^k_{\ell -1}), w^k)\Big)\nu^k(dw^k),$$
for $\ell=m, \ldots, 1$.

Store the information $C^\star_{k,\ell-1}(\mathbf{o}^k_{\ell-1})$ and $\mathbb{V}^k_{\ell-1}(\mathbf{o}^k_{\ell-1})$.
 }
\caption{Pseudocode for a near optimal control designed by (\ref{DPP})}
\end{algorithm}
\vspace{3mm}

%\section{Appendix A: Large deviation results}\label{LDappendix}

\begin{appendix}
\section{Large Deviation results}\label{LDappendix}

In this section, we presents some technical results related to the convergence of the algorithm at the level of the Brownian motion discretization scheme. Recall that $\Delta T^k_n =\min_{j\in\{1,2,\dots, d\}}{\{\Delta^{k,j}_n\}}$ a.s. Here, for each $1\le j\le d$ and $k\ge 1$, $\Delta^{k,j}_n \stackrel{d}{=} \epsilon_k^2 \tau^j_n$ for an iid sequence of random variables $(\tau^j_n)_{n\ge 1}$ with distribution equals to $\inf \{t>0; |W^j(t)|=1\}$ for a standard $d$-dimensional Brownian motion $(W^j)_{j=1}^d$. We observe there exists $\lambda >0$ (see Lemma 2 in \cite{Burq_Jones2008}) such that

\begin{equation}\label{tildepsifunc}
\varphi_d(\lambda):=\mathbb{E}\exp(\lambda \min\{\tau^1_1,\ldots,\tau_1^d\})\le \varphi_1(\lambda):=\mathbb{E}\exp(\lambda \tau^1_1) < \infty.
\end{equation}

Let $\psi_d(\lambda): = \text{ln}~\varphi_d(\lambda)$ defined on $\Lambda:=\{\lambda \in \mathbb{R}; \varphi_d(\lambda) < \infty\}$. We extend $\psi_d(\lambda):=+\infty$ for $\lambda \notin\Lambda$. The Cramer transform is defined by $I_d(a):=\sup_{\lambda \in \mathbb{R}}[\lambda a - \psi_d(\lambda)]$.

%Let $(\tau^i)_{i=1}^d$ be a iid list of random variables with distribution $\inf\{t\ge 0; |W(t)|=1\}$ for a real-valued Brownian motion $W$. Let $I_d$ be the Cramer-Legendre transform $\min_{1\le i\le d}\tau^i$.

At first, we recall the following elementary inequality (see the Supplementary material of \cite{LEAO_OHASHI2017.1}).
\begin{lemma}\label{l.estimativas}
Let $Z_1,\ldots,Z_n$ be an i.i.d. sequence of absolutely continuous positive random variables. Then,
for every $\beta\in(0,1)$ and $r\ge 1$, we have
$$\mathbb{E}\left[\left(\max_{1\le i\le n} Z_i\right)^r\right] \leq  \Big(\mathbb{E}[Z_1^{r/(1-\beta)}]\Big)^{(1-\beta)}n^{1-\beta}.$$
\end{lemma}
In this section, $C$ is a constant which may differ from line to line.
\begin{lemma}\label{meshlemma}
For every $q\ge 1$ and $\beta\in (0,1)$, there exists a constant $C$ which depends on $q\ge 1$ and $\beta$ such that

$$
\mathbb{E}|\max_{n\ge 0}\Delta T^{k}_{n+1}|^q 1\!\!1_{\{T^k_n\le T\}} \le C \epsilon_k^{2q - 2(1-\beta)}
$$
for every $k\ge 1$.

%$$
%\mathbb{E}|\vee_{n\ge 0}\Delta T^{k}_{n+1}|^q %1\!\!1_{\{T^k_n\le T\}} \le C \Big(\epsilon_k^{2q} \Big\lceil %\frac{\epsilon^{-2}_k T}{\chi_d}\Big\rceil^{(1-\beta)}\Big)
%$$
%for every $k\ge 1$.
\end{lemma}
\begin{proof}
%We start by recalling that $\Delta T^k_n =\min_{j\in\{1,2,\dots, d\}}{\{\Delta T^{k,j}_n\}}$ a.s where $\Delta T^{k,j}_n = \epsilon_k^2 \tau^j_n$ (in law) for an iid sequence of random variables $(\tau^j_i)_{i\ge 1}$ with distribution equals to $\inf \{t>0; |W^j(t)|=1\}$ for a sequence of independent  real-valued Brownian motions $(W^j)_{j=1}^d$.

The proof is almost identical to Lemma 2.2 in \cite{LEAO_OHASHI2017.1}. For sake of completeness, we give the details. Let $N^k(t)=\max\{n; T^k_n\le t\}; 0\le t\le T$. We split
\begin{eqnarray*}
\max_{0\le n \le  N^{k}(T)}\Delta T^{k}_{n+1} &=&\Big(\max_{0\le n \le  N^{k}(T)}\Delta T^{k}_{n+1}\Big)1\!\!1_{\{N^{k}(T) < 2 e(k,T)\}}\\
&+& \Big(\max_{0\le n \le  N^{k}(T)}\Delta T^{k}_{n+1} \Big)1\!\!1_{\{N^{k}(T) \ge 2e(k,T)\}}.
\end{eqnarray*}
We observe $\Big(\max_{0\le n \le  N^{k}(T)}\Delta T^{k}_{n+1} \Big)1\!\!1_{\{N^{k}(T) \ge 2e(k,T)\}}\le \max\{T; \overline{T}^+_k-\overline{T}_k \} 1\!\!1_{\{N^{k}(T) \ge 2e(k,T)\}}$, where we use the notation (\ref{localt}) and we set $\overline{T}^+_k:= \min\{T^k_n; T < T^k_n\}$. Fix $q\ge 1$. By H\"older's inequality, if $p_1,p_2>1$ are conjugate numbers, we have
\begin{eqnarray*}
\mathbb{E}\Big|\max_{0\le n \le  N^{k}(T)}\Delta T^k_{n+1}\Big|^q &\le& C \mathbb{E}\Big|\max_{0\le n \le  2e(k,T)}\Delta T^k_{n+1}\Big|^q\\
 &+& C \big(\mathbb{E}\big| \max\{T; \overline{T}^+_k-\overline{T}_k \} \big|^{p_1q}\big)^{\frac{1}{p_1}}\big(\mathbb{P}\{N^{k}(T)\ge 2 e(k,T)\}\big)^{\frac{1}{p_2}}\\
& &\\
&\le& C \mathbb{E}\Big|\max_{0\le n \le  2e(k,T)}\Delta T^k_{n+1}\Big|^q  + C \big(\mathbb{P}\{N^{k}(T)\ge 2 e(k,T)\}\big)^{\frac{1}{p_2}}\\
&=:&I^{k}_1 + I^{k}_2,
\end{eqnarray*}
where we recall $e(k,T) = \big\lceil \frac{\epsilon^{-2}_k T}{\chi_d}\big\rceil$. Clearly,

$$\mathbb{P}\{N^{k}(T)\ge 2 e(k,T)\} = \mathbb{P}\Big\{T^{k}_{2 e(k,T)} \le T\Big\}.$$
We shall write

$$
T^k_{n} \stackrel{d}{=} \epsilon^2_k \sum_{\ell=1}^n \alpha_\ell,
$$
where $\alpha_\ell \stackrel{d}{=} \min\{\tau^1_\ell,\ldots, \tau^d_\ell\}$. Then, $(\alpha_\ell)_{\ell=1}^\infty$ is an iid sequence wit mean $\chi_d$. Then, the standard large deviation theory yields

\begin{eqnarray*}
\mathbb{P}\{T^{k}_{2 e(k,T)} \le T\} &=& \mathbb{P}\Big\{ \frac{2 e(k,T)}{2e(k,T)} \epsilon^2_k \sum_{\ell=1}^{2e(k,T)} \alpha_\ell \le T \Big\}\\
 & &\\
 &=& \mathbb{P}\Big\{ \frac{1}{2e(k,T)} \sum_{\ell=1}^{2e(k,T)} \alpha_\ell \le \frac{T}{2e(k,T) \epsilon_k^2} \Big\}\\
  & &\\
&\le& \mathbb{P}\Big\{ \frac{1}{2e(k,T)} \sum_{\ell=1}^{2e(k,T)} \alpha_\ell \le \frac{\chi_d}{2} \Big\}\\
& &\\
&\le& \exp\Big(-2e(k,T)I_d\Big(\frac{\chi_d}{2}\Big)\Big).
\end{eqnarray*}
Let us fix $\beta \in (0,1)$. Take $n=2e(k,T)$ in Lemma \ref{l.estimativas} and notice that
$$\left(\mathbb{E}[(\Delta T^{k}_1)^{q/(1-\beta)}]\right)^{1-\beta} = \epsilon_k^{2q}\left(\mathbb{E}[\min\{\tau^1_1, \ldots, \tau^d_1\}^{q/(1-\beta)}]\right)^{1-\beta}=: C\epsilon_k^{2q},$$
where $C$ is a constant depending on $\beta$ and $q$. Therefore, by applying Lemma \ref{l.estimativas}, we have
$$I^{k}_1\le C \epsilon^{2q}_k e(k,T)^{(1-\beta)}; k\ge 1.$$
Finally, we notice $e(k,T)^{1-\beta}\lesssim (1 + \epsilon_k^{-2(1-\beta)})$ for every $k\ge 1$. This concludes the proof.
\end{proof}

\begin{lemma}\label{Tkrate}
For each $t\in (0,T]$ and $0 < \delta < t$, we have
$$
\lim_{k\rightarrow+\infty}T^k_{e(k,t)}=t~\text{a.s}
$$
and
$$
\mathbb{P}\Big\{ T^k_{e(k,t)} < t-\delta \Big\} \le  \exp \Big[-e(k,t) I_d \Big( \chi_d \Big(1-\frac{\delta}{t} \Big)\Big)\Big],
$$
for every $k\ge 1$. For every $\eta \in (0,1)$, there exists a constant $C$ which depends on $\chi_d$ and $\eta$ such that

\begin{equation}\label{astk1}
\mathbb{P}\{T^k_{e(k,T)} < T-\epsilon_k^\eta\}\le C \exp \big( -\epsilon_k^{2(\eta-1)}\big),
\end{equation}
for every $k\ge 1$ sufficiently large. For every $\eta \in (0,1)$ and $\theta >0$, there exists a constant which depends on $T$, $\eta$ and $\theta$ such that

\begin{equation}\label{astk2}
\int_{\{T^k_{e(k,T)} < T\}}|T-T^k_{e(k,T)}|^\theta d\mathbb{P}\le C \epsilon_k^{\eta \theta},
\end{equation}
for every $k\ge 1$ sufficiently large.
%Take $\eta_k$ such that

%\begin{itemize}
%  \item $\eta^2_k \epsilon_k^{-2}\uparrow +\infty$
%  \item $\eta^3_k\epsilon_k^{-2} = O(1)$ as $k\rightarrow +\infty$.
%\end{itemize}
%Then,
%\begin{eqnarray*}
%\mathbb{P}\{T^{k}_{e(k,T)} \le T-\eta_k\}&\le& \exp\Big(-e(k,T)I^*_d \Big(\chi_d\Big(1-\frac{\eta_k}{T}\Big)\Big)\Big)\\
%&\lesssim& \exp\big(\eta_k^3 \epsilon^{-2}_k\big)\exp\big( -\epsilon^{-2}_k\eta^2_k \big)\rightarrow 0
%\end{eqnarray*}
%as $k\rightarrow +\infty$. In particular, for every $\theta >0$, we have

%$$
%\int_{\{T^k_{e(k,T)} < T\}}|T-T^k_{e(k,T)}|^\theta d\mathbb{P}\lesssim \eta^\theta_k \rightarrow 0,
%$$
%as $k\rightarrow +\infty$.
\end{lemma}
\begin{proof}

We shall write $T^k_{n} \stackrel{d}{=} \epsilon^2_k \sum_{\ell=1}^n \alpha_\ell,$ where $\alpha_\ell \stackrel{d}{=} \min\{\tau^1_\ell,\ldots, \tau^d_\ell\}$. Then, $(\alpha_\ell)_{\ell=1}^\infty$ is an iid sequence wit mean $\chi_d = \mathbb{E}[\min_{1\le i\le d}\tau^i_1]$. The classical large deviation principle yields
\begin{eqnarray*}
\mathbb{P}\{T^{k}_{e(k,t)} \le t-\delta\} &=& \mathbb{P}\Big\{ \frac{e(k,t)}{e(k,t)} \epsilon^2_k \sum_{\ell=1}^{e(k,t)} \alpha_\ell \le t-\delta \Big\}\\
 & &\\
 &=& \mathbb{P}\Big\{ \frac{1}{e(k,t)} \sum_{\ell=1}^{e(k,t)} \alpha_\ell \le \frac{t-\delta}{e(k,t) \epsilon_k^2} \Big\}\\
  & &\\
&\le& \mathbb{P}\Big\{ \frac{1}{e(k,t)} \sum_{\ell=1}^{e(k,t)} \alpha_\ell \le \chi_d\Big(1-\frac{\delta}{t}\Big) \Big\}\\
& &\\
&\le& \exp\Big(-e(k,t)I_d \Big(\chi_d\Big(1-\frac{\delta}{t}\Big)\Big)\Big).
\end{eqnarray*}
Next, we take a sequence $\eta_k\downarrow 0$ such that

$$\eta^2_k \epsilon_k^{-2}\uparrow +\infty,\quad \eta^3_k\epsilon_k^{-2} = O(1)$$
as $k\rightarrow +\infty$.

In the sequel, $I_d^{(i)}$ denotes the $i$th-derivative of $I_d$.
Let $\mathcal{D}_{I_d} =\{x \in \mathbb{R};  I_d(x) < \infty\}$. It is known that (see Lemma I.14 in \cite{frank})

\begin{itemize}
  \item $I_d(\chi_d)=0$ and $\chi_d$ is a minimal of the convex function $I_d$ .
  \item $I_d \ge 0$ and continuous (strictly convex)
  \item $x \mapsto I_d(x)$ is non-increasing for $x < \chi_d$,  $x \mapsto I_d(x)$ is non-decreasing for $x > \chi_d$.
  \item $I^{(2)}_d(\chi_d) = \frac{1}{\sigma^2_d}$, where $\sigma^2_d$ is the variance of $\min_{1\le i\le d}\tau^i_1$.
%\item Since $r\mapsto \frac{d^2\ln (\varphi_d(r))}{dr^2}>0$ is a strictly positive function on $\text{int}(\mathcal{D}_{\log(\varphi_d)})$, the $\frac{d^2 I^*_d(p)}{dp^2} >0$ for every $p \in \text{int}\big(\mathcal{D}_{I^*_d}\big)$.
   \item $I_d \in C^\infty (\text{int} (\mathcal{D}_{I_d}))$.
\end{itemize}
Since $\min_{1\le i\le d}\tau^i_1$ is an absolutely continuous random variable, we know (see e.g Proposition 1.9 in \cite{vares}) that $\chi_d \in \text{int}(\mathcal{D}_{I_d})$ and hence $I^{(1)}_d(\chi_d)=0$. Taylor expansion

$$I_d\Big(\chi_d \big(1-\frac{\eta_k}{T}\big)\Big) - I_d(\chi_d) = I^{(1)}_d(\chi_d)\Big(\frac{-\chi_d\eta_k}{T}\Big) + \frac{I^{(2)}_d(\chi_d)}{2}\Big(\frac{\chi_d\eta_k}{T}\Big)^2 + R_2\Big( \chi_d \Big(1-\frac{\eta_k}{T}\Big)\Big) $$
Because  $I_d \in C^\infty (\text{int} (\mathcal{D}_{I_d}))$, $\chi_d \in \text{int}(\mathcal{D}_{I_d})$ and $\min_{1\le i\le d}\tau^i_1$ is an absolutely continuous random variable, then there exists an open neighborhood $V_{\chi_d}$ of $\chi_d$ and $M < \infty$ such that

$$\sup_{x \in V_{\chi_d}}\big| I^{(3)}_d(x)\big| \le M < \infty.$$
There exists $k_0 = k_0(\chi_d)$ such that $\Big(\chi_d (1-\frac{\chi_d\eta_k}{T}\big)\Big) \in V_{\chi_d}$ for every $k\ge k_0$. Standard estimates on the remainder of the Taylor expansion yields

$$|R_2(x)|\le M \frac{|x-\chi_d|^{3}}{3!}, $$
for every $x\in V_{\chi_d}$.
Then,

%For instance, if $\epsilon_k = 2^{-k}$, then we take $\eta_k = 2^{-\alpha k}$ for $0 < \alpha <1$.

\begin{eqnarray*}
\mathbb{P}\{T^{k}_{e(k,T)} \le T-\eta_k\}&\le& \exp\Big(-e(k,T)I_d \Big(\chi_d\Big(1-\frac{\eta_k}{T}\Big)\Big)\Big)\\
&\lesssim& \exp\big(M\eta_k^3 \epsilon^{-2}_k\big)\exp\big( -\epsilon^{-2}_k\eta^2_k \big)\rightarrow 0,
\end{eqnarray*}
as $k\rightarrow +\infty$. As a consequence, whenever $0 < \eta < \frac{2}{3}\le \gamma < 1$, we obtain

$$\mathbb{P}\{T^k_{e(k,T)} < T-\epsilon_k^\eta\}\le \mathbb{P}\{T^k_{e(k,T)} < T-\epsilon_k^\gamma\}\le \exp(M\epsilon_{k_0}^{3\gamma-2})\exp\big(-\epsilon_k^{2(\gamma-1)}\big),$$
for every $k\ge k_0$. Now, we observe for a given $0 < \eta < 1$, we can choose $\beta (\eta)$ such that $\frac{2}{3}\le \eta + \beta(\eta) < 1$ and obviously $\eta \le \eta + \beta(\eta) < 1$. This shows (\ref{astk1}). Therefore, for $\theta>0$ and $\eta \in (0,1)$, we have

\begin{eqnarray*}
\int_{\{T^k_{e(k,T)} < T\}}|T-T^k_{e(k,T)}|^\theta d\mathbb{P} &=& \int_{\{T^k_{e(k,T)} < T, |T-T^k_{e(k,T)}|> \epsilon^\eta_k \}}|T-T^k_{e(k,T)}|^\theta d\mathbb{P}\\
&+& \int_{\{T^k_{e(k,T)} < T, |T-T^k_{e(k,T)}|\le \epsilon_k^\eta \}}|T-T^k_{e(k,T)}|^\theta d\mathbb{P}\\
&\le& \epsilon_k^{\theta \eta} + T^\theta \mathbb{P} \big\{ T^k_{e(k,T)} < T-\epsilon_k^{\eta} \big\}\\
&\le& \epsilon^{\theta \eta}_k + C\exp\big( -\epsilon^{2(\eta-1)}_k \big),
\end{eqnarray*}
for every $k\ge 1$ sufficiently large. This concludes (\ref{astk2}).
\end{proof}

%\section{Appendix B: Construction of optimal controls: Measurable selection and $\epsilon$-controls}\label{MSTappendix}

\section{Construction of optimal controls: Measurable selection and $\epsilon$-controls}\label{MSTappendix}

In this section, we provide the proofs of the technical details associated with our dynamic programming algorithm described in Section \ref{mainrsection}. In the sequel, it will be important to deal with universally measurable sets. For readers who are not familiar with this class of sets, we refer to e.g \cite{bertsekas}. If $(R , \mathcal{B}(R))$ is a Borel space, let $P(R)$ be the space of all probability measures defined on the Borel $\sigma$-algebra $\mathcal{B}(R)$ generated by $R$. The universal $\sigma$-algebra $\mathcal{E}(R)$ is defined by

$$\mathcal{E}(R):=\bigcap_{p\in P(R)}\mathcal{B}(R,p),$$
where $\mathcal{B}(R,p)$ is the $p$-completion of $\mathcal{B}(R)$ w.r.t $p\in P(R)$. In order to keep notation simple, for a given $k\ge 1$, we denote $m=e(k,T)$.

\begin{lemma} \label{comp_fund} Let $G^{k}_j: \mathbb{H}^{j}_k \rightarrow E$ be a universally measurable function, where $(E , \mathcal{B}(E))$ is a Borel space and let $\Xi^{k,u^k}_j$ be the map given by (\ref{Xioperator}), for $u^k\in U^{k,m}_0$. Then, the composition $G^k_j \circ \Xi^{k,u^k}_j $ is $\mathcal{F}^k_{T^k_j}$-mensurable, for every $u^k\in U^{k,m}_0$ and $j=1, \ldots ,m$.
\end{lemma}
\begin{proof} We need to show that given $D \in \mathcal{B} (E)$, the inverse image $(\Xi^{k,u^k}_j )^{-1} [(G^{k}_j)^{-1} (D)] \in \mathcal{F}^k_{T^k_{j}}$. Since $(G^{k}_j)^{-1} (D)$ is a universally measurable set, it is sufficient to check that $(\Xi^{k,u^k}_j )^{-1} (H) \in \mathcal{F}^k_{T^k_j}$ for every universally measurable set $H$ in $\mathbb{H}^{j}_k$. Let $\mu$ be a probability measure on $\mathbb{H}^{j}_k$ given by

\[
\mu (C) = \mathbb{P} \left[ (\Xi^{k,u^k}_j )^{-1} (C) \right], ~ ~ C \in \mathcal{B} (\mathbb{H}^{j}_k).
\]

For a given universally measurable set $H$ in $\mathbb{H}^{j}_k$, we can select (see e.g Lemma 7.26 in \cite{bertsekas}) $\tilde{H} \in \mathcal{B} (\mathbb{H}^{j}_k)$ such that

\[
\mathbb{P} \left[ (\Xi^{k,u^k}_j )^{-1} (\tilde{H}) \bigtriangleup (\Xi^{k,u^k}_j )^{-1} (H)\right] = \mu \big( \tilde{H} \bigtriangleup H\big)= 0.
\]
The set $(\Xi^{k,u^k}_j )^{-1} (\tilde{H}) \in \mathcal{F}^k_{T^k_j}$, so there exists $W \in \widetilde{\mathcal{F}}^k_{T^k_j}$ such that $\mathbb{P} \left[W \bigtriangleup (\Xi^{k,u^k}_j )^{-1} (\tilde{H}) \right]=0$. Then, $\mathbb{P} \left[W \bigtriangleup (\Xi^{k,u^k}_j )^{-1} (H) \right]=0$ and hence, $(\Xi^{k,u^k}_j )^{-1} (H) \in \mathcal{F}^k_{T^k_j}$. This concludes the proof.

\end{proof}

Let us now present a selection measurable theorem which will allow us to \textit{aggregate} the map $\phi^k\mapsto V^k(\cdot,\phi^k)$ into a list of upper semi-analytic functions $\mathbb{V}^k_j:\mathbb{H}^{j}_k\rightarrow \mathbb{R}; j=0, \ldots, m$. More importantly, we will construct an $\epsilon$-optimal control at the level of the optimization problem (\ref{discreteOCPROBLEM}). Recall that a structure of the form (\ref{controlledstate}) is fixed and it is equipped with a Borel function $\gamma_m:\mathbb{H}^{m}_k\rightarrow \mathbf{D}_{n,T}$ realizing (\ref{compoID}) with a given initial condition $x_0\in \mathbb{R}^n$. For such structure, we write $V^k$ as the associated value process given by (\ref{discretevalueprocess}). We start with the map $\mathbb{V}^{k}_m:\mathbb{H}^{m}_k\rightarrow\mathbb{R}$ defined by

$$\mathbb{V}^{k}_m(\textbf{o}^{k}_m):=\xi(\gamma_m(\textbf{o}^{k}_m)); \textbf{o}^{k}_m\in \mathbb{H}^{m}_k.$$
By construction, $\mathbb{V}^{k}_m$ is a Borel function. We recall $(\Delta T^k_1, \Delta A^k(T^k_1 ))\stackrel{d}{=}(\Delta T^k_j, \Delta A^k(T^k_j ))$, for $j=1, \ldots, m$.

\begin{lemma}\label{conditionalexprep}
Assume that $\xi_{X^k}(u^k)\in L^1(\mathbb{P})$ for $u^k \in U^{k,m}_0$. Then,

\begin{equation}\label{1oiter_Geral}
 \mathbb{E}\left[ \xi_{X^k}(u^k)|\mathcal{F}^k_{T^k_{m-1}}\right] = \int_{\mathbb{W}_k} \mathbb{V}^{k}_{m}\big(\Xi^{k,u^k}_{m-1}, \mathfrak{X}^k_{m}(u^k_{m-1}, \Xi^{k,u^k}_{m-1}, w^k)\big) \nu^k(dw^k)  ~a.s.
\end{equation} where $\nu^k$ is the law of the $\mathbb{W}_k$-valued random vector $(\Delta T^k_1, \Delta A^k(T^k_1 ))$.
\end{lemma}

\begin{proof}
At first, we recall that $\widetilde{\mathbb{F}}^k=\{\widetilde{\mathcal{F}}^k_t; 0\le t < \infty\}$ is the raw filtration generated by $A^k$ given in section \ref{discretesec}. Let $\widetilde{U}^{k,m}_0$ be the subset of controls $v^k=(v^k_0, \ldots, v^k_{m-1})$ in $U^{k,m}_0$ such that $v^k_j \in \widetilde{\mathcal{F}}^k_{T^k_j}$ for $j=0,\ldots, m-1$. At first, we observe
\begin{equation} \label{iqual_1}
\mathbb{E}\left[ \xi_{X^k}(u^k)\big|\mathcal{F}^k_{T^k_{m-1}}\right] = \mathbb{E}\left[ \xi_{X^k}(u^k)\big|\widetilde{\mathcal{F}}^k_{T^k_{m-1}}\right]\quad a.s\end{equation}
for every $u^k = (u^k_0, \ldots, u^k_{m-1})\in U^{k,m}_0$.  Elementary computation yields

$$
\mathbb{E}\left[ \xi_{X^k}(g^k)\big|\widetilde{\mathcal{F}}^k_{T^k_{m-1}}\right] =
\int_{\mathbb{W}_k} \mathbb{V}^{k}_{m}\big(\Xi^{k,g^k}_{m-1}, \mathfrak{X}^k_{m}(g^k_{m-1}, \Xi^{k,g^k}_{m-1}, w^k)\big) \nu^k(dw^k)
$$
for every control $g^k = (g^k_0, \ldots, g^k_{m-1})\in \widetilde{U}^{k,m}_0$.  Now, we recall for a given $u^k = (u^k_0, \ldots, u^k_{m-1})\in U^{k,m}_0$, there exists $g^k = (g^k_0, \ldots, g^k_{m-1}) \in \widetilde{U}^{k,m}_0$ which fulfills the following: for each $i$, there exists a set $\tilde{G}^k_i \in \widetilde{\mathcal{F}}^k_{T^k_{i}}$ of full probability such that $u^k_i (\omega) = g^k_i(\omega);\omega \in \tilde{G}^k_i$ for $i=0, \ldots , m-1$. Let us denote $\tilde{G}^k = \cap_{i=0}^{m-1} \widetilde{G}^k_i \in \widetilde{\mathcal{F}}^k_{T^k_{m-1}}$. Then, $u^k=g^k$ in $\tilde{G}^k$ and hence

\begin{eqnarray} \nonumber
\nonumber\mathbb{E}\left[\xi_{X^k}(u^k) |\widetilde{\mathcal{F}}^k_{T^k_{m-1}} \right] &=&\mathbb{E}\left[ \xi_{X^k}(u^k)  |\widetilde{\mathcal{F}}^k_{T^k_{m-1}} \right] \mathbbm{1}_{ \tilde{G}^k } =  \mathbb{E}\left[ \xi_{X^k}(u^k) \mathbbm{1}_{ \tilde{G}^k } |\widetilde{\mathcal{F}}^k_{T^k_{m-1}} \right]\\
\nonumber&=& \mathbb{E}\left[ \xi_{X^k}(g^k) \mathbbm{1}_{ \tilde{G}^k } |\widetilde{\mathcal{F}}^k_{T^k_{m-1}} \right]\\
\nonumber&=& \int_{\mathbb{W}_k} \mathbbm{1}_{ \tilde{G}^k } \mathbb{V}^{k}_{m}\big(\Xi^{k,g^k}_{m-1}, \mathfrak{X}^k_{m}(g^k_{m-1}, \Xi^{k,g^k}_{m-1}, w^k)\big) \nu^k(dw^k)\\
\nonumber&=& \int_{\mathbb{W}_k} \mathbbm{1}_{ \tilde{G}^k } \mathbb{V}^{k}_{m}\big(\Xi^{k,u^k}_{m-1}, \mathfrak{X}^k_{m}(u^k_{m-1}, \Xi^{k,u^k}_{m-1}, w^k)\big) \nu^k(dw^k)\\
\label{iqual_2}&=&
\int_{\mathbb{W}_k} \mathbb{V}^{k}_{m}\big(\Xi^{k,u^k}_{m-1}, \mathfrak{X}^k_{m}(u^k_{m-1}, \Xi^{k,u^k}_{m-1}, w^k)\big) \nu^k(dw^k) ~a.s.
\end{eqnarray}
Identities (\ref{iqual_1}) and (\ref{iqual_2}) allow us to conclude the proof.
\end{proof}

\begin{lemma}\label{measurabilityissue1}
The map
$$
(\mathbf{o}^k_{m-1},a^k_{m-1})\mapsto
\int_{\mathbb{W}_k} \mathbb{V}^{k}_{m}\big(\mathbf{o}^k_{m-1}, \mathfrak{X}^k_{m}(a^k_{m-1}, \mathbf{o}^k_{m-1}, w^k)\big) \nu^k(dw^k)
$$
is a Borel function from $\mathbb{H}^{m-1}_k\times\mathbb{A}$ to $\overline{\mathbb{R}}$.
\end{lemma}
\begin{proof}
We just need to imitate the proof of Prop. 7.29 in \cite{bertsekas}, so we omit the details.
\end{proof}

\begin{lemma}\label{iterUM}
Let $\mathbb{V}^{k}_{m-1}:\mathbb{H}^{m-1}_k\rightarrow \overline{\mathbb{R}}$ be the function defined by

$$\mathbb{V}^{k}_{m-1}(\mathbf{o}^{k}_{m-1}):=\sup_{a^k_{m-1}\in \mathbb{A}} \int_{\mathbb{W}_k} \mathbb{V}^{k}_{m}\big(\mathbf{o}^k_{m-1}, \mathfrak{X}^k_{m}(a^k_{m-1}, \mathbf{o}^k_{m-1}, w^k)\big) \nu^k(dw^k),$$
for $\mathbf{o}^{k}_{m-1}\in \mathbb{H}^{m-1}_k$. Then, $\mathbb{V}^{k}_{m-1}$ is upper semianalytic and for every $\epsilon> 0$, there exists an analytically measurable function $C^\epsilon_{k,m-1}:\mathbb{H}^{m-1}_k\rightarrow\mathbb{A}$ which realizes

\begin{equation}\label{1oiter1}
\mathbb{V}^{k}_{m-1}(\mathbf{o}^{k}_{m-1})\le
\int_{\mathbb{W}_k} \mathbb{V}^{k}_{m}\big(\mathbf{o}^k_{m-1}, \mathfrak{X}^k_{m}(C^\epsilon_{k,m-1}(\mathbf{o}^{k}_{m-1}), \mathbf{o}^k_{m-1}, w^k)\big) \nu^k(dw^k) +\epsilon,
\end{equation}
for every $\mathbf{o}^{k}_{m-1}\in \{\mathbb{V}^k_{m-1} < + \infty\}$.
\end{lemma}

\begin{proof}
The fact that $\mathbb{V}^{k}_{m-1}$ is upper semianalytic follows from Prop 7.47 in \cite{bertsekas} and Lemma \ref{measurabilityissue1} which say the map given by
$$f(\mathbf{o}^{k}_{m-1},a^k_{m-1})= \int_{\mathbb{W}_k} \mathbb{V}^{k}_{m}\big(\mathbf{o}^k_{m-1}, \mathfrak{X}^k_{m}(a^k_{m-1}, \mathbf{o}^k_{m-1}, w^k)\big) \nu^k(dw^k) $$
is a Borel function (hence upper semianalytic). By construction, $\mathbb{H}^{m-1}_k\times\mathbb{A}$ is a Borel space. Let

$$\mathbb{V}^{k}_{m-1}(\mathbf{o}^{k}_{m-1})= \sup_{a^k_{m-1}\in \mathbb{A}} f(\mathbf{o}^{k}_{m-1},a^k_{m-1}); \quad \mathbf{o}^{k}_{m-1}\in \mathbb{H}^{m-1}_k.$$
Prop 7.50 in \cite{bertsekas} yields the existence of an analytically measurable function $C^\epsilon_{k,m-1}:\mathbb{H}^{m-1}_k\rightarrow\mathbb{A}$ such that

$$f\big(\mathbf{o}^{k}_{m-1},C^\epsilon_{k,m-1}(\mathbf{o}^{k}_{m-1})\big)\ge \mathbb{V}^{k}_{m-1}(\mathbf{o}^{k}_{m-1})-\epsilon,$$
for every $\mathbf{o}^{k}_{m-1}\in \{\mathbb{V}^{k}_{m-1} < +\infty\}$.
\end{proof}

\begin{lemma}\label{iterDOIS}
If $\mathbb{H}^{m-1}_k = \{\mathbb{V}^k_{m-1} < + \infty\}$, then for every $\epsilon>0$ and $u^k\in U^{k,m}_0$, there exists a control $\phi^{k,\epsilon}_{m-1}\in U^{k,m}_{m-1}$ such that

\begin{equation}\label{iterdois3}
\mathbb{V}^{k}_{m-1}(\Xi^{k,u^k}_{m-1}) = V^k(T^k_{m-1}, u^k) \le \mathbb{E}\big[V^k(T^k_m, u^k\otimes_{m-1}\phi^{k,\epsilon}_{m-1}) |\mathcal{F}^k_{T^k_{m-1}}\big] + \epsilon~a.s.
\end{equation}
\end{lemma}

\begin{proof} For each $\alpha^{k}_{m-1} \in U^{k,m}_{m-1}$, (\ref{1oiter_Geral}) yields  $\mathbb{E}\big[\xi_{X^k}(u^k\otimes_{m-1}\alpha^{k}_{m-1})|\mathcal{F}^k_{T^k_{m-1}}\big]$ equals (a.s) to

\[
\int_{\mathbb{W}_k} \mathbb{V}^{k}_{m}\big(\Xi^{k,u^k}_{m-1}, \mathfrak{X}^k_{m}(\alpha^k_{m-1}, \Xi^{k,u^k}_{m-1}, w^k)\big) \nu^k(dw^k) .
\] As a consequence, we obtain that

\begin{equation} \label{Ine_Fund_1}
\mathbb{V}^{k}_{m-1}(\Xi^{k,u^k}_{m-1}) \geq \mathbb{E}\big[\xi_{X^k}(u^k\otimes_{m-1}\alpha^{k}_{m-1})|\mathcal{F}^k_{T^k_{m-1}}\big] ~ a.s,
\end{equation} for every $\alpha^{k}_{m-1} \in U^{k,m}_{m-1}$. By composing $\mathbb{V}^{k}_{m-1}$ with $\Xi^{k,u^k}_{m-1}$, we obtain that $\mathbb{V}^{k}_{m-1}(\Xi^{k,u^k}_{m-1})$ is an $\mathcal{F}^k_{T^k_{m-1}}$-measurable function (see Lemma \ref{comp_fund}). Then, the definition of $\esssup$ and (\ref{Ine_Fund_1}) yield that

\begin{equation} \label{Ineq_1}
\mathbb{V}^{k}_{m-1}(\Xi^{k,u^k}_{m-1}) \geq \esssup_{\theta^{k}_{m-1} \in U^{k,m}_{m-1}}
\mathbb{E} \Bigg[ V^{k} \left(T^k_{m}, u^{k,m-1} \otimes_{m-1} \theta^{k}_{m-1} \right)    \mid \mathcal{F}^k_{T^k_{m-1}}\Bigg] = V^k(T^k_{m-1}, u^k)~ a.s.
\end{equation}

For $\epsilon>0$, let $C^\epsilon_{k,m-1}:\mathbb{H}^{k,m-1}\rightarrow\mathbb{A}$ be the analytically measurable function which realizes (\ref{1oiter1}). We take $\phi^{k,\epsilon}_{m-1}:=C^{\epsilon}_{k,m-1}(\Xi^{k,u^k}_{m-1})$ as the composition of an analytically measurable function with an $\mathcal{F}^k_{T^k_{m-1}}$-measurable one. In this case, by Lemma \ref{comp_fund}, we know that $C^{\epsilon}_{k,m-1}\circ \Xi^{k,u^k}_{m-1}$ is an $\mathcal{F}^k_{T^k_{m-1}}$-measurable function. This shows that $\phi^{k,\epsilon}_{m-1}$ is an admissible control. By (\ref{1oiter1}), we have

\begin{eqnarray}
\nonumber \mathbb{V}^{k}_{m-1}(\Xi^{k,u^k}_{m-1}) &\leq& \int_{\mathbb{W}_k} \mathbb{V}^{k}_{m}\big(\Xi^{k,u^k}_{m-1}, \mathfrak{X}^k_{m}(\phi^{k,\epsilon}_{m-1}, \Xi^{k,u^k}_{m-1}, w^k)\big) \nu^k(dw^k) + \epsilon \\
\nonumber&=& \mathbb{E}\big[\xi_{X^k}(u^k\otimes_{m-1}\phi^{k,\epsilon}_{m-1})|\mathcal{F}^k_{T^k_{m-1}}\big] + \epsilon\\
\label{Ineq_3}&=& \mathbb{E}\big[V^k(T^k_m, u^k\otimes_{m-1}\phi^{k,\epsilon}_{m-1}) |\mathcal{F}^k_{T^k_{m-1}}\big] + \epsilon\\
\nonumber&\leq& \esssup_{\alpha^{k}_{m-1} \in U^{k,m}_{m-1}}
\mathbb{E} \Bigg[ V^{k} \left(T^k_{m}, u^{k,m-1} \otimes_{m-1} \alpha^{k}_{m-1} \right)    \mid \mathcal{F}^k_{T^k_{m-1}}\Bigg] + \epsilon ~~a.s.
\end{eqnarray}

As $\epsilon > 0$ is arbitrary, we conclude that

\begin{equation}  \label{Ineq_2}
\mathbb{V}^{k}_{m-1}(\Xi^{k,u^k}_{m-1}) \leq \esssup_{\alpha^{k}_{m-1} \in U^{k,m}_{m-1}}
\mathbb{E} \Bigg[ V^{k} \left(T^k_{m}, u^{k,m-1} \otimes_{m-1} \alpha^{k}_{m-1} \right)    \mid \mathcal{F}^k_{T^k_{m-1}}\Bigg] = V^k(T^k_{m-1}, u^k)~ a.s.
\end{equation}
Hence, (\ref{Ineq_1}) and (\ref{Ineq_2}) yield

\[
\mathbb{V}^{k}_{m-1}(\Xi^{k,u^k}_{m-1}) = V^k(T^k_{m-1}, u^k) ~ a.s, ~ ~ u^k\in U^{k,m}_0.
\]
As a consequence of (\ref{Ineq_3}), we also obtain that

\[
V^k(T^k_{m-1}, u^k) = \mathbb{V}^{k}_{m-1}(\Xi^{k,u^k}_{m-1}) \leq \mathbb{E}\big[V^k(T^k_m, u^k\otimes_{m-1}\phi^{k,\epsilon}_{m-1}) |\mathcal{F}^k_{T^k_{m-1}}\big] + \epsilon,
\]
and this concludes the proof.
\end{proof}

We are now able to iterate the argument as follows. From a backward argument, we can define the sequence of functions $\mathbb{V}^{k}_j:\mathbb{H}^{j}_k\rightarrow\overline{\mathbb{R}}$

\begin{equation}\label{Vkfunction}
\mathbb{V}^{k}_{j}(\mathbf{o}^{k}_{j}):= \sup_{a^k_{j}\in \mathbb{A}} \int_{\mathbb{W}_k} \mathbb{V}^{k}_{j+1}\big(\mathbf{o}^k_{j}, \mathfrak{X}^k_{j+1}(a^k_{j}, \mathbf{o}^k_{j}, w^k)\big) \nu^k(dw^k),
\end{equation}
for $\mathbf{o}^k_{j}\in \mathbb{H}^{j}_k$ and $j=m-1,\ldots, 1$.

\begin{lemma}\label{measurabilityissue2}
For each $j=m-1, \ldots, 1$, the map

$$(\mathbf{o}^k_{j},a^k_{j})\mapsto  \int_{\mathbb{W}_k} \mathbb{V}^{k}_{j+1}\big(\mathbf{o}^k_{j}, \mathfrak{X}^k_{j+1}(a^k_{j}, \mathbf{o}^k_{j}, w^k)\big) \nu^k(dw^k)
$$
is upper semianalytic from $\mathbb{H}^{j}_k\times \mathbb{A}$ to $\overline{\mathbb{R}}$.
\end{lemma}
\begin{proof}
The same argument used in the proof of Lemma \ref{measurabilityissue1} applies here. We omit the details.
\end{proof}

\begin{proposition}\label{detarg}
The function $\mathbb{V}^{k}_j:\mathbb{H}^{j}_k\rightarrow\overline{\mathbb{R}}$ is upper semianalytic for each $j=m-1,\ldots, 1$. Moreover, for every $\epsilon>0$, there exists an analytically measurable function $C^\epsilon_{k,j}:\mathbb{H}^{j}_k\rightarrow\mathbb{A}$ such that

$$
\mathbb{V}^{k}_j(\mathbf{o}^k_{j})\le \int_{\mathbb{W}_k} \mathbb{V}^{k}_{j+1}\big(\mathbf{o}^k_{j}, \mathfrak{X}^k_{j+1}(C^\epsilon_{k,j}(\mathbf{o}^k_{j}), \mathbf{o}^k_{j}, w^k)\big) \nu^k(dw^k) + \epsilon,
$$
for every $\mathbf{o}^k_{j}\in \{\mathbb{V}^{k}_j < + \infty\}$, where $j=m-1,\ldots, 1$.
\end{proposition}
\begin{proof}
We just repeat the argument of the proof of Lemma \ref{iterUM} jointly with Lemma \ref{measurabilityissue2}.
\end{proof}

We are now able to define the value function at step $j=0$ as follows

%\begin{equation}\label{detarg1}
$$\mathbb{V}^{k}_0:= \sup_{a^k_{0}\in \mathbb{A}} \int_{\mathbb{W}_k} \mathbb{V}^{k}_{1}\big(\mathbf{o}^k_{0}, \mathfrak{X}^k_{1}(a^k_{0}, \mathbf{o}^k_{0}, w^k)\big) \nu^k(dw^k),$$ for any initial value  $x_0 \in \mathbb{R}^d$. Therefore, if $\mathbb{V}^k_0 < +\infty$ then for $\epsilon>0$, there exists $C^\epsilon_{k,0}\in \mathbb{A}$ which realizes

$$
\mathbb{V}^{k}_0 \leq  \int_{\mathbb{W}_k} \mathbb{V}^{k}_{1}\big(\mathbf{o}^k_{0}, \mathfrak{X}^k_{1}(C^\epsilon_{k,0}, \mathbf{o}^k_{0}, w^k)\big) \nu^k(dw^k) + \epsilon.
$$

\begin{proposition}\label{AGREGATION}
For each $j=m-1, \ldots, 0$ and a control $u^k\in U^{k,m}_0$, we have

$$
\mathbb{V}^{k}_{j}(\Xi^{k,u^k}_{j}) = V^k(T^k_{j}, u^k) ~a.s.
$$
Let $C^{\epsilon}_{k,j}:\mathbb{H}^{j}_k\rightarrow \mathbb{A}; j=m-1,\ldots, 0$ be the functions given in Proposition \ref{detarg}. Observe that $C^\epsilon_{k,j}$ can be computed from (\ref{Vkfunction}). If $\mathbb{H}^{j}_k = \{\mathbb{V}^k_j < +\infty\}$, for $j=m-1, \ldots, 0$, then for every $\epsilon>0$, there exists a control $u^{k,\epsilon}_j$ defined by

\begin{equation}\label{explicitcontrol}
u^{k,\epsilon}_j := C^\epsilon_{k,j}(\Xi^{k,u^k}_j); j=m-1, \ldots, 0
\end{equation}
which realizes

\begin{equation}\label{vkineq}
V^k(T^k_j,u^k)\le \mathbb{E}\big[V^k(T^k_{j+1}, u^k\otimes_{j}u^{k,\epsilon}_j)|\mathcal{F}^k_{T^k_j}  \big] + \epsilon~a.s,
\end{equation}
for every $j=m-1, \ldots, 0$.
\end{proposition}
\begin{proof}
The statements for $j=m-1$ hold true due to (\ref{iterdois3}) in Lemma \ref{iterDOIS}. Now, by using Proposition \ref{detarg} and a backward induction argument, we conclude the proof.
\end{proof}

We are now able to construct an $\epsilon$-optimal control by means of the dynamic programming principle.

\begin{proposition}\label{epsiloncTH}
Let $(\xi, (X^k)_{k\ge 1})$ be an admissible pair w.r.t the control problem (\ref{discreteOCPROBLEM}). Then, $\phi^{*,k,\epsilon} = (\phi^{k,\eta_k(\epsilon)}_0, \phi^{k,\eta_k(\epsilon)}_1, \ldots, \phi^{k,\eta_k(\epsilon)}_{m-1})$ constructed via (\ref{explicitcontrol}) is $\epsilon$-optimal, where $\eta_k(\epsilon) = \frac{\epsilon}{e(k,T)}$. In other words, for every $\epsilon> 0$ and $k\ge 1$, $\phi^{*,k,\epsilon}\in U^{k,m}_0$ realizes

\begin{equation}\label{epsiloncTHzero}
\sup_{u^k\in U^{k,m}_0}\mathbb{E}\big[\xi_{X^k}(u^k)\big]\le \mathbb{E}\big[\xi_{X^k}(\phi^{*,k,\epsilon})\big]+ \epsilon.
\end{equation}

\end{proposition}

\begin{proof}
Fix $\epsilon>0$ and let $\eta_k(\epsilon) = \frac{\epsilon}{m}$, where we recall $m=e(k,T)$. The candidate for an $\epsilon$-optimal control is
$$\phi^{*,k,\epsilon} = (\phi^{k,\eta_k(\epsilon)}_0, \phi^{k,\eta_k(\epsilon)}_1, \ldots, \phi^{k,\eta_k(\epsilon)}_{m-1}),$$
where $\phi^{k,\eta_k(\epsilon)}_i; i=m-1, \ldots, 0$ are constructed via (\ref{explicitcontrol}). Let us check it is indeed $\epsilon$-optimal. From (\ref{vkineq}), we know that

\begin{equation}\label{epsiloncTH1}
\sup_{u^k\in U^{k,e(k,T)}_0}\mathbb{E}\big[ \xi_{X^k}(u^k)\big]\le \mathbb{E}\big[V^k(T^k_1, \phi^{k,\eta_k(\epsilon)}_0 )\big] + \eta_k(\epsilon)
\end{equation}
and

\begin{equation}\label{epsiloncTH2}
V^k(T^k_1, \phi^{k,\eta_k(\epsilon)}_0 )\le \mathbb{E}\big[V^k(T^k_2,\phi^{k,\eta_k(\epsilon)}_0\otimes_1 \phi^{k,\eta_k(\epsilon)}_1)|\mathcal{F}^k_{T^k_1}\big] + \eta_k(\epsilon)~a.s.
\end{equation}
Inequalities (\ref{epsiloncTH1}) and (\ref{epsiloncTH2}) yield

\begin{equation}\label{epsiloncTH3}
\sup_{u^k\in U^{k,e(k,T)}_0}\mathbb{E}\big[ \xi_{X^k}(u^k)\big]\le \mathbb{E}\big[V^k(T^k_2,\phi^{k,\eta_k(\epsilon)}_0\otimes_1 \phi^{k,\eta_k(\epsilon)}_1)\big] + 2\eta_k(\epsilon),
\end{equation}
where (\ref{vkineq}) implies $V^k(T^k_j,\phi^{k,\eta_k(\epsilon)}_0\otimes_1 \phi^{k,\eta_k(\epsilon)}_1\otimes_2 \ldots \otimes_{j-1} \phi^{k,\eta_k(\epsilon)}_{j-1})$ is less than or equals to

\begin{equation}\label{epsiloncTH4}
\mathbb{E}\big[ V^k(T^k_{j+1},\phi^{k,\eta_k(\epsilon)}_0\otimes_1\phi^{k,\eta_k(\epsilon)}_1\otimes_2 \ldots \otimes_{j} \phi^{k,\eta_k(\epsilon)}_{j})|\mathcal{F}^k_{T^k_j}\big]+\eta_k(\epsilon)~a.s,
\end{equation}
for $j=1, \ldots, m-1$. By iterating the argument starting from (\ref{epsiloncTH3}) and using (\ref{epsiloncTH4}), we conclude (\ref{epsiloncTHzero}).
\end{proof}

\section{Simulation of the increments $\Delta A^k(T^k_n)$} \label{explanation}
In this section, we briefly discuss one fundamental step in our numerical scheme: The simulation of the increments $\Delta A^k(T^k_n)$ for a $d$-dimensional Brownian motion $B = (B^1, \ldots, B^d)$ in the pseudo-algorithm described in Section \ref{numersection}. Recall that in the one-dimensional case, the distribution of $\Delta A^k(T^k_n)$ follows a $\frac{1}{2}$-Bernoulli distribution concentrated at $\{-\epsilon_k,\epsilon_k\}$. In the multi-dimensional case, this is not the case and one has to work with the conditional distribution of $\Delta A^k(T^k_n)$ given $\Delta T^k_n$ for $d>1$. By the strong Markov property, it is sufficient to discuss this conditional distribution only at the first step $n=1$. We also restrict the discussion to the case $d=2$ because the analysis for $d>2$ is totally similar. In the sequel, the Gaussian kernel is denoted by 
$$\phi_t(x):= \frac{1}{\sqrt{2\pi t}}e^{-\frac{-x^2}{2t}}$$ 
for $t>0$ and $x \in \mathbb{R}$.

%Let $Y$ be a Gaussian random variable with mean $\mu \in \mathbb{R}$ and variance $\sigma^2>0$ and let $-\infty\le a <  b \le +\infty$. The distribution of $Y$ conditioned on the event $ a < Y < b$ is called a truncated Gaussian distribution with parameters $(\mu,\sigma^2; a,b)$. Its probability density function is given by 

%$$\frac{1}{\sigma}\frac{\varphi(\frac{x-\mu}{\sigma})}{\Phi(\frac{b-\mu}{\sigma}) - \Phi(\frac{a-\mu}{\sigma})}\mathds{1}_{[a,b]}(x),$$
%where 
%$$\varphi(y):=\frac{1}{\sqrt{2\pi}}\exp \big(-\frac{1}{2}y^2\big)$$
%and $\Phi$ is the cumulative standard Gaussian distribution. In the sequel, for simplicity of exposition, we discuss the bidimensional case $B = (B^1, B^2)$. 

%Let $\mathcal{G}^k_{1}$ be the smallest sigma algebra which contains the class of sets 
%$$\{T^k_1 = T^{k,1}_1\}\cap \{T^k_1 \in F\}$$
% where $F$ is a Borel set in $\mathbb{R}$. 
\begin{remark}
For a given Borel set $E$, 
\begin{equation}\label{cond0}
\mathbb{P}\Big[ B^{2}(T^k_1) \in E\big| T^k_1= T^{k,1}_1=t\Big] = \mathbb{P}\Big[B^2(t) \in E\big| \sup_{0\le s\le t}|B^2(s)| < \epsilon_k\Big].
\end{equation}

%and 
%$$\mathbb{P}\Big[ B^{1}(T^k_1) \in E\big| T^k_1=T^{k,2}_1=t\Big] = \mathbb{P}\Big[B^1(t) \in E\big| \sup_{0\le s\le t}|B^1(s)| < \epsilon_k\Big],$$ 
for every $0 < t<T$.    
\end{remark}
\begin{proof}
Just observe that 

$$\{T^k_1= T^{k,1}_1 = t\} = \{T^{k,1}_1=t\}\cap \Big\{\sup_{0\le u\le T^k_1}|B^2(u)| < \epsilon_k\Big\} = \{T^{k,1}_1=t\}\cap \Big\{\sup_{0\le u\le t}|B^2(u)| < \epsilon_k\Big\}.$$
Then, by the independence of $B^1$ and $B^2$, for each Borel set $E$, we have 
\begin{eqnarray*}
\mathbb{P}\Big[ B^{2}(T^k_1) \in E\big| T^k_1= T^{k,1}_1=t\Big]&=& \mathbb{P}\Big[ B^{2}(t) \in E\big| T^k_1=T^{k,1}_1=t\Big]\\
&=& \mathbb{P}\Big[ B^{2}(t) \in E\big| \sup_{0\le u\le t} |B^2(u)|< \epsilon_k\Big].
\end{eqnarray*}
\end{proof}
Of course, similarly, we have
$$\mathbb{P}\Big[ B^{1}(T^k_1) \in E\big| T^k_1=T^{k,2}_1=t\Big] = \mathbb{P}\Big[B^1(t) \in E\big| \sup_{0\le s\le t}|B^1(s)| < \epsilon_k\Big],$$ 
for every Borel set $E$ and $t>0$.      

%Let $m_t = \inf_{0\le s\le t} B^2(s)$ and $M_t = \sup_{0\le s\le t}B^2(s)$. 

By the very definition, 

\begin{equation}\label{conds1}
\mathbb{P}\Big[ B^{2}(t) \in E\big| \sup_{0\le u\le t} |B^2(u)|< \epsilon_k\Big]
= \frac{\mathbb{P}[B^2(t) \in E, \sup_{0\le s\le t}|B^2(s)|  < \epsilon_k]}{\mathbb{P}[\sup_{0\le s\le t}|B^2(s)|  < \epsilon_k]}.
\end{equation}
Of course, the law (\ref{conds1}) is concentrated on the small open set $(-\epsilon_k,\epsilon_k)$. The simulation of (\ref{conds1}) is based on the trivariate distribution of the final, minimal and maximal value (see e.g. \cite{borodin} p. 174)

\begin{equation}\label{cond2}
\mathbb{P}\Big[B^2(t) \in dx, \sup_{0\le s\le t}|B^2(s)|  < \epsilon_k\Big]=\sum_{n=-\infty}^{+\infty} \Big[ \phi_t (x-4n\epsilon_k) -  \phi_t(x-2\epsilon_k -4n\epsilon_k)\Big]dx, 
\end{equation}
for $-\epsilon_k < x < \epsilon_k$. Observe we can write the density in (\ref{cond2}) as follows (see e.g Th 2.2 in \cite{riedel})

$$
\sum_{n=-\infty}^{+\infty} \Big[ \phi_t (x-4n\epsilon_k) -  \phi_t(x-2\epsilon_k -4n\epsilon_k) \Big]
$$
\begin{equation}\label{cond4}
= \phi_t(x) + L(t,x,\epsilon_k), 
\end{equation}
where 
\begin{eqnarray*} 
L(t,x,\epsilon_k)&:=& \sum_{m=1}^\infty [\phi_t(x+4m\epsilon_k) - \phi_t(x+2\epsilon_k+4(m-1)\epsilon_k) ]\\
& +& \sum_{m=1}^\infty [\phi_t(x-4m\epsilon_k) - \phi_t(x-2\epsilon_k-4(m-1)\epsilon_k) ],
\end{eqnarray*}
for $-\epsilon_k < x < \epsilon_k$.

The simulation of the conditional distribution (\ref{cond0}) in the pseudo-algorithm described in Section \ref{numersection} should be based on the density given by (\ref{cond2}) and (\ref{cond4}). This can be achieved by using classical methods. We refer the reader to \cite{riedel} and \cite{becker} for detailed analyses on the numerical schemes for the final, minimal and maximal values of the Brownian motion via the density (\ref{cond4}). 
%Since the terms in the series decrease with increasing $m$, the evaluation of the sum is stopped when the summands fall below a given bound. 

%at least in the simple case of the geometric Brownian motion (\ref{Seq}). 

In the context of the present article, for a given $t>0$, we should simulate (\ref{cond0}) via (\ref{cond4}) in the regime $\epsilon_k \downarrow 0$. The main difficulty lies on the series given by (\ref{cond4}). The evaluation of (\ref{cond4}) has to be done by approximating the series. The numerical experiments given in Section \ref{numersection} indicate that working only with the Gaussian kernel $\phi_t(x)$ in (\ref{cond4}) is enough to have an excellent approximation. Indeed, the leading term of the density described by (\ref{cond2}) and (\ref{cond4}) is actually $\phi_t(x)$ for $-\epsilon_k <  x < \epsilon_k$ in the regime $\epsilon_k\downarrow 0$. 

Fix $t>0$. Observe that   

$$0<\sum_{m=1}^n [\phi_t(x+2\epsilon_k+4(m-1)\epsilon_k) - \phi_t(x+4m\epsilon_k) ]$$
\begin{equation}\label{cond5}
\le \frac{1}{\sqrt{2\pi t}}\Big\{e^{-\frac{(x+2\epsilon_k)^2}{2t}} - e^{-\frac{(x+4\epsilon_k n)^2}{2t}} \Big\}\le \frac{1}{\sqrt{2\pi t}} \Big\{e^{-\frac{\epsilon_k^2}{2t}} - e^{-\frac{\epsilon^2_k(4n+1)^2}{2t}}\Big\} ,
\end{equation}
for every $ x \in (-\epsilon_k, \epsilon_k)$, $k,n\ge 1$. Similarly, 

$$0< \sum_{m=1}^n [\phi_t(x-2\epsilon_k-4(m-1)\epsilon_k) - \phi_t(x-4m\epsilon_k) ]$$
\begin{equation}\label{cond6}
\le \frac{1}{\sqrt{2\pi t}}\Big\{e^{-\frac{(x-2\epsilon_k)^2}{2t}} - e^{-\frac{(x-4\epsilon_k n)^2}{2t}} \Big\}\le \frac{1}{\sqrt{2\pi t}} \Big\{e^{-\frac{\epsilon_k^2}{2t}} - e^{-\frac{\epsilon^2_k(4n+1)^2}{2t}}\Big\},
\end{equation}
for every $ x \in (-\epsilon_k, \epsilon_k)$, $k,n\ge 1$. 
\begin{lemma}\label{uniseries}
Let $\epsilon_k = 2^{-k}$ for $k\ge 1$. The two series in $L(t,x,\epsilon_k)$ converge uniformly w.r.t. $x \in (-\epsilon_k,\epsilon_k)$ and $k\ge 1$. 
\end{lemma}
\begin{proof}
Without loss of generality, we assume $t=1$. A second-order Taylor expansion on the function $a\mapsto e^{\frac{-a^2}{2}}$ yields 

\begin{eqnarray*}
\nonumber\Big|e^{\frac{-(x+4n 2^{-k}-2^{-k+1})^2}{2}} - e^{\frac{-(x+4n 2^{-k})^2}{2}} \Big| &=& 2^{-k+1} (x+4n2^{-k} -2^{-k+1})e^{\frac{-(x+4n2^{-k} -2^{-k+1})^2}{2}}\\
&+& \frac{1}{2} \big(1-c^2_{x,k,n}\big) e^{\frac{- c^2_{x,k,n}}{2}} 2^{-2k+2}\\
&<& 2^{-2k+1}(4n-1)e^{-\frac{2^{-2k}(4n-3)^2}{2}}\\
&<& 4 \Big(ne^{-n} + \Big(\frac{1}{n}\Big)^{\alpha}\Big),
\end{eqnarray*}
for every $x \in (-2^{-k}, 2^{-k})$, $k,n\ge 1$. Here, we can choose e.g. $\alpha=1.01$ and $c_{x,k,n}$ is a number depending on $x,k,n$ such that $0 < x+4n 2^{-k} -2^{-k+1} < c_{x,k,n} <  x+4n 2^{-k}$. Then, 

$$\Big|e^{\frac{-(x+4n 2^{-k}-2^{-k+1})^2}{2}} - e^{\frac{-(x+4n 2^{-k})^2}{2}} \Big|\le 4 \Big(ne^{-n} + \Big(\frac{1}{n}\Big)^{1.01}\Big),$$
for every $x \in (-2^{-k}, 2^{-k})$, $k,n\ge 1$. Similar analysis can be made for 

$$\Big|e^{\frac{-(x-4n 2^{-k} +2^{-k+1})^2}{2}} - e^{\frac{-(x-4n2^{-k})^2}{2}} \Big|.$$ 
By the Weierstrass M-test, we conclude the proof.

%By applying mean value theorem to the the function $a\mapsto e^{-\frac{1}{2}a^2}$, for each $-\epsilon_k < x < \epsilon_k$, $k,n\ge 1$, there exists $c_{x,k,n} \in (x+4\epsilon_k n-2\epsilon_k, x+4\epsilon_k n) $ such that  

%\begin{eqnarray}
%\nonumber\Big|e^{\frac{-(x+4\epsilon_k n-2\epsilon_k)^2}{2}} - e^{\frac{-(x+4\epsilon_k n)^2}{2}} \Big| &=& 2\epsilon_k  c_{x,k,n} e^{\frac{- %c^2_{x,k,n}}{2}}\\
%\label{cond7}&\le& (\epsilon_k + 4 \epsilon_k n) 2\epsilon_k e^{\frac{-(4\epsilon_k n-3\epsilon_k)^2}{2}}.
%\end{eqnarray}
%When $1\le n \le k$, we have $\epsilon_k\le \epsilon_n$ and hence (\ref{cond7}) and using the fact that $\sup_{n\ge 1}n\epsilon_n < \infty$, we have 

%\begin{eqnarray*}
%\Big|e^{\frac{-(x+4\epsilon_k n-2\epsilon_k)^2}{2}} - e^{\frac{-(x+4\epsilon_k n)^2}{2}} \Big| &\le& (\epsilon_n + 4n\epsilon_n)2\epsilon_n\\
%&\lesssim& \epsilon_n 
%\end{eqnarray*}
%for every $x \in (-\epsilon_k,\epsilon_k)$ and  $1\le n \le k$. Moreover, there exists $\alpha>1$ (take e.g $\alpha=1.01$) such that  

%$$
%\Big|e^{\frac{-(x+4\epsilon_k n-2\epsilon_k)^2}{2}} - e^{\frac{-(x+4\epsilon_k n)^2}{2}} \Big| \lesssim n e^{-n} + \frac{1}{n^\alpha}
%$$
%for every $x \in (-\epsilon_k,\epsilon_k)$ and  $1\le k < n$. Similar analysis can be made for 

%$$\Big|e^{\frac{-(x-4\epsilon_k n+2\epsilon_k)^2}{2}} - e^{\frac{-(x-4\epsilon_k n)^2}{2}} \Big|.$$ 
%By the Weierstrass M-test, we conclude the proof.  
\end{proof}
Lemma \ref{uniseries} jointly with (\ref{cond5}) and (\ref{cond6}) yield 

$$
\sup_{-2^{-k} < x < 2^{-k}}|L(t,x,2^{-k})|\rightarrow 0,
$$
as $k\rightarrow +\infty$. 

Having said that, the simulation of (\ref{cond0}) described in Section \ref{numersection} is based on a reduced version of the full density (\ref{cond4}), where we set $L(t,x,2^{-k})=0$ for $t>0$, $x \in (-2^{-k},2^{-k})$ and $k\ge 1$. More precisely, we actually use the truncated Gaussian distribution with parameters $(0,t, -2^{-k},2^{-k})$ whose density is described by

$$\frac{\phi_t(x)}{\int_{(-2^{-k},2^{-k})}\phi_t(y)dy}\mathds{1}_{(-2^{-k},2^{-k})}(x),~x\in \mathbb{R}.$$

\end{appendix}

\begin{funding}
The second author was supported by Math-AmSud (grant 88887.197425/2018-00) and FAPDF (grant 00193-00001506/2021-12).  
\end{funding}

\begin{acks}[Acknowledgments]
The authors would like to thank two anonymous referees for their very important and constructive comments that improved the
quality of this paper.
\end{acks}


\begin{thebibliography}{4}

%\bibitem{bandini} Bandini, E., Cosso, A., Fuhrman, M. and Pham, H. (2018).  Randomization method and backward SDEs for optimal control of partially observed path-dependent stochastic systems. \textit{Ann. Appl. Probab}, \textbf{28}, 3, 1634-1678.

\bibitem{bayer} Bayer, C., Friz, P. K., Gassiat, P., Martin, J. and Stemper,B. (2019). A regularity structure for rough volatility. \textit{Math. Finance}, 1-51.

\bibitem{barles} Barles, G. and Souganidis, P.E. (1991). Convergence of approximation schemes for fully nonlinear second order
equations. \textit{Asymptotic Anal}. \textbf{4}, 271-283.


\bibitem{becker} Becker, M. (2010). Exact simulation of final, minimal and maximal values of Brownian motion and jump-diffusions with applications to option pricing. \textit{Comput Manag Sci}. 7, 1-17. 

%\bibitem{bah} Bahlali, K., Khelfallahb, N. and Mezerdi, B. (2009). Necessary and sufficient conditions for near-optimality in stochastic control
%of FBSDEs.\textit{Systems and Control Letters}, \textbf{58}, 857-864.

\bibitem{bertsekas} Bertsekas, D.P. and Shreve, S. \textit{Stochastic optimal control: The discrete-time case}. Athena Scientific Belmomt Massachusett, 1996.

%\bibitem{bezerra} Bezerra, S.C., Ohashi, A. and Russo, F. and Souza, F. A. (2020). Discrete-type approximations for non-Markovian optimal stopping problems: Part II. \textit{Methodol. Comput. Appl. Probab}, \textbf{22}, 1221-1255.

%\bibitem{biagini} Biagini F., Hu, Y., Oksendal B., and Sulem, A. (2002). A stochastic maximum principle for processes driven by fractional Brownian
%motion. \textit{Stochastic Process Appl}, \textbf{100}, 1-2, 233-253.


%\bibitem{bouchard} Bouchard, B. and Touzi, N. (2004). Discrete-time approximation and Monte-Carlo simulation of backward stochastic differential equations. \textit{Stochastic Process. Appl}. \textbf{111}, 175-206.


\bibitem{Burq_Jones2008} Burq, Z. A. and Jones, O. D. (2008). Simulation of brownian motion at first-passage times. \textit{Math. Comput. Simul.} \textbf{77}, 1, 64-71.

%\bibitem{bremaud} Br\'emaud, P. (1981). \textit{Point Processes and Queues. Martingale dynamics}. Springer,
%New York.

\bibitem{borodin} Borodin, A. N. and Salminen, P., Handbook of Brownian Motion: Facts and Formulae. Birkhauser, 2002.

%\bibitem{}bonnans} Bonnans, J.F., Justina Gianatti, J. and Silva, F.J. (2018). On the time discretization of stochastic optimal control problems: The dynamic programming approach. \text{ESAIM: COCV}, \textbf{25}, article numer 63.  		

%\bibitem{buckdahn_shuai} Buckdahn, R, and Shuai, J. (2014). Peng's maximum principle for a stochastic control
%problem driven by a fractional and a standard Brownian motion. \textit{Science China Mathematics}, \textbf{57}, 10, 2025-2042.



%\bibitem{chiani} Chiani, M., Dardari, D., Simon, M.K. (2003). New Exponential Bounds and Approximations for the Computation of Error Probability in Fading Channels. IEEE Transactions on Wireless Communications, 4(2), 840-845.

\bibitem{cheridito} Cheridito, P., Kawaguchi, H. and Maejima, M. (2003). Fractional Ornstein-Uhlenbeck processes. \textit{Electron. J. Probab}, \textbf{8}, 3, 14 p.


\bibitem{claisse} Claisse, J., Talay, D. and Tan, X. A Pseudo-Markov Property for Controlled Diffusion Processes. \textit{SIAM J. Control Optim}, \textbf{54}, 2, 1017-1029.

\bibitem{lipref} Cobzas, S., Miculescu, R. and Nicolae, A. Lipshitz functions. Lecture Notes in Mathematics 2241.

%\bibitem{elliot} Cohen, S. and Elliot, R. Stochastic calculus and its applications. Second edition. Birkhauser.

\bibitem{cont1} Cont, R. and Fournie, D. (2013). Functional Ito calculus and stochastic integral representation of martingales, \textit{Ann. of Probab}, \textbf{41}, 1, 109-133.

\bibitem{cont2} Cont, R. Functional It\^o calculus and functional Kolmogorov equations, in: V Bally et al: Stochastic integration by parts and Functional Ito calculus (Lectures Notes of the Barcelona Summer School on Stochastic Analysis, Centro de Recerca de Matematica, July 2012), Springer: 2016.


%\bibitem{coquet1} Coquet, F; M\'{e}min, J; Slominski, L. (2001). On weak
%convergence of filtrations. \textit{Lecture Notes in Math.},~\textbf{1755}, 306-328.


\bibitem{Davis_79} Davis, M. Martingale methods in stochastic control, in Stochastic Control and Stochastic Differential Systems, Lecture Notes in Control and Information Sciences 16 Springer-Verlag, Berlin 1979.


\bibitem{frank} den Hollander, Frank. Large deviations. Fields Institute monographs.

%\bibitem{dellacherie} Dellacherie, C. and Meyer, P. A. \textit{Probability and Potential B}. Amsterdam: North-Holland, 1982.


%\bibitem{dellacherie2} Dellacherie, C. and Meyer, P. \textit{Probabilit\'es et potentiel}. Hermann, Paris, 1987.

%\bibitem{diehl} Diehl, J., Friz, P. K. and Gassiat, P. (2017). Stochastic control with rough paths. \textit{Appl.
%Math. Optim}, \textbf{75}, 285–315.


\bibitem{dolinsky}  Dolinsky, Y. (2012). Numerical schemes for G-expectations. \textit{Electron. J. Probab}, \textbf{17},
1–15.

\bibitem{duncan} Duncan, T.E. and Pasik-Duncan, B. (2013). Linear-quadratic fractional Gaussian control. \textit{SIAM J. Control Optim}, \textbf{51}, 6, 4504-4519.

\bibitem{dupire} Dupire, B. \textit{Functional It\^o calculus}. Portfolio Research Paper 2009-04. Bloomberg.


\bibitem{touzi2} Ekren, I., Touzi, N. and Zhang, J. (2016). Viscosity Solutions of Fully Nonlinear Parabolic Path Dependent PDEs: Part I. ~\textit{Ann. Probab.}~, \textbf{44}, 2, 1212-1253.

\bibitem{elkaroui}  El Karoui, N. (1979). {\it Les Aspects Probabilistes du
Contr\^ole Stochastique }, in Ecole d'Et\'e de Probabilit\'es de
Saint-Flour IX, Lecture Notes in Math. {\bf 876}.


\bibitem{Engelking} Engelking, R.  General Topology. Sigma series in Pure Mathematics.

\bibitem{fahim} Fahim, A., Touzi, N. and Warin, X. (2011). A probabilistic numerical method for fully nonlinear parabolic PDEs. \textit{Ann. Appl. Probab}, \textbf{21}, 4, 1322-1364.

%\bibitem{fischer} Fischer, M. and Nappo, G. (2009). On the Moments of the Modulus of Continuity of It\^o Processes. \textit{Stoch. Analysis Appl}, 28, 1, 103-122.
\bibitem{fuhrman} Fuhrman, M. and Pham, H. (2015). Randomized and backward SDE representation for optimal control of non-Markovian SDEs. \textit{Ann. Appl. Probab}, \textbf{25}, 4, 2134-2167.


\bibitem{gordon} Gordon, Y., Litvak, A.E., Sch\"{u}tt, C. and Werner, E. (2006). On the Minimum of Several Random Variables. \textit{P. Am. Math. Soc}, \textbf{134}, 12, 3665-3675.

\bibitem{Grigelionis} Grigelionis, B. and  Mackevicius. V. (2003). The finiteness of moments of a stochastic exponential. \textit{Stat. Probab. Letters}, 64, 243-248.

%\bibitem{han} Han, Y., Hu, Y. and Song, J. (2013). Maximum principle for general controlled systems driven by fractional Brownian motions. \textit{Appl
%Math Optim}, \textbf{7}, 279-322.

\bibitem{he} He, S-w., Wang, J-g., and Yan, J-a. \textit{Semimartingale Theory and Stochastic Calculus}, CRC Press, 1992.


%\bibitem{Hu} Hu, Y. (2005). Integral transformations and anticipative calculus for fractional Brownian motions. \textit{Memoirs of the American Mathematical Society}, \textbf{175}, 825.

\bibitem{Hu1} Hu, Y. and Zhou, X. (2005). Stochastic control for linear systems driven by fractional noises. \textit{SIAM J Control Optim}, \textbf{43},
2245-2277.

\bibitem{jaber1} Jaber, E. A., Miller, E. and Pham, H. (2021). Integral operator Riccati equations arising in stochastic Volterra control problems. \textit{SIAM J Control Optim}, \textbf{59}, 2, 1581-1603.


\bibitem{jaber2} Jaber, E. A., Miller, E. and Pham, H. (2021). Linear-Quadratic control for a class of stochastic Volterra equations: solvability and approximation. \textit{Ann. Appl. Probab}, \textbf{31}, 5, 2244-2274.

%\bibitem{jacod} Jacod, J., and Skohorod. A.V. (1994). Jumping
%filtrations and martingales with finite variation. \textit{Lecture Notes in Math.} \textbf{1583}, 21-35. Springer.

%\bibitem{knight} Knight, F. (1963). Random walk and a sojourn density process of Brownnian motion. \textit{Trans. Amer. Math. Soc}, \textbf{109}, 56-86.

%\bibitem{karatzas} Karatzas, I. and Shreve, S. Brownian Motion and Stochastic Calculus. Sprineger-Verlag. Second edition.

\bibitem{pham} Kharroubi, I. and Pham, H. (2014). Feynman-Kac representation for Hamilton-Jacobi-
Bellman IPDE. \textit{Ann. Probab}, \textbf{43}, 4, 1823-1865.

\bibitem{pham1} Kharroubi, I. Langren\`e, N. and Pham, H. (2014). A numerical algorithm for fully nonlinear HJB equations: An approach by control randomization. \textit{Monte Carlo Methods and Applications}, \textbf{20}, 2.


\bibitem{pham2} Kharroubi, I. Langren\`e, N. and Pham, H. (2015). Discrete time approximation of fully nonlinear HJB equations via BSDEs with nonpositive jumps. \textit{Ann. Appl. Probab.}, \textbf{25}, 4, 2301-2338.


%\bibitem{koshnevisan} Khoshnevisan, D. and Lewis, T.M. (1999). Stochastic calculus for Brownian motion on a Brownian fracture. \textit{Ann. Appl. Probab}. \textbf{9}, 3, 629-667.


\bibitem{krylov1} Krylov, N. V. (1999). Approximating value functions for controlled degenerate diffusion processes by using piecewise
constant policies. \textit{Electron. J. Probab}. \textbf{4}, 1-19.

\bibitem{krylov2} Krylov, N.V. (2000). On the rate of convergence of finite-difference approximations for Bellmans equations with
variable coefficients. \textit{Probab. Theory Relat. Fields}, \textbf{117}, 1, 1-116.

%\bibitem{kushner1} Kushner, H.J. \textit{Approximation and Weak Convergence Methods for Random Processes,
%with Applications to Stochastic Systems Theory}, MIT Press Series in Signal Processing,
%Optimization, and Control, Vol. 6 (MIT Press, 1984).

\bibitem{kushner2} Kushner, H.J and Dupuis, P. \textit{Numerical Methods for Stochastic Control Problems
in Continuous Time}, 2nd edn., Applications of Mathematics, Vol. 24, Springer-Verlag, 2001.


\bibitem{lamberton} Lamberton, D. Optimal stopping and American options, Daiwa Lecture Ser., Kyoto, 2008.


%\bibitem{mazliak} Mazliak, L. and Nourdin, I. (2008). Optimal control for rough differential equations. \textit{Stoch. Dyn}. 08, 23.


\bibitem{LEAO_OHASHI2013} Le\~{a}o, D. and Ohashi, A. (2013). Weak approximations for Wiener functionals. \textit{Ann. Appl. Probab,} 23, \textbf{4}, 1660-1691.


%\bibitem{LEAO_OHASHI2017} Le\~{a}o, D. and Ohashi, A. Corringendum ``Weak approximations for Wiener functionals". To appear in Annals of Applied Probability.


\bibitem{LEAO_OHASHI2017.1} Le\~ao,D. Ohashi, A. and Simas, A. B. (2018). A weak version of path-dependent functional It\^o calculus. \textit{Ann. Probab}, \textbf{46}, 6, 3399-3441.


\bibitem{LEAO_OHASHI2017.2} Le\~ao, D., Ohashi, A. and Russo, F. (2019). Discrete-type approximations for non-Markovian optimal stopping problems: Part I. \textit{J. Appl. Probab}. \textbf{56}, 4, 981-1005.


%\bibitem{ohashifrancys2018.1} Le\~ao, D., Ohashi, A. and Souza, F. Mean-variance hedging for rough stochastic volatility. In preparation.


%\bibitem{Lim} Lim, A.E.B. and Zhou, X.Y. Verification theorems for stochastic near-optimal control in the framework of viscosity solutions. 1999 IFAC 11 th Triennial World Congress. Beijing, P.R. China.

%\bibitem{memin} M\'emin, J. (2003). Stability of Doob-Meyer Decomposition Under Extended Convergence. \textit{Acta Mathematicae Applicatae Sinica},
% $   \textbf{19}, 2, 177-190

%\bibitem{meng} Meng, Q. and Shen, Y. (2015). A revisit to stochastic near-optimal controls: The critical case. \textit{Systems and Control Letters}, \textbf{82}, 79-85.

%\bibitem{metivier} M\'etivier, M. \textit{Semimartingales: A Course on Stochastic Processes}. De Gruyter studies in mathematics. 1982.



%\bibitem{milstein} Milstein, G.N and Tretyakov, M.V. \textit{Stochastic Numerics for Mathematical Physics}. Springer-Verlag.


%\bibitem{mishura} Mishura, Y.  \textit{Stochastic calculus for fractional Brownian motion}. Springer-Verlag.


%\bibitem{hafa} Mokhtar Hafayed, M., Abbas, S. and Veverka, P. (2013). On necessary and sufficient conditions for near-optimal
%singular stochastic controls, \textit{Optim Lett}, \textbf{7}, 949-966.



\bibitem{nualart}  Nualart, D. The Malliaviin calculus and related topics. Probability and its
Applications. Second edition. Springer.

%\bibitem{nualart}  Nualart, D and Rascanu, A. (2002). Differential equations driven by fractional Brownian motion, \textit{Collect. Math}. \textbf{53}, 1, 55-81.

\bibitem{nutz1} Nutz, M. and van Handel,  R. (2013). Constructing Sublinear Expectations on Path Space. \textit{Stochastic Process. Appl}, \textbf{123}, 8, 3100-3121.

\bibitem{nutz2} Nutz, M. (2012). A Quasi-Sure Approach to the Control of Non-Markovian Stochastic Differential Equations
\textit{Electron. J. Probab}, \textbf{17}, 23, 1-23.


\bibitem{ohashifrancys} Ohashi, A and Souza, F.A. (2020). $L^p$ uniform random walk-type approximation for fracional Brownian motion with Hurst exponent $0 < H < \frac{1}{2}$. \textit{Electron. Commun. Probab}, 25, 1-13.


\bibitem{vares} Olivieri, E. and Vares, M.E. Large deviations and metaestability. Encyclopedia of Mathematics and its applications.


\bibitem{peng} Peng, S. G-expectation, G-Brownian motion and related stochastic calculus of It\^o type. In \textit{Stochastic analysis and applications}. Springer, Berlin, Heidelberg, 2007, 541-567.



\bibitem{possamai} Possama\"i, D., Tan, X. and Zhou, C. (2018). Stochastic control for a class of nonlinear kernels and applications. \textit{Ann. Probab}, \textbf{46}, 1, 551-603

%\bibitem{protter} Protter, P. Stochastic Integration and Differential Equations : A New Approach. 1990. Third edition.


%\bibitem{zhang1} Pham, T, and Zhang, J. (2013). Some Norm Estimates for Semimartingales. \textit{Electronic Journal of Probability}, \textbf{18}, 109, 1-25.

\bibitem{qiu} Qiu, J. (2018). Viscosity Solutions of Stochastic Hamilton--Jacobi--Bellman Equations. \textit{SIAM J. Control Optim}, \textbf{56}, 5, 3708-3730.


\bibitem{ren} Ren, Z and Tan, X. (2017). On the convergence of monotone schemes for path-dependent PDEs. \textit{Stochastic Process. Appl}, \textbf{127}, 6, 1738-1762.

\bibitem{riedel} Riedel, K. (2021). The value of the high, low and close in the estimation of
Brownian motion. \textit{Stat. Inference Stoch. Process}. \textbf{24}, 179–210. 


\bibitem{saporito} Saporito, Y. (2019) Stochastic Control and Differential Games with Path-Dependent Influence of Controls on Dynamics and Running Cost. \textit{SIAM J. Control Optim}, \textbf{57}, 2, 1312-1327.


\bibitem{soner} Soner, M., Touzi, N. and Zhang, J. (2012). The wellposedness of second order backward SDEs. \textit{Probab. Theory Relat. Fields}, 153, 149-190.

%\bibitem{striebel} Striebel, C. (1984). Martingale conditions for the optimal control of continuous time stochastic systems. \textit{Stochastic Process. Appl}, \textbf{18}, 329-347.

%\bibitem{stroock} Stroock, D. Probability theory: An analytical view. Cambridge University Press.

\bibitem{tan} Tan, X. (2014). Discrete-time probabilistic approximation of path-dependent stochastic control problems. \textit{Ann. Appl. Probab,} \textbf{24}, 5, 1803-1834.

\bibitem{viens} Viens, F. and Zhang, J. (2019). A martingale approach for fractional Brownian motions and related path
dependent PDEs, \textit{Ann. Appl. Probab}, \textbf{29}, 3489–3540.

%\bibitem{tan1} Tan, X. (2013). A splitting method for fully nonlinear
%degenerate parabolic PDEs. \textit{Electron. J. Probab}. \textbf{18}, 15, 1–24.


\bibitem{wang} Wang, H., Yong, J., and Zhang, J. (2022). Path dependent Feynman-Kac formula for forward backward stochastic Volterra integral equations. \textit{Ann. inst. Henri Poincare (B) Probab. Stat}, \textbf{58}, 2, 603-638.


\bibitem{zhang2} Zhang, J and Zhuo, J. (2014). Monotone schemes for fully nonlinear parabolic path dependent
PDEs, \textit{Journal of Financial Engineering}, \textbf{1}.

%\bibitem{zhang3} Zhang, J. (2004). A numerical scheme for BSDEs. \textit{Ann. Appl. Probab}, \textbf{14}, 459-488.

\bibitem{zhou} Zhou, X.Y. (1998). Stochastic near-optimal controls: Necessary and sufficient conditions for near optimality.
\textit{SIAM. J. Control. Optim}, \textbf{36}, 3,  929-947.
\end{thebibliography}
\end{document}